\documentclass{article}
\usepackage[utf8]{inputenc}
\usepackage{fullpage}
\usepackage{amssymb}
\usepackage{amsmath}
\usepackage{amsthm}
\usepackage{mathtools}
\usepackage{color}
\usepackage{enumitem}

\usepackage{caption}
\usepackage{tikz}
\usepackage{xcolor}
\usepackage{upgreek}
\usepackage[textwidth=20mm,textsize=tiny,shadow]{todonotes}
\usepackage{hyperref}
\newcommand{\email}[1]{\href{mailto:#1}{#1}}
\usepackage{bbm}
\usepackage{enumitem}

\newcommand{\N}{\mathbb{N}}
\newcommand{\R}{\mathbb{R}}
\newcommand{\C}{\mathbb{C}}
\newcommand{\K}{\mathbb{K}}
\renewcommand{\P}{\mathbb{P}}
\newcommand{\PP}{\mathbb{P}}
\newcommand{\ce}{\coloneqq}
\newcommand{\ec}{\eqqcolon}
\newcommand{\1}{\mathbf{1}}
\newcommand{\calA}{\mathcal{A}}
\newcommand{\calB}{\mathcal{B}}
\newcommand{\calC}{\mathcal{C}}
\newcommand{\calF}{\mathcal{F}}

\newcommand{\calL}{\mathcal{L}}

\newcommand{\calN}{\mathcal{N}}
\newcommand{\calO}{\mathcal{O}}
\newcommand{\calP}{\mathcal{P}}
\newcommand{\frakP}{\mathfrak{P}}

\newcommand{\calT}{\mathcal{T}}

\newcommand{\calX}{\mathcal{X}}
\newcommand{\bfA}{\mathbf{A}}
\newcommand{\bfe}{\mathbf{e}}
\newcommand{\bff}{\mathbf{f}}
\newcommand{\bfF}{\mathbf{F}}
\newcommand{\bfH}{\mathbf{H}}
\newcommand{\bfJ}{\mathbf{J}}
\newcommand{\bfP}{\mathbf{P}}
\newcommand{\bfR}{\mathbf{R}}
\newcommand{\bfs}{\mathbf{s}}
\newcommand{\bfT}{\mathbf{T}}
\newcommand{\bfV}{\mathbf{V}}
\newcommand{\bfu}{\mathbf{u}}
\newcommand{\bfv}{\mathbf{v}}
\newcommand{\bfw}{\mathbf{w}}

\newcommand{\frakA}{\mathfrak{A}}
\newcommand{\frakF}{\mathfrak{F}}
\newcommand{\frakH}{\mathfrak{H}}
\newcommand{\frakJ}{\mathfrak{J}}
\renewcommand{\frakP}{\mathfrak{P}}
\newcommand{\frakT}{\mathfrak{T}}
\newcommand{\frakV}{\mathfrak{V}}
\newcommand{\fraku}{\mathfrak{u}}
\newcommand{\frakv}{\mathfrak{v}}
\newcommand{\frakw}{\mathfrak{w}}
\newcommand{\frakform}{\mathfrak{\form}}

\newcommand{\rmd}{\mathrm{d}}

\newcommand{\hra}{\hookrightarrow}
\newcommand{\CVH}{C_{V \hra H}}

\newcommand{\grad}{\nabla}

\newcommand{\sumj}{\sum_{j=0}^{N-1}}

\DeclareMathOperator{\TextRe}{Re}
\DeclareMathOperator{\TextIm}{Im}
\renewcommand{\Re}{\TextRe}
\renewcommand{\Im}{\TextIm}

\newcommand{\seq}{\subseteq}

\newcommand{\ot}{\otimes}

\newcommand{\LZ}{{L_2(\R^N,\PP_Z)}}
\newcommand{\PZ}{\mathbb{P}_Z}
\newcommand{\dPZ}{~\mathrm{d}\mathbb{P}_Z}

\newcommand{\px}{{p_{\mathrm{x}}}}
\newcommand{\pt}{{p_{\mathrm{t}}}}
\newcommand{\pz}{{p_{\mathrm{z}}}}
\newcommand{\Nt}{N}
\newcommand{\Cx}{{C_{\mathrm{x}}}}

\newcommand{\Cz}{{C_{\mathrm{z}}}}

\newcommand{\Yx}{{Y_{\mathrm{x}}}}
\newcommand{\Yt}{{Y_{\mathrm{t}}}}

\newcommand{\Yz}{{Y_{\mathrm{z}}}}

\newcommand{\e}{\mathrm{e}}
\renewcommand{\i}{\mathrm{i}}
\newcommand{\Kzz}{{\K^{2\times2}}}
\newcommand{\Rzz}{{\R^{2\times2}}}
\newcommand{\ols}{{\overline{s}}}

\DeclareMathOperator{\dist}{dist}

\DeclareMathOperator{\tr}{tr}
\DeclareMathOperator{\dom}{D}
\DeclareMathOperator{\diverg}{div}
\renewcommand{\div}{\diverg}
\DeclareMathOperator{\gradient}{grad}
\renewcommand{\grad}{\gradient}
\DeclareMathOperator{\spt}{supp}
\DeclareMathOperator*{\esssup}{ess\,sup}
\DeclareMathOperator{\Div}{Div}
\DeclareMathOperator{\Hess}{Hess}

\makeatletter
\def\@row#1,{#1\@ifnextchar;{\@gobble}{&\@row}}
\def\@matrix{%
    \expandafter\@row\my@arg,;%
    \@ifnextchar({\\ \get@in@paren{\@matrix}}{\after@matrix}%
    }
\def\matrixtype#1#2#3{%
    \ifmmode\def\after@matrix{\end{#2}\right#3}%
    \else\def\after@matrix{\end{#2}\right#3$}$\fi
    \left#1\begin{#2}\get@in@paren{\@matrix}%
    }
\def\@column#1,{#1\@ifnextchar;{\@gobble}{\\ \@column}}
\newcommand\vect{}
\def\svect(#1){\left(\begin{smallmatrix}\@column#1,;\end{smallmatrix}\right)}
\def\vect{\get@in@paren{\@vect}}
\def\@vect{\left(\begin{matrix}\expandafter\@column\my@arg,;\end{matrix}\right)}
\def\get@in@paren#1({\def\my@arg{}\def\my@rest{}\def\after@get{#1}\get@arg}
\let\e@a\expandafter
\def\get@arg#1){\e@a\kl@test\my@rest#1(;}
\def\kl@test#1(#2;{\e@a\def\e@a\my@arg\e@a{\my@arg#1}%
                   \ifx:#2:\let\my@exec\after@get
                   \else\let\my@exec\get@arg
                        \e@a\def\e@a\my@arg\e@a{\my@arg(}%
                        \def@rest#2;%
                   \fi\my@exec}
\def\def@rest#1(;{\def\my@rest{#1\kl@zu}}
\def\kl@zu{)}

\mathchardef\capitaly\mathcode`\Y
\mathcode`\Y=\string"8000

\begingroup \catcode`\Y=\active
  \gdef Y{\capitaly\futurelet\myfirsttoken\testfirsttoken}
\endgroup
\newcommand\testfirsttoken{\ifx\myfirsttoken$\expandafter\getsecondtoken\fi}
\def\getsecondtoken${\futurelet\mysecondtoken\testsecondtoken}
\newcommand\testsecondtoken{\ifx\mysecondtoken.\mkern-2mu\fi
                            \ifx\mysecondtoken,\mkern-2mu$\else$\fi}

\newcommand\MyPairedDelimiter{%
  \@ifstar{\My@Paired@Delimiter{{}}}
          {\My@Paired@Delimiter{}}%
}
\newcommand\My@Paired@Delimiter[4]{%
  \newcommand#2{%
    \@ifstar{\start@PD{#1}{\delimitershortfall=-1sp}{#3}{#4}}
            {\start@PD{#1}{}{#3}{#4}}%
  }%
}
\newcommand\start@PD[5]{%
  #1\mathopen{\mathpalette\put@delim@helper{\put@delim{#2}{#3}{.}{#5}}}%
  #5%
  \mathclose{\mathpalette\put@delim@helper{\put@delim{#2}{.}{#4}{#5}}}%
}
\newcommand\put@delim@helper[2]{%
  \hbox{$\m@th\nulldelimiterspace=0pt #2#1$}%
}
\newcommand\put@delim[5]{%
  \setbox\z@\hbox{$\m@th#5{#4}$}%
  \setbox\tw@\null
  \ht\tw@\ht\z@ \dp\tw@\dp\z@
  #1#5%
  \left#2\box\tw@\right#3%
}

\makeatother
\MyPairedDelimiter*{\abs}{\lvert}{\rvert}
\MyPairedDelimiter*{\norm}{\lVert}{\rVert}
\MyPairedDelimiter{\set}{\{}{\}}

\providecommand{\form}{a}
\newcommand{\bfform}{\mathbf{\form}}
\providecommand{\scpr}[2]{\left( #1 \,\middle|\, #2 \right)}

\renewcommand{\sp}{\scpr}
\newcommand{\from}{\colon}

\newtheorem{Satz}{Satz}[section]
\newtheorem{definition}[Satz]{Definition} 
\newtheorem{theorem}[Satz]{Theorem}
\newtheorem{lemma}[Satz]{Lemma}	
\newtheorem{proposition}[Satz]{Proposition}
\newtheorem{corollary}[Satz]{Corollary}
\newtheorem{remark}[Satz]{Remark}
\newtheorem{example}[Satz]{Example}
\newtheorem{assumption}[Satz]{Assumption}
\newtheorem{notation}[Satz]{Notation}

\allowdisplaybreaks

\title{Approximation of Random Evolution Equations of Parabolic type}
\author{Katharina Klioba\footnote{Technische Universit\"at Hamburg, Institut f\"ur Mathematik, Am Schwarzenberg-Campus 3 (E), 21073 Hamburg, Germany, \email{katharina.klioba@tuhh.de}}, Christian Seifert\footnote{Technische Universit\"at Hamburg, Institut f\"ur Mathematik, Am Schwarzenberg-Campus 3 (E), 21073 Hamburg, Germany, \email{christian.seifert@tuhh.de}}}
\date{\today}

\begin{document}

\maketitle

\begin{abstract}
 In this paper, we present an abstract framework to obtain convergence rates for the approximation of random evolution equations corresponding to a random family of forms determined by finite-dimensional noise. The full discretization error in space, time, and randomness is considered, where polynomial chaos expansion (PCE) is used for the semi-discretization in randomness. The main result are regularity conditions on the random forms under which convergence of polynomial order in randomness is obtained depending on the smoothness of the coefficients and the Sobolev regularity of the initial value. In space and time, the same convergence rates as in the deterministic setting are achieved. To this end, we derive error estimates for vector-valued PCE as well as a quantified version of the Trotter--Kato theorem for form-induced semigroups. We apply the abstract framework to an anisotropic diffusion model with random diffusion coefficients.
 

\end{abstract}

\section{Introduction}

Evolution equations are a classical topic in partial differential equations (PDEs). The functional analytic treatment of a time-dependent PDE results in a description of it as an ordinary differential equation in a Banach or Hilbert space; that is, the PDE is formulated as an abstract Cauchy problem; see, e.g., the classical monographs on this topic \cite{EngelNagel2000,HillePhillips1957, Pazy1983}.

In applications, the coefficients of PDEs often originate from material laws involving material parameters that may be unknown or can only be determined up to some uncertainty. One possibility of modelling such situations with uncertain data consists of choosing suitable random variables (or stochastic processes, random fields) that describe the coefficients in the PDE. This approach results in a random time-dependent PDE, which typically cannot be solved analytically. Assuming well-posedness of the model thus obtained, in order to then approximate the solution, one has to perform suitable approximations with respect to the randomness, the spatial, and the temporal variables. The aim of this paper is to provide such an approximation, including convergence rates, in the framework of random evolution equations given by abstract Cauchy problems with random generators. We work mainly from a functional analytic point of view; we will work solely in the Hilbert space setting.

To be more precise, we focus on parabolic PDEs, where the generators $A_z$ in a spatial Hilbert space $H$ are parametrized by a random parameter $z$, and we study the abstract Cauchy problems
\[
    u_z'(t) = -A_z u_z(t) \quad(t>0), \quad u_z(0) = u_{0,z},
\]
with initial conditions $u_{0,z} \in H$. If we let $(T_z(t))_{t\geq 0}$ be the strongly continuous semigroup generated by $-A_z$, that is, $T_z(t):=\e^{-tA_z}$ for $t\geq 0$, we can thus write
\[u_z(t) = T_z(t) u_{0,z} = \e^{-tA_z} u_{0,z}.\]
Now, the approximation in randomness means approximating the initial condition and the generator $-A_z$ so that the dependence on $z$ is finite-dimensional and, in turn, propagating the approximation to the semigroup and then its application to the (approximated) initial value. This type of uncertainty quantification will be accompanied by a space-time approximation to obtain a fully discretized approximate problem and the corresponding approximate solution. We aim for convergence results for the solutions of the fully discretized approximate problems to the solution of the original problem in a natural norm, including convergence rates depending on suitable smoothness assumptions on the data. Since we work in the functional analytic setting, we essentially deal with approximating the semigroup, and in turn also the resolvents of the generators. In this way, we obtain an additional aspect that complements the usual point of view in numerical analysis, which is aimed at approximating the solutions directly.

\subsection{Random PDEs: a literature overview}
PDEs with parametric or random coefficients have been treated extensively in the past decades; see \cite{GhanemSpanos1991, LeMaitreKnio2010, Sullivan2015, Xiu2010} and references therein, as well as \cite{BabuskaNobileTempone2007, BabuskaTemponeZouraris2004, BachmayrCohenDungSchwab2017, ChkifaCohenSchwab2015, CohenDevoreSchwab2011, EigelGittelsonSchwabZander2015, FrauenfelderSchwabTodor2005, GottliebXiu2008, TodorSchwab2007, XiuHesthaven2005, XiuKarniadakis2002, XiuKarniadakis2002a, XiuKarniadakis2003,   XiuShen2009} (just to mention a few), with an emphasis on stochastic computation as well as approximation. From a numerical analysis point of view, a large portion of the references mentioned above are devoted to elliptic PDEs.
In this stationary elliptic case, we note that we can use the tensor product structure of Bochner--Lebesgue $L_2$-spaces, so that discretizations in randomness and space can be performed simultaneously; see, e.g., \cite{BabuskaTemponeZouraris2004}. Moreover, the case in which the elliptic coefficient is affine in the random inputs has been studied thoroughly; see, e.g., \cite{CohenDevoreSchwab2011, FrauenfelderSchwabTodor2005,  TodorSchwab2007}. In contrast, we consider parabolic, that is, time-dependent problems and allow for a fairly general dependence of the coefficients on the random inputs; in particular, they need not be analytic as in \cite{ChkifaCohenSchwab2015} (where parabolic problems are also treated).
Concerning parabolic problems, we mention \cite{BarthStein2022}, where random pathwise space-time discretizations for parabolic second-order PDEs with scalar diffusion coefficients have been studied, as well as \cite{ HoangSchwab2013, NobileTempone2009}. In the latter, the scalar diffusion coefficients are affine in the random variables with bounded support, or their distribution and semi-discretizations in space and randomness are considered. Moreover, in \cite{ZhangGunzburger2012}, stochastic collocation in combination with a full space-time discretization was considered. Furthermore, \cite{GyoengyMillet2009} provides convergence rates for space-time discretizations for a class of stochastic evolution equations with multiplicative noise.

\subsection{Main Results}
To illustrate our results, let us consider, as an example, the heat equation with random coefficients
\begin{alignat*}{4}
    \partial_t \widetilde{u}(t,x,\omega) & = \diverg_x \widetilde{M}_\omega(x) \grad_x \widetilde{u}(t,x,\omega) &&\quad(t>0, x\in G, \omega\in\Omega),\\
    \widetilde{u}(t,x,\omega) & = 0 &&\quad(t>0, x\in\partial G, \omega\in\Omega), \\
    \widetilde{u}(0,x,\omega) & = \widetilde{u}_0(x,\omega) &&\quad(x \in G,\omega\in\Omega),
\end{alignat*}
where $G\subseteq\R^d$ is a bounded domain, $(\Omega,\calF,\P)$ is a probability space,  for $\omega\in\Omega$, $\widetilde{M}_\omega$ maps to $\K^{d\times d}$ and is $\P$-almost surely uniformly elliptic and bounded, and $\widetilde{u}_0$ is the initial condition. Here, $\K\in\{\R,\C\}$ is the scalar field of all vector spaces appearing throughout and the divergence and gradient are taken w.r.t.\ the spatial variable $x$.
We will assume that the random fields $(x,\omega)\mapsto \widetilde{M}_\omega(x)$ and $(x,\omega)\mapsto u_0(x,\omega)$ are determined by \textit{finite-dimensional noise}; that is, there exists a random vector $Z\from \Omega\to\R^N$ such that $\widetilde{M}_\omega(x) = M_{Z(\omega)}(x)$ for all $x\in G$ and $\omega\in\Omega$ for some $M\from G \times \R^N\to \K^{d\times d}$. Such finite-dimensional noise can be obtained by, e.g., a truncated Karhunen--Lo\`{e}ve expansion \cite{KacSiegert1947,Karhunen1947, Loeve1948}.
Pushing the random heat equation forward w.r.t.\ $Z$, we obtain another random heat equation, with random parameter $z \in \R^N$ w.r.t.\ the probability measure $\P_Z$: 
\begin{alignat*}{4}
    \partial_t u(t,x,z) & = \diverg_x M_z(x) \grad_x u(t,x,z) &&\quad(t>0, x\in G, z\in\R^N),\\
    u(t,x,z) & = 0 &&\quad(t>0, x\in\partial G, z\in\R^N), \\
    u(0,x,z) & = u_0(x,z) &&\quad(x\in G,z\in\R^N).
\end{alignat*}
We now take the functional analytic point of view by setting $H\ce L_2(G)$ and $A_z\ce -\diverg_x M_z \grad_x$ with domain $D \ce H^2(G)$ in $H$ for $z\in\R^N$. Note that $A_z$ is the operator associated with some form $\form_z$ on some Hilbert space $V \hra H$ for $z\in\R^N$. In this specific case, $V \ce H_0^1(G)$ and $\form_z(u,v) \ce - \sp{M_z\grad_x u}{\grad_x v}$. This results in the family of random abstract Cauchy problems
\[
    u_z'(t) = -A_z u_z(t) \quad(t>0), \quad u_z(0) = u_{0,z}
\]
for $z\in\R^N$. Considering $(A_z)_{z\in \R^N}$ as an operator $\bfA$ on the Hilbert space $\bfH\ce L_2(\R^N,\PZ;H)$ results in the abstract Cauchy problem
\[\bfu'(t) = -\bfA \bfu(t) \quad(t>0),\quad \bfu(0) = \bfu_0.\]

We now perform three independent discretizations:
\begin{enumerate}[label=\arabic*.)]
\item discretization in randomness: here we choose the (generalized) Polynomial Chaos Expansion (PCE); cf. \cite{Wiener1938, XiuKarniadakis2002},
\item discretization in space: here we choose an abstract Galerkin method (a finite element method for the concrete example above); cf. \cite{Bubnov1913, Galerkin1915},
\item discretization in time: here we use an $A$-stable method; cf. \cite{BrennerThomee1979, Dahlquist1963}.
\end{enumerate}
Clearly, in the deterministic case, approximation results and convergence rates are classical and well-known; see, e.g., \cite{FujitaSaitoSuzuki2001, Thomee2006} and references therein. As it turns out, there is an intricate relationship between spatial and temporal discretization. We will exploit the corresponding relationship for the discretization in randomness and in space-time, and find a similar behaviour in this case.

Using PCE, we obtain a deterministic coupled system as an approximation, which again has the form of an abstract Cauchy problem, now in $\frakH_n\ce H^{d_n}$, where $d_n$ is the number of $H$-valued equations and is determined by the discretization parameter $n$ and the dimension $N$ of the noise. This system can then be approximated in space with discretization parameter $m$ and in time with discretization parameter $k$, resulting in a fully discretized approximation $\bfu_{n,m,k}$. What we will trace explicitly is the interplay between these two discretization steps (randomness vs.\ space-time).

Our main theorem, which is formulated in Theorem \ref{thm:jointrate_symmetric}, provides a joint convergence rate induced by the corresponding rates for the single discretizations.

\begin{theorem}
    Suppose that the assumptions of Theorem \ref{thm:jointrate_symmetric} hold. In particular, assume that for some $\ell \in \N_0$, the random family of forms $(\form_z)_{z \in \R^N}$ is $\PZ$-almost surely uniformly bounded and uniformly coercive and satisfies $[z \mapsto \form_z(u,v)] \in C^\ell(\R^N)$ for all $u,v\in V$. Further, suppose there exists a constant $C_\ell\ge 0$ such that
    \begin{equation}
    \label{eq:mainAssIntro}
        |\rho(z)^{\alpha/2}\partial_z^\alpha a_z(u,v)| \le C_\ell \|u\|_D \|v\|_H 
    \end{equation}
    for $\PZ$-almost every $z\in\R^N$ and for all $u \in D$, $v \in V$, and multi-indices $\alpha \in \N_0^N$ with $|\alpha| \le \ell$. Let $\bfJ_{n,m}$ be a suitable embedding for $n,m \in \N$ and let $s>0$. Then for sufficiently smooth (in randomness and space) initial values $\bfu_0$, we have 
    \begin{align*}
        \|\bfJ_{n,m}\bfu_{n,m,k}(t)-\bfu(t)\|_\bfH
        & \le C \bigl(n^{-\ell} + m^{-\px} + \tau^\pt\bigr) \norm{\bfu_0}_{H_\rho^{2\ell}(\R^N,\PZ;D^{\max\{s+1,\pt\}})}
    \end{align*}        
        for all $n,m,k\in\N$ and temporal grid points $t$.
\end{theorem}

Here, $n$ is the PCE parameter, $m$ the spatial discretization parameter, $k$ the temporal discretization parameter, and $\tau$ the maximal step size for the time discretization. Moreover, $\ell\in\N$, $\px,\pt >0$ are the rates for the PCE, the spatial and the temporal methods, respectively, and $s>0$ is such that the spatial method converges with rate $\px$ on $D^s$, the common (in $z$) domain of the $A_z^s$'s. Further, $\rho$ is a suitable weight needed in the case of exponential distributions $\PZ$. It can be omitted for normal, uniform, or, more generally, Beta distributed $\PZ$. In particular, we do not require bounded supports of the distributions of the random variables. If the forms $\form_z$ are not $\PZ$-almost surely symmetric, additional spatial regularity is required, as detailed in \cite{thesisKatharina}.

\subsection{Some words on our method of proof}

In order to show the main theorem, we follow the discretization steps as indicated above, i.e.\ we first make use of PCE in order to obtain a deterministic coupled system; see also \cite{Xiu2010}. On the level of scientific computation, this would allow making use of existing code for deterministic systems in order to numerically solve these problems. However, as our objective is to obtain an analytic approximation result, we then perform the discretizations in space and time in order to obtain a fully discretized system.

One has to be careful to use PCE in a meaningful way to maintain the abstract Cauchy problem structure when deriving the deterministic system. In particular, the semi-discretization $\bfu_n(t)$ of $\bfu(t)$ in randomness that still satisfies an abstract Cauchy problem does not agree with the PCE approximation of $\bfu(t)$ in most cases. Therefore, the estimation of the semi-discretization error $\|\bfu(t)-\bfu_n(t)\|_\bfH$ is more than a simple application of a PCE error bound. Instead, in Theorem \ref{thm:errorunuTrotterKato}, we employ a modified Trotter--Kato argument. In order to prove this Trotter--Kato argument, we provide pointwise in $z$ estimates of $A_z$ and its resolvent as well as its semigroup in Lemma \ref{lem:PointwiseEstimatesCoeffCond}. Here we make use of one of our core assumptions, namely estimates on the derivative (in $z$) of the forms $\form_z$ as stated in \eqref{eq:mainAssIntro}. The pointwise estimates are then lifted to mapping properties of $\bfA$ and its resolvent as well as its semigroup on suitable Sobolev subspaces of $\bfH$ in Proposition \ref{prop:SobolevEstBfA} via a composition estimate. In analogy to the deterministic case, we find that spatial regularity is required for the estimation of the semi-discretization error in randomness.

To estimate the semi-discretization error in space-time, standard theory yields an error estimate involving the graph norm of the vector of PCE coefficients of the initial values $\bfu_0$. In order to pass from the norm of this coefficient vector back to the norm of $\bfu_0$ in a suitable subspace of $\bfH$, we follow an approach based on the Heinz inequality in Lemma \ref{lem:DfrakAnalpha}.

We note that we treat linear equations in this article mainly for simplicity. We comment on possible extensions to non-linear situations in Remark \ref{rem:nonlinear_extensions}.

\subsection{Outline of the paper}
Let us outline the rest of the paper. In Section \ref{sec:RandomEvolutionEquations}, we formally introduce the random evolution equations considered. In Section \ref{sec:approx}, we collect the facts needed for the three discretizations separately, including a novel quantified version of the Trotter--Kato theorem for form-induced semigroups in Theorem \ref{thm:result_evolution} and error estimates for vector-valued multivariate PCE in Subsection \ref{subsec:random}. We recall deterministic space-time discretization in Section \ref{sec:jointRate_deterministic}. The main part is Section \ref{sec:jointRate}, where we prove the convergence theorem as indicated above, including the rates, in Theorem \ref{thm:jointrate_symmetric}. We apply the result in Section \ref{sec:applications} to a parabolic equation with random anisotropic diffusion. In Appendix \ref{sec:appendix}, we collect estimates for some Sturm--Liouville operators required in Subsection \ref{subsec:random}.

\section{Random Evolution Equations}
\label{sec:RandomEvolutionEquations}

Let $H$ be a separable Hilbert space and $(\Omega,\calF,\P)$ a probability space.
For $\omega\in\Omega$ let $A_\omega$ be a closed and densely defined operator in $H$.
Let $\calN_\Omega\subseteq \Omega$ be a $\P$-null set such that for all $\omega \in\Omega\setminus\calN_\Omega$ the operator $-A_\omega$ generates a strongly continuous semigroup $(T_\omega(t))_{t\geq 0}$ and there exists $\sigma\in\R$ such that $(\sigma,\infty)$ is contained in the resolvent set of $-A_\omega$ for all $\omega\in\Omega\setminus \calN_\Omega$. 

Suppose that $(A_\omega)_{\omega\in\Omega}$ is measurable, i.e., $\Omega\ni \omega\mapsto \1_{\Omega\setminus \calN_\Omega}(\omega) (\lambda-A_\omega)^{-1} \in \calL(H)$, where $\calL(H)$ is the space of bounded linear operators on $H$, is weakly measurable for some, and hence all, $\lambda \in (-\infty,-\sigma)$. Note that by Pettis' measurability theorem \cite[Corollary 1.11]{Pettis1938} and separability of $H$ this is equivalent to strong measurability of $\Omega\ni \omega\mapsto \1_{\Omega\setminus \calN_\Omega}(\omega) (\lambda-A_\omega)^{-1}$ for all $\lambda \in (-\infty,-\sigma)$. Moreover, the exponential formula then yields that $\Omega \ni \omega \mapsto \1_{\Omega\setminus\calN_\Omega}(\omega) T_\omega(t)$ is strongly measurable for all $t\geq 0$.

We consider the random evolution equation
\[u_\omega'(t) = -A_\omega u_\omega(t) \quad(t>0), \quad u_\omega(0) = u_{0,\omega}\in H.\]
We will study this random family of evolution equations in the following way. 
Let $\bfA$ in $\bfH\ce L_2(\Omega,\P)\otimes H \cong L_2(\Omega,\P;H)$ be the operator defined by
\begin{align*}
    \dom(\bfA) & \ce \Big\{f\in \bfH;\; f(\omega)\in \dom(A_\omega) \text{ for $\P$-a.e.~$\omega\in\Omega$},\, \int_{\Omega} \norm{A_\omega f(\omega)}_H^2\,\rmd \P(\omega) < \infty\Big\},\\
    \bfA f & \ce \bigl[\Omega\ni \omega\mapsto A_\omega f(\omega) \in H\bigr]
\end{align*}

As a consequence of \cite[Proposition 2.3.11]{Thomaschewski2003}, we obtain the following proposition. 

\begin{proposition}\label{prop:multiplication_operator}
    Let $K\geq 1$, $\sigma \in\R$ such that $\norm{T_\omega(t)}\leq K\e^{\sigma t}$ for all $t\geq 0$ and $\omega\in \Omega\setminus\calN_\Omega$.
    Then $-\bfA$ is the generator of a $C_0$-semigroup $(\bfT(t))_{t\geq0}$ satisfying $\norm{\bfT(t)}\leq K\e^{\sigma t}$ for all $t\geq 0$. Furthermore, for $t\geq 0$ and $f\in \bfH$ we have $(\bfT(t)f)(\omega) = \1_{\Omega\setminus \calN_\Omega}(\omega)T_\omega(t)f(\omega)$ for $\P$-a.e.~$\omega\in\Omega$.
\end{proposition}

We then consider
\[\bfu'(t) = -\bfA \bfu(t) \quad(t>0), \quad \bfu(0) = \bfu_0\in \bfH.\]
In order to approximate the solution $\bfu$, we have to combine three types of approximation, namely \emph{spatial} approximation taking care of approximation w.r.t.\ the space $H$, \emph{temporal} approximation for the time $t$ (typically considered on bounded intervals), and \emph{randomness} approximation taking care of the randomness in $L_2(\Omega,\P)$.

\section{Approximation Methods}
\label{sec:approx}

In this section, we review the three types of approximation separately. We start with spatial approximation, then turn to temporal approximation, and conclude with randomness approximation.

\subsection{Spatial approximation}
\label{subsec:spatial}

Let $V,H$ be separable Hilbert spaces, $V\hookrightarrow H$ with $V$ dense in $H$ such that we obtain a Gelfand triple $V\hookrightarrow H \hookrightarrow V^*$ with the antidual $V^*$ of $V$. 

A \emph{form} $\form$ on $V$ is a sesquilinear mapping $\form\from V \times V \to \K$, and we write $\form(u)\ce\form(u,u)$ for the corresponding quadratic form. For a given form $\form$ we call $\form^*\from V\times V\to \K$, $\form^*(u,v)\ce \overline{\form(v,u)}$ the \emph{adjoint form}.
A form $\form$ is called \emph{sectorial} if there exists $c\geq 0$ such that $\abs{\Im \form(u)} \leq c\Re \form(u)$ for all $u\in V$ or, put differently, the \emph{numerical range} $\calN(\form) \ce \set{\form(u);\; u\in V ,\, \norm{u}_V=1}$ is contained in the sector $\Sigma_\theta\ce\set{z \in \C\setminus \set{0};\; \abs{\arg z} \leq \theta}\cup\set{0}$ with $\theta\ce\arctan c$.
We call $\form$ \emph{symmetric} if $\form^* = \form$ and \emph{bounded} if there exists $K\geq 0$ such that $\abs{\form(u,v)} \leq K \norm{u}_V \norm{v}_V$ for all $u,v \in V$. 
Moreover, $\form$ is called \emph{coercive} if there exists $\kappa>0$ such that $\Re \form(u) \geq \kappa \norm{u}_V^2$ for all $u \in V$. Note that bounded coercive forms are sectorial of angle less than $\frac{\pi}{2}$.

Let $\form$ be a bounded form. Then we set $\calA\from V \to V^*$, $\calA u \ce\form(u,\cdot)$ and define the operator $A$ in $H$ associated with $\form$ via $\dom(A)\ce\calA^{-1}(H)\subseteq H$ and $A \ce \calA|_{\dom(A)}$, i.e.,
\begin{equation*}
    \form(u,v) = (\calA u)(v)=\sp{Au}{v}_H,\quad (u \in \dom(A), v\in V).
\end{equation*}
Additionally, suppose that $\form$ is coercive.
Then by the Lax--Milgram lemma $\calA$ is an isomorphism, which yields that the stationary problem is well-posed, i.e.,
\begin{equation*}
    \forall\, F\in V^*\,\exists! \, u\in V\,\forall\,v \in V: \form(u,v) = F(v) \quad\text{and}\quad \norm{u}_V \leq \norm{\calA^{-1}}\norm{F}_{V^*}.
\end{equation*}
Moreover, $A$ is m-accretive and therefore $-A$ generates a contractive and (in case $\K=\C$) analytic $C_0$-semigroup $T\from [0,\infty)\to \calL(H)$, i.e., $u\from[0,\infty)\to H$ is a solution of
\[
    u'(t) = -Au(t) \quad(t>0),\quad u(0) = u_0\in H
\]
if and only if $u(t) = T(t)u_0$ ($t\geq 0$). Furthermore, $A$ admits fractional powers and $0$ lies in the resolvent set of $A$. 

\begin{remark}
\label{rem:graph_norm}
    Let $\form$ be a bounded and coercive form and $A$ the associated operator. Then, for $s>0$, zero also lies in the resolvent set of $A^s$ (see e.g. \cite[Proposition~3.1.1(e)]{Haase2006}). Thus, the graph norm $\|\cdot\|_{A^{s}}$ of $A^s$ defined as $\|\cdot\|_{A^{s}}\ce (\|A^s \cdot\|_H^2+\|\cdot\|_H^2)^{1/2}$ and the norm $\|A^s\cdot\|_H$ are equivalent on $\dom(A^s)$. Indeed, for $u\in \dom(A^s)$, we have
    \[
        \norm{A^s u}_H^2 \leq \norm{u}_{A^s}^2 = \norm{u}_H^2 + \norm{A^s u}_H^2 
        \leq \big(\big\|(A^s)^{-1}\big\|^2_{\calL(H)}+1\big)\norm{A^s u}_H^2.
    \]
\end{remark}

Let us now turn to approximation.

\begin{definition}
    An \emph{approximating sequence} of $V$ is a sequence $(V_m)_{m\in\N}$ of closed subspaces of $V$ such that $\dist_V(v,V_m)\to 0$ as $m \to \infty$ for all $v\in V$.
\end{definition}
For an approximating sequence $(V_m)_{m\in\N}$ of $V$ and $m\in\N$ let $H_m\subseteq H$ be the closure of $V_m$ in $H$. Thus, $V_m\hookrightarrow H_m$, and, trivially, $V_m$ is dense in $H_m$ for all $m\in\N$. Note that typically (e.g., in applications), approximating sequences consist of finite-dimensional subspaces, and then $H_m=V_m$ as vector spaces.

\begin{definition}
    A \emph{space discretization} of $A$ consists of an approximating sequence $(V_m)_{m\in\N}$ of $V$ and a sequence $(A_m)_{m\in\N}$ of operators $A_m$ in $H_m$ such that $-A_m$ generates a $C_0$-semigroup $(T_m(t))_{t\geq 0}$ on $H_m$ for $m\in\N$. If the approximating sequence used is clear from the context, we also refer to $(A_m)_{m \in \N}$ as a space discretization of $A$.
\end{definition}

Let $(V_m)_{m\in\N}$ be an approximating sequence of $V$, and let $P_m\from H\to H_m\subseteq H$ be the $H$-orthogonal projection of $H$ onto $H_m$ and $J_m\from H_m\to H$ the canonical embedding for $m\in\N$. Then $J_mP_m\to I$ strongly. Indeed, $(J_mP_m)_{m\in\N}$ is uniformly bounded and strong convergence for $f\in V$ holds due to $\|J_mP_mf-f\|_H = \|P_mf-f\|_H = \dist_H(f,H_m) \le \CVH \dist_V(f,V_m) \to 0$, where $\CVH$ denotes the embedding constant. Strong convergence on $H$ now follows from the uniform boundedness principle.

For $m\in\N$, we define the approximating forms $\form_m\ce\form|_{V_m\times V_m}$. Then $\form_m$ is trivially bounded and coercive for all $m\in\N$ with the same constants as $\form$. Let $\calA_m\from V_m\to V_m^*$, $\calA_m u\ce\form_m(u,\cdot)$, and let $A_m$ be the m-accretive operator in $H_m$ associated with $\form_m$ for $m\in\N$.
Further, let $T_m\from [0,\infty)\to \calL(H_m)$ be the contractive and (in case $\K=\C$) analytic $C_0$-semigroup generated by $-A_m$. Hence, for given initial data $u_{0,m}\in H_m$, $u_m\from[0,\infty)\to H_m$ is a solution of
\[
    u_m'(t) = -A_mu_m(t) \quad(t>0),\quad u_m(0) = u_{0,m}
\]
if and only if $u_m(t) = T_m(t)u_{0,m}$ for $t\geq 0$. Thus, we obtain a space discretization of $A$, which is called the \emph{(Bubnov-)Galerkin discretization}. Coercivity of the form ensures convergence of Galerkin approximations by virtue of the \emph{uniform Banach--Ne\v{c}as--Babu\v{s}ka condition (BNB)}; cf.\ e.g.\ \cite[Proposition~2.4]{arendtGalerkin}.

To quantify the convergence of the approximating semigroups $(T_{m}(t))_{t\geq0}$ to $(T(t))_{\geq0}$, we need the notion of \emph{spatial convergence rates} and two quantities relating the approximation speed of $(V_m)_{m \in \N}$ to the properties of the operators $\calA$ and $\calA^*$, respectively.
\begin{definition}
\label{def:convratespace}
    Let $\Yx \hra H$ and $\px> 0$. 
    \begin{enumerate}
    \item
    A space discretization $(A_m)_{m \in \N}$ of $A$ is said to be \emph{convergent of order $\px$} on $\Yx$ (w.r.t.~the norm of $H$) \emph{for the stationary problem} if there exists a constant $\Cx\geq 0$ such that
    \begin{align*}
        \norm{J_mA_m^{-1} P_mu_0-A^{-1}u_0}_{H}\le \Cx \frac{\norm{u_0}_{\Yx}}{m^{\px}} \quad (u_0 \in \Yx,\, m \in \N).
    \end{align*}
    \item
    A space discretization $(A_m)_{m \in \N}$ of $A$ is said to be \emph{convergent of order $\px$} on $\Yx$ (w.r.t.~the norm of $H$) \emph{for the evolution problem} if for all $\bar{T}\geq 0$ there exists a constant $\Cx=\Cx(\bar{T})\geq 0$ such that
    \begin{align*}
        \norm{J_mT_m(t)P_mu_0-T(t)u_0}_{H}\le \Cx \frac{\norm{u_0}_{\Yx}}{m^{\px}} \quad (t\in[0,\bar{T}],\, u_0 \in \Yx,\, m \in \N).
    \end{align*}
     \end{enumerate}
\end{definition}

\begin{definition}[based on (6.6) in \cite{arendtGalerkin}]
    Let $\calX \hra V^*$ be a Banach space and $(V_m)_{m \in \N}$ an approximating sequence of $V$. For $m \in \N$, define
    \begin{align*}
        \gamma_m(\calX) &\ce \sup_{f \in \calX,\,\|f\|_\calX=1} \dist_V(\calA^{-1}f, V_m)
    \end{align*}
    and $\gamma_m^*(\calX)$ analogously with $\calA$ replaced by $\calA^*$. We say that $(\gamma_m(\calX))_{m\in\N}$ and $(\gamma_m^*(\calX))_{m\in\N}$ \emph{decay with rate} $p_1,p_2 \geq 0$, respectively, if there exist constants $C_{\gamma,\calX}, C_{\gamma,\calX}^*\geq0$ such that for all $m \in \N$
    \begin{align*}
        \gamma_m(\calX) \leq \frac{C_{\gamma,\calX}}{m^{p_1}} \quad\text{and}\quad
            \gamma_m^*(\calX) &\le \frac{C^*_{\gamma,\calX}}{m^{p_2}}. 
    \end{align*}
\end{definition} 
As an immediate consequence of this definition,
\begin{equation*}
    \dist_V(u,V_m) \le \gamma_m(\calX) \|\calA u\|_\calX \quad (u \in \calA^{-1}\calX).
\end{equation*}
The sum of the respective decay rates yields a convergence order for the stationary problem for suitable $\calX$.

\begin{theorem}
\label{thm:result_stationary}
     Let $\form\from V\times V\to \C$ be a bounded coercive form and $(A_m)_{m \in \N}$ a space discretization of $A$. Let $s>0$, $p_1, p_2\geq 0$ such that $(\gamma_m(\dom(A^s)))_{m\in\N}$ decays with rate $p_1=p_1(s)$ and $(\gamma_m^*(H))_{m\in\N}$ decays with rate $p_2$. Then
     there exists a constant $\Cx=\Cx(s)\geq 0$ such that
    \begin{align*}
        \norm{J_mA_m^{-1} P_mu_0-A^{-1}u_0}_{H}\le \Cx \frac{\norm{u_0}_{A^{s}}}{m^{p_1+p_2}}
    \end{align*}
     for all $u_0 \in \dom(A^s)$ and $m \in \N$. 
     Moreover, if $p_1+p_2>0$, the space discretization $(A_m)_{m\in\N}$ converges of order $p_1+p_2$ on $\dom(A^{s})$ for the stationary problem.
\end{theorem}
Since $p_1$ depends on $s$, the order of convergence scales with the smoothness of the initial value encoded in the domain of the $s$-th fractional power of $A$. The proof is essentially contained in \cite[Theorem 6.3]{arendtGalerkin} and yields the possible choice of constant
$\Cx = \Cx(s) = K^2\kappa^{-1}C_{\gamma,\dom(A^s)}C^*_{\gamma,H}$, where $K$ and $\kappa$ denote the boundedness and coercivity constants of $\form$, respectively. Note that \cite{arendtGalerkin} requires finite-dimensional subspaces for approximating sequences, which is not needed for \cite[Theorem 6.3]{arendtGalerkin}.
We now turn to the evolution problem. 

\begin{theorem}
\label{thm:result_evolution}
     Let $\form\from V\times V\to \C$ be a bounded, coercive, and symmetric form and $(A_m)_{m \in \N}$ a space discretization of $A$. Let $s>0$, $p_1, p_2\geq 0$ such that $(\gamma_m(\dom(A^s)))_{m\in\N}$ decays with rate $p_1=p_1(s)$ and $(\gamma_m^*(H))_{m\in\N}$ decays with rate $p_2$. Then for all $\bar{T}\geq 0$ there exists $\Cx=\Cx(\bar{T},s)$ such that
    \[
        \norm{J_mT_m(t)P_mu_0-T(t)u_0}_H \leq \Cx \frac{\norm{u_0}_{A^{s+1}}}{m^{p_1+p_2}}
    \]
    for all $t\in[0,\bar{T}]$, $u_0\in \dom(A^{s+1})$, and $m\in\N$. Moreover, if $p_1+p_2>0$, the space discretization $(A_m)_{m\in\N}$ converges of order $p_1+p_2$ on $\dom(A^{s+1})$ for the evolution problem.
\end{theorem}

\begin{proof}
    (i) 
    Let $K,\kappa>0$ be the boundedness and coercivity constants of $\form$ and let $\Delta_m$ 
    be defined by the resolvent difference in $H$ (restricted to $\dom(A^s)$), i.e.,
    \begin{align*}
        \dom(\Delta_m) \ce \dom(A^s),\quad
        \Delta_m f \ce J_mA_m^{-1}P_mf-A^{-1}f \quad(f\in\dom(\Delta_m)).
    \end{align*}
    By Theorem \ref{thm:result_stationary}, we have
    \begin{align*}
        \norm{\Delta_m f}_H = \norm{(J_mA_m^{-1}P_m-A^{-1})f}_{H} &
        \leq \frac{K^2}{\kappa}C_{\gamma,\dom(A^s)}C^*_{\gamma,H} \frac{\norm{f}_{A^s}}{m^{p_1+p_2}} 
    \end{align*}
    for all $f\in \dom(A^s)$.
    
    (ii)
    Let $\bar{T}\geq 0$. Note that symmetry of $\form$ implies symmetry of $\form_m$ and thus self-adjointness of $A_m$ for all $m\in\N$.
    Then by \cite[Proposition~2.3a]{kappel1}, there exists $C=C(\bar{T},s)\geq 0$ such that
    \begin{align*}
    \label{eq:quantConvRate}
        \norm{J_mT_m(t)P_mu_0-T(t)u_0}_H &\leq C \norm{\Delta_m}_{\calL(\dom(A^s),H)}\norm{u_0}_{A^{s+1}}
        \le \frac{ K^2}{\kappa}C C_{\gamma,\dom(A^s)}C_{\gamma,H}^* \frac{\norm{u_0}_{A^{s+1}}}{m^{p_1+p_2}} 
    \end{align*}
    for all $t\in[0,\bar{T}]$, $u_0\in \dom(A^{s+1})$, and $m\in\N$, where we used (i) to estimate the norm of $\Delta_m$.
\end{proof}

\begin{remark}
\label{rem:stationary-evolution}
    Theorem \ref{thm:result_evolution} actually states that convergence of the spatial discretization for the stationary problem can be turned into convergence of the spatial discretization for the evolution problem by slightly increasing the regularity required for the initial value. A version of Theorem \ref{thm:result_evolution} is available for non-symmetric forms provided that $u_0 \in \dom(A^{s+1+\varepsilon})$ for some $\varepsilon >0$ based on \cite[Proposition~2.2b]{kappel1} (see \cite[Theorem~3.8]{thesisKatharina}). If we consider time intervals $[(\bar{T})^{-1},\bar{T}]$ and symmetric forms, convergence of order $p_1+p_2$ is already obtained on $D(A^{s+1/2})$, cf. \cite[Corollary~3.9b]{thesisKatharina}. 
\end{remark}

\subsection{Temporal approximation}
\label{subsec:temporal}

Let $X$ be a Banach space and $(T(t))_{t\geq 0}$ a $C_0$-semigroup on $X$ with generator $-A$. Recall that a subspace $D\seq X$ is called a \textit{core} for $A$ if $\overline{A|_D} = A$.

\begin{definition}
\label{def:timediscretization}
    A \emph{time discretization method} for $(T(t))_{t\geq 0}$ is a function 
    $F\from [0,\infty)\to \calL(X)$ satisfying 
    $F(0) = I$. Let $F\from [0,\infty)\to \calL(X)$ be a time discretization method for $(T(t))_{t\geq 0}$. 
    Then $F$ is called \emph{stable} if there exist $K\geq 1$ and $\lambda\in\R$ such that for all $\Nt\in \N$ and $(\tau^i)_{1\leq i\leq \Nt}$ in $[0,\infty)$, we have
    \[
        \bigg\|\prod_{i=1}^\Nt F(\tau^i)\bigg\|\leq K\e^{\lambda \sum_{i=1}^\Nt \tau^i}.
    \]
    Let $\Yt\hra X$ with $\Yt\subseteq \dom(A)$ be a core for $-A$, and $\pt>0$. Then $F$ is called \emph{consistent of order $\pt$} on $\Yt$ if there exist $C\geq 0$ and $\tau_0>0$ such that for all $\tau\in[0,\tau_0]$, we have
    \[
        \norm{F(\tau)f-T(\tau)f} \leq C \tau^{\pt+1} \norm{f}_{\Yt} \quad(f\in \Yt).
    \]
    Moreover, $F$ is called \emph{convergent of order $\pt$} on $\Yt$ if for all $\bar{T}\geq 0$ there exist $C_{\bar{T}}\geq 0$ and $\tau_0=\tau_0(\bar{T})\in (0,\bar{T}]$ such that for all $t\in [0,\bar{T}]$, $\Nt\in\N$, and $(\tau^i)_{1\leq i\leq \Nt}$ with $\tau^i=\tau^i(t,\Nt)\in [0,\tau_0]$ for $1 \le i \le \Nt$ and  $\sum_{i=1}^\Nt \tau^i = t$, we have
    \[
        \bigg\|\prod_{i=1}^\Nt F(\tau^i)f - T(t)f\bigg\| \leq C_{\bar{T}} \Big(\max_{i=1,\ldots,\Nt} \tau^i\Big)^\pt \norm{f}_{\Yt} \quad (f\in \Yt).
    \]
\end{definition}

\begin{remark}
    Let $\Yt\hra X$ with $\Yt$ dense in $X$ such that $\Yt\subseteq \dom(A)$ and $\Yt$ is invariant under $(T(t))_{t\geq 0}$. Then $\Yt$ is a core for $-A$; cf.\ \cite[Proposition II.1.7]{EngelNagel2000}. Note that every core for $-A$ is dense in $X$.
\end{remark}

\begin{example}\label{ex:time_discretization_methods}
    We recall classical time discretization methods.
    \begin{enumerate}[label=(\alph*)]
     \item 
     Clearly, the trivial choice $F=T$ yields a stable and consistent and therefore also convergent time discretization method of arbitrary order on $X$. 
     \item \label{item:IEintro}
     Let $\lambda\geq 0$ such that $[\lambda,\infty)$ is contained in the resolvent set of $-A$, and define $F\from [0,\infty)\to \calL(X)$ by 
     \[F(\tau)\ce\begin{cases}
        I, & \tau=0,\\
        (I+\tau A)^{-1} = \frac{1}{\tau}(\frac{1}{\tau}+A)^{-1}, & \tau\in (0,\frac{1}{\lambda}),\\
        (I+\frac{1}{\lambda} A)^{-1}, & \tau\geq \frac{1}{\lambda}.
        \end{cases}\]
    Then $F$ is called the \emph{implicit Euler method} and is stable and consistent of order $1$ on $\dom(A^2)$. More generally, it converges of order $\pt \in (0,1]$ on $\dom(A^{2\pt})$ if $(T(t))_{t \ge 0}$ is bounded, see \cite[Theorem~1.3]{GomilkoTomilov14}. If, additionally, $T$ is analytic, the convergence order $\pt \in (0,1]$ is attained already on $\dom(A^\pt)$. In particular, if $-A$ generates a contractive $C_0$-semigroup, then $\lambda=0$ and $\norm{F(\tau)}\leq 1$ for all $\tau\geq 0$.
    \item \label{item:CNintro}
    Let $\lambda\geq 0$ such that $[\lambda,\infty)$ is contained in the resolvent set of $-A$, and define $F\from [0,\infty)\to \calL(X)$ by 
     \[F(\tau)\ce\begin{cases}
        I, & \tau=0,\\
        (I-\frac{\tau}{2}A)(I+\frac{\tau}{2} A)^{-1}, & \tau\in (0,\frac{2}{\lambda}),\\
        (I-\frac{1}{\lambda}A)(I+\frac{1}{\lambda} A)^{-1}, & \tau\geq \frac{2}{\lambda}.
        \end{cases}\]
    Then $F$ is called the \emph{Crank--Nicolson scheme} and is consistent, but may not always be stable, as the left shift semigroup on $C_0(\R)$ illustrates. Provided that $F$ is stable, it converges of order $\pt \in (0,2]$ on $D(A^{3\pt/2})$ for bounded semigroups \cite[Corollary~4.4]{Kovacs07}.
    \end{enumerate}
\end{example}

Implicit Euler and Crank--Nicolson are instances of a larger, highly relevant class of methods, so-called \emph{rational schemes}. 

\begin{definition}
\label{def:rational_time_discretization}
    Let $r\from \C \to \C\cup\{\infty\}$ be a rational function. We say that a time discretization method $F$ for $(T(t))_{t\ge 0}$ is \emph{induced by $r$} if $F(\tau)=r(\tau\cdot(-A))$ for all sufficiently small $\tau >0$. Such schemes $F$ are called \emph{rational schemes}.
    
    If $\abs{r(z)}\leq 1$ for all $z\in\C$ with $\Re z \leq 0$ and $r(z)-\e^z = o(z)$ as $z\to 0$ then $r$ is called \emph{A-acceptable}. It is called \emph{A-stable} if $\abs{r(z)}<1$ for all $z\in \C$ with $\Re z <0$. 
\end{definition}

In particular, considering $r(z)=(1-z)^{-1}$ and $r(z)=(1+\frac{z}{2})(1-\frac{z}{2})^{-1}$ we recover the implicit Euler and the Crank--Nicolson method from Example \ref{ex:time_discretization_methods}, respectively.
The notions of A-acceptability and A-stability originate from \cite{BrennerThomee1979,Dahlquist1963}. For schemes induced by A-acceptable $r$ such that $r(z)-\e^z = \calO(\abs{z}^{p+1})$ as $z\to 0$ for some $p>0$, it was shown in \cite[Theorem~3]{BrennerThomee1979} that $F$ is a convergent time discretization method of order $p$ on $\dom(A^{p+1})$. If $X$ is a Hilbert space and $A$ is self-adjoint, then $F$ even converges of order $p$ on $\dom(A^p)$; see, e.g., \cite[Theorem~7.1]{Thomee2006}.

The well-known Chernoff product formula (see \cite[Proposition 2.1]{Smolyanow2007}) states that a stable and consistent time discretization method is convergent, and the convergence is uniform on compact time intervals. In the spirit of the Chernoff product formula, stability and consistency of order $\pt$ imply convergence of order $\pt$ on an invariant subspace, provided that the semigroup is exponentially bounded on this space. A proof of this folklore statement can be found in \cite[Proposition 4.12]{isem15}, see also related results in \cite[Theorems 2 and 3]{ManginoRasa2007}. 

\begin{proposition}[quantified Chernoff product formula]
\label{prop:time_convergence}
    Let $\Yt\hra X$ with $\Yt$ dense in $X$ such that $\Yt\subseteq \dom(A)$ and $\Yt$ is invariant under $(T(t))_{t\geq 0}$, and assume that $T$ is exponentially bounded on $\Yt$.
    Let $F\from [0,\infty)\to \calL(X)$ be a stable time discretization method that is consistent of order $\pt>0$ on $\Yt$. Then $F$ is convergent of order $\pt$ on $\Yt$.
\end{proposition}

So far, the time steps $\tau^i$ may depend on the particular value of $t\in [0,\bar{T}]$. In applications, one typically works with one family of $N_k\in \N$ time steps $\tau_k=(\tau_k^i)_{1\le i \le N_k}$ and only considers times generated by the grid of these time steps. To obtain a finer discretization, $N_k$ is increased, which corresponds to larger parameters $k \in \N$.
That is, for $\bar{T}\geq 0$ and $k,N_k\in\N$, a family of \emph{time steps} is a family $\tau_k=(\tau_k^i)_{1\leq i\leq N_k} \in [0,\bar{T}]^{N_k}$ such that $\tau_k^i \to 0$ for all $1\leq i\leq N_k$ and $N_k \to \infty$ as $k \to \infty$ as well as $\sum_{i=1}^{N_k} \tau_k^i = \bar{T}$. Then $\calT_k\ce \{\sum_{i=1}^j \tau_k^i;\; j\in\{0,\ldots, N_k\}\}$ is called the \emph{grid} associated with $\tau_k$. Moreover, for a time discretization method $F\from [0,\infty)\to \calL(X)$ and $t\in\calT_k$, we write $F_k(t)\ce\prod_{i=1}^{N_{t,k}} F(\tau_k^i)$, where $N_{t,k}\ce \max\{j\in\{0,\ldots,N_k\}:\; \sum_{i=1}^j \tau_k^i = t\}$. If $F$ converges of order $\pt>0$ on $\Yt$ then for all $k \in \N$, there exist $C_{\bar{T}}\geq 0$ and $\tau_0=\tau_0(\bar{T})>0$ such that if $\max_{i=1,\ldots,N_k} \tau_k^i\leq \tau_0$ then
    \[
        \max_{t\in \calT_k} \norm{F_k(t)f - T(t)f} \leq C_{\bar{T}} \Big(\max_{i=1,\ldots,N_k} \tau_k^i\Big)^\pt \norm{f}_{\Yt} \quad (f\in \Yt).
    \]

\subsection{Randomness approximation}
\label{subsec:random}

Let $(\Omega,\calF,\P)$ be a probability space, $N \in \N$, and $Z = (Z_0,\ldots,Z_{N-1})\from \Omega\to \R^N$ a random variable.
For $n\in\N_0$, let
\[
    \calP_n^N\ce \Bigg\{p\from \R^N\to\K; \; \exists (c_\alpha)_{|\alpha|\le n} \subseteq \K:\,p(x) = \sum_{\alpha\in\N_0^N, \abs{\alpha}\leq n} c_\alpha x^\alpha\quad(x\in\R^N)\Bigg\}
\]
be the vector space of polynomials in $N$ variables of order at most $n$. Note that $\dim \calP_n^N = {\binom{n+N}{n}}$. Moreover, let $\calP^N\ce \bigcup_{n\in\N_0}\calP_n^N$
be the vector space of polynomials. We will assume that $\calP^N\subseteq L_2(\R^N,\P_Z)$ is dense, i.e., $\P_Z$ has moments of all orders. Then $(\calP_n^N)_{n\in\N_0}$ is an approximating sequence of $L_2(\R^N,\P_Z)$.  

\begin{definition}
    For $n\in\N_0$, let $R_n\from L_2(\R^N,\P_Z)\to \calP_n^N\subseteq L_2(\R^N,\P_Z)$ be the orthogonal projection. Then $R_n f$ is called the \emph{polynomial chaos approximation} or \emph{PC approximation} of $f\in L_2(\R^N,\PZ)$ of \emph{order} $n$.
\end{definition}

\begin{remark}
    Let $f\in L_2(\R^N,\P_Z)$ and $n\in\N_0$. Then $R_n f$ is the best approximation of $f$ in $\calP_n^N$. Therefore, $\norm{f-R_nf}_{L_2(\R^N,\PZ)} = \dist(f, \calP_n^N)\to 0$ as $n\to\infty$.
\end{remark}

Let $(Z_0,\ldots,Z_{N-1})$ be independent. Then 
\[
    L_2(\R^N,\P_Z) = \bigotimes_{j=0}^{N-1} L_2(\R,\P_{Z_j}).
\]
By assumption, $\calP^1$ is dense in $L_2(\R,\P_{Z_j})$, so we can construct an orthogonal basis $(\Phi_n^j)_{n \in \calN_j}$ of polynomials $\Phi_n^j \in \calP_{n}^1$ of degree $\deg(\Phi_n^j)=n$. Here, the index set $\calN_j=\{0,\ldots,L_j\}$ for some $L_j\in \N_0$ in the case of a discrete probability measure with finite support and $\calN_j=\N_0$ otherwise. Due to the assumptions we will impose on the distribution of each $Z_j$ in the following, we can restrict our considerations to the case $\calN_j=\N_0$ for all $0 \le j \le N-1$.
Let $\calN\ce\calN_0\times\ldots\times\calN_{N-1}=\N_0^N$. For $\alpha\in\calN$, let $\Phi_\alpha\ce\bigotimes_{j=0}^{N-1} \Phi^j_{\alpha_j}$. Then $(\Phi_\alpha)_{\alpha\in\calN}$ is an orthogonal basis of $L_2(\R^N,\P_{Z})$.
Thus, it suffices to consider the one-dimensional case $N=1$ and then use a tensor product construction.
\begin{remark}
\label{rem:FourierCoeff}
    Let $(\Phi_\alpha)_{\alpha\in\calN}$ be an orthogonal basis of $L_2(\R^N,\P_Z)$ such that $\Phi_\alpha\in \calP_n^N$ for $\abs{\alpha}\leq n$ and $n \in \N_0$. Then $(\Phi_{\alpha})_{\abs{\alpha}\leq n}$ is an orthogonal basis of the closed subspace $\calP_n^N\subseteq L_2(\R^N,\P_Z)$, and hence for $f\in L_2(\R^N,\P_Z)$ and $n\in\N_0$, we have
    \[
        R_n f = \sum_{\abs{\alpha}\leq n} \frac{1}{\norm{\Phi_\alpha}_{L_2(\R^N,\P_Z)}^2} \sp{f}{\Phi_\alpha}_{L_2(\R^N,\P_Z)} \Phi_\alpha = \sum_{\abs{\alpha}\leq n} \hat{f}_\alpha \Phi_\alpha,
    \]
    where $\hat{f}_\alpha\ce \norm{\Phi_\alpha}_{\LZ}^{-2}\sp{f}{\Phi_\alpha}_{\LZ}$ is the generalized Fourier coefficient of $f$ for $\alpha \in \calN$.
\end{remark}

We list the orthogonal polynomials corresponding to common distributions, such as the standard normal distribution.

\begin{example}[Classical orthogonal polynomials]
\label{ex:distributions}
    Let $N=1$.
    \begin{enumerate}
        \item      
            Let $Z$ be standard normally distributed, i.e., $\P_Z(B) = \int_B (2\pi)^{-1/2} \e^{-z^2/2}\rmd z$ for Borel sets $B\seq\R$. 
            The (one-dimensional) \emph{Hermite polynomials} are defined as
            \begin{equation*}
                H_k(z)\ce (-1)^k \e^{\frac{z^2}{2}} \frac{\mathrm{d}^k}{\mathrm{d}z^k} \e^{-\frac{z^2}{2}}, \quad z \in \R, k\in \N_0.
            \end{equation*}
            Then $(H_k)_{k\in\N_0}$ is an orthogonal basis of $L_2(\R,\P_Z)$ and $\spt \P_Z = \R$.
        \item
            Let $\alpha,\beta>-1$ and $Z$ be Beta-distributed with parameters $\alpha$ and $\beta$, i.e.,
            \[
                \P_Z(B) = \int_B \1_{(-1,1)}(z) \frac{\Gamma(\alpha+\beta+2)}{2^{\alpha+\beta+1}\Gamma(\alpha+1)\Gamma(\beta+1)} (1-z)^\alpha (1+z)^\beta \rmd z
            \]
            for Borel sets $B\seq\R$. The (one-dimensional) \emph{Jacobi polynomials} are defined as
            \begin{equation*}
                J_k(z)\ce J_k^{(\alpha,\beta)}(z) \ce \frac{(-1)^k}{2^kk!}(1-z)^{-\alpha}(1+z)^{-\beta}\frac{\mathrm{d}^k}{\mathrm{d}z^k} \left[ (1-z)^{k+\alpha} (1+z)^{k+\beta} \right]
            \end{equation*}
            for $z\in(-1,1)$ and $k\in\N_0$.
            Then $(J_k)_{k\in\N_0}$ is an orthogonal basis of $L_2(\R,\P_Z)$ and $\spt \P_Z = [-1,1]$.
            Note that $\alpha=\beta$ yields \emph{Gegenbauer polynomials}, $\alpha=\beta=0$ yields \emph{Legendre polynomials} corresponding to $Z$ being uniformly distributed on $[-1,1]$, and $\alpha=\beta=-\frac{1}{2}$ yields \emph{Chebyshev polynomials}. 
        \item
            Let $\alpha>-1$ and $Z$ be Gamma-distributed with shape parameter $\alpha+1$ and rate $1$, i.e., $\P_Z(B) = \int_B \1_{(0,\infty)}(z) \Gamma(\alpha+1)^{-1} z^\alpha \e^{-z}\rmd z$ for Borel sets $B\seq\R$. 
            The (one-dimensional) \emph{Laguerre polynomials} are defined as
            \begin{equation*}
                L_k(z)\ce L_k^{(\alpha)}(z)\ce \frac{z^{-\alpha} e^z}{k!}\frac{\mathrm{d}^k}{\mathrm{d}z^k} (z^{k+\alpha}e^{-z}),\quad z>0, k \in \N_0.
            \end{equation*}
            Then $(L_k)_{k\in\N_0}$ is an orthogonal basis of $L_2(\R,\P_Z)$ and $\spt \P_Z = [0,\infty)$.
            Note that the special case $\alpha=0$ yields that $Z$ is exponentially distributed with rate $1$.
    \end{enumerate}
\end{example}

Note that, by \cite{OrthPolLesky1962}, up to affine transformations, these three distributions are exactly the probability distributions whose corresponding orthogonal polynomials yield an orthogonal basis of eigenfunctions of a Sturm--Liouville differential operator (see Example \ref{ex:Qs_1D} for the concrete operators). As these operators are essential for the proof of error estimates for the PC approximation, we assume the following. 

\begin{assumption}
\label{ass:Zdistribution}
    Let $N \in \N$ and $Z=(Z_0,\ldots,Z_{N-1})\from \Omega \to \R^N$ be a random variable. Let $(Z_0, \ldots, Z_{N-1})$ be independent and assume that each of them is either standard normally distributed, Gamma-distributed (with rate $1$), or Beta-distributed, explicitly allowing for different $Z_j$ to have a different distribution. 
\end{assumption}

\begin{definition}
\label{def:convratePCE}
    Let $\Yz \hra L_2(\R^N,\P_Z)$ and $\pz>0$. The polynomial chaos approximation is said to be \emph{convergent of order $\pz$} on $\Yz$ if there exists a constant $\Cz \geq0$ such that 
    \begin{equation*}
       \|R_nf-f\|_{L_2(\R^N,\P_Z)}\leq \Cz n^{-\pz} \|f\|_{\Yz}
    \end{equation*}
    for all $f \in \Yz$ and $n \in \N_0$.
\end{definition}

In order to obtain convergence orders, we need sufficient regularity of the corresponding function. 
For $\ell\in\N_0$, we recall the Sobolev spaces $H^\ell(\R^N,\P_{Z})$ consisting of all $f \in L_2(\R^N,\P_Z)$ such that the distributional derivatives up to order $\ell$ belong to $L_2(\R^N,\P_Z)$ as well. 
We recall the one-dimensional convergence orders first and write $\mathbf{z}$ for the function $z\mapsto z$.
\begin{proposition}[{see \cite[Theorems~6.2.4, 6.2.5, 6.2.6]{Funaro1992}}]
\label{prop:PCEconvergence1D}
    Let $N=1$ and $\ell\in \N_0$.
    \begin{enumerate}
        \item  
            Let $Z$ be standard normally distributed. Then there exists $C\geq 0$ such that for all $f\in H^{2\ell}(\R,\P_Z)$ and $n > 2\ell$,
            \[
            \norm{f-R_n f}_{L_2(\R,\P_Z)} \leq C n^{-\ell} \norm{\frac{\mathrm{d}^{2\ell}}{\mathrm{d}z^{2\ell}} f}_{L_2(\R,\P_Z)}.
            \]
        \item  
            Let $Z$ be Beta-distributed. Then there exists $C\geq 0$ such that for all $f\in H^\ell(\R,\P_Z)$ and $n>\ell$,
            \begin{align*}
                \norm{f-R_n f}_{L_2(\R,\P_Z)} &\leq C n^{-\ell} \norm{(1-\mathbf{z}^2)^{\ell/2}\frac{\mathrm{d}^{\ell}}{\mathrm{d}z^{\ell}} f}_{L_2(\R,\P_Z)} 
                \leq C n^{-\ell} \norm{\frac{\mathrm{d}^{\ell}}{\mathrm{d}z^{\ell}} f}_{L_2(\R,\P_Z)}.
            \end{align*}  
        \item  
            Let $Z$ be Gamma-distributed (with rate 1). Then there exists $C\geq 0$ such that
            \[\norm{f-R_n f}_{L_2(\R,\P_Z)} \leq C n^{-\ell} \norm{\mathbf{z}^{\ell}\frac{\mathrm{d}^{2\ell}}{\mathrm{d}z^{2\ell}} f}_{L_2(\R,\P_Z)}\]
            for all $f\in L_2(\R,\P_Z)$ being weakly differentiable up to order $2\ell$ such that $\mathbf{z}^{k/2} \frac{\mathrm{d}^{k}}{\mathrm{d}z^{k}} f \in L_2(\R,\P_Z)$ for all $k\in\{0,\ldots,2\ell\}$, and $n>2\ell$.        
    \end{enumerate}
\end{proposition}

\begin{remark}
\label{rem:PCEconvergence1D}
    Note that the statements in Proposition \ref{prop:PCEconvergence1D} can be formulated for $n\in\N_0$ as well by adjusting the constant $C$, if necessary.
\end{remark}

We can now prove an error estimate yielding an order of convergence for the polynomial chaos approximation.
Denote by $J_{\mathrm{Normal}}, J_{\mathrm{Gamma}}, J_{\mathrm{Beta}}$ the set of all those $j\in \{0,\ldots,N-1\}$ such that $Z_j$ is standard normally distributed, Gamma-distributed, or Beta-distributed, respectively.
Motivated by the norms appearing in the error estimates of Proposition \ref{prop:PCEconvergence1D}, let $\rho\from\R^N\to\R^N$ be defined by
\begin{equation}
\label{eq:defRho}
    \rho(z)_j\ce \begin{cases}
        1 & j\in J_{\mathrm{Normal}}\cup J_{\mathrm{Beta}},\\
        z_j & j\in J_{\mathrm{Gamma}}.
    \end{cases}
\end{equation}
For $\ell\in \N_0$, we define $H^{\ell}_\rho(\R^N,\P_Z)$ to be the set of all $f \in L_2(\R^N, \P_Z)$ such that $f$ is weakly differentiable up to order $\ell$ and $\rho(\cdot)^{\alpha/2} \partial^\alpha f \in L_2(\R^N,\P_Z)$ for all $\alpha \in \calN$ with $\abs{\alpha}\leq \ell$. Here, $\rho(z)^{\alpha/2}=\prod_{j=0}^{N-1} \rho(z)_j^{\alpha_j/2}\in \R$ for all $z \in \R^N$, $\alpha \in \calN$. We equip $H^{\ell}_\rho(\R^N,\P_Z)$ with the norm $\norm{\cdot}_{H^{\ell}_\rho(\R^N,\P_Z)}$ given by
\[
    \norm{f}_{H^{\ell}_\rho(\R^N,\P_Z)}^2 \ce \sum_{\abs{\alpha}\leq \ell} \big\|\rho(\cdot)^{\alpha/2} \partial^\alpha f\big\|_{\LZ}^2,
\]
which makes $H^{\ell}_\rho(\R^N,\P_Z)$ a Hilbert space. Here and in the following, when summing over $|\alpha|\le \ell$, the index $\alpha$ is assumed to be in $\calN=\N_0^N$. The weighted Sobolev spaces are those spaces $\Yz$ from Definition \ref{def:convratePCE} on which the PC approximation converges of a certain order, as stated in the following main result of this subsection for the scalar-valued case. It extends \cite[Thm.~3.6]{Xiu2010} to the multi-dimensional case (i.e., from $\R$ to $\R^N$) with explicit constants and additionally covers Beta- and Gamma-distribution.

\begin{theorem}
\label{thm:PCEbound}
    Suppose that $Z$ satisfies Assumption \ref{ass:Zdistribution}. Let $\ell \in \N_0$.
    Then 
    \begin{equation*}
        \|R_nf-f\|_{L_2(\R^N, \P_Z)} \leq C_{\ell,N} n^{-\ell} \|f\|_{H^{2\ell}_\rho(\R^N,\P_Z)}
    \end{equation*}
    for all $f \in H^{2\ell}_\rho(\R^N,\P_Z)$ and $n\in\N_0$ with constant
    \begin{align*}
        C_{\ell,N} \ce N^{\ell/2} \prod_{r=0}^{\ell-1} \Big(\max_{0 \le j \le N-1} C_j(2r)\Big),
    \end{align*}
    where $C_j(2r)$ is as defined in Proposition \ref{prop:sobnormQfvsf}. In particular, the polynomial chaos approximation is convergent of order $\ell$ on $H^{2\ell}_\rho(\R^N,\P_Z)$ given that $\ell >0$.
\end{theorem}

For the proof of Theorem \ref{thm:PCEbound} we first need a lemma.
Let $Q_j$ be the Sturm--Liouville differential operator associated with the orthogonal polynomials corresponding to $\P_{Z_j}$ for $j\in\{0,\ldots,N-1\}$ and denote its $k$-th eigenvalue by $\lambda_k^j$ for $k\in \N_0$; cf. Example \ref{ex:Qs_1D}. In higher dimension, we consider $Q = \sum_{j=0}^{N-1} \widetilde{Q}_j$ with $\widetilde{Q}_j\ce I\otimes \ldots \otimes I \otimes Q_j \otimes I \otimes \ldots \otimes I$, where $Q_j$ acts on the $j$-th component.

\begin{lemma}
\label{lem:Fouriercoeffltimes}
    Suppose that $Z$ satisfies Assumption \ref{ass:Zdistribution}. Let $\ell\in\N_0$, $\alpha \in \calN$, and $f \in H^{2\ell}_\rho(\R^N,\P_Z)$. Then 
    \begin{equation*}
        \sp{f}{\Phi_\alpha}_{L_2(\R^N, \PP_Z)} = \bigg(\sumj \lambda_{\alpha_j}^j\bigg)^{-\ell} \sp{ Q^\ell f}{\Phi_\alpha}_{L_2(\R^N, \PP_Z)}.
    \end{equation*}
\end{lemma}

\begin{proof}
    We have $\Phi_\alpha = \otimes_{j=0}^{N-1} \Phi^j_{\alpha_j}$ with the orthogonal polynomial $\Phi^j_{\alpha_j}$ of degree $\alpha_j$ to $\P_{Z_j}$ for $j\in\{0,\ldots,N-1\}$. Then $Q_j\Phi^j_{\alpha_j} = \lambda^j_{\alpha_j}\Phi^j_{\alpha_j}$ for $j\in\{0,\ldots,N-1\}$.
    Thus, for $\ell \in \N_0$ we observe
    \[
        Q\Phi_\alpha = \sum_{j=0}^{N-1} \widetilde{Q}_j \Phi_\alpha = \sum_{j=0}^{N-1} \lambda_{\alpha_j}^j\Phi_\alpha \quad \Rightarrow \quad Q^\ell\Phi_\alpha = \bigg(\sum_{j=0}^{N-1} \lambda_{\alpha_j}^j\bigg)^\ell \Phi_\alpha.
    \]
    Since $f\in H^{2\ell}_\rho(\R^N,\P_Z)$ yields $Q^\ell f\in L_2(\R^N,\P_Z)$ (see Proposition \ref{prop:sobnormQfvsf}) and $Q$ is symmetric in $L_2(\R^N,\P_Z)$, we deduce that
    \begin{align*}
        \sp{f}{\Phi_\alpha}_{L_2(\R^N,\P_Z)} & = \bigg(\sum_{j=0}^{N-1} \lambda_{\alpha_j}^j\bigg)^{-\ell}\sp{f}{Q^\ell\Phi_\alpha}_{L_2(\R^N,\P_Z)} 
        = \bigg(\sum_{j=0}^{N-1} \lambda_{\alpha_j}^j\bigg)^{-\ell}\sp{Q^\ell f}{\Phi_\alpha}_{L_2(\R^N,\P_Z)}. \qedhere
    \end{align*}
\end{proof}

\begin{proof}[Proof of Theorem \ref{thm:PCEbound}]
    Let $\ell \in \N_0$ and $f \in H^{2\ell}_\rho(\R^N,\P_Z)$. 
    By Parseval's identity, Remark \ref{rem:FourierCoeff}, and Lemma \ref{lem:Fouriercoeffltimes}, the approximation error is given by 
    \begin{align*}
         \|f-R_nf\|_\LZ^2 &= \sum_{|\alpha|\ge n+1} \hat{f}_\alpha^2  \|\Phi_\alpha \|_\LZ^2 = \sum_{|\alpha|\ge n+1} \frac{\sp{ f}{\Phi_\alpha}_\LZ^2}{\|\Phi_\alpha \|_\LZ^2} \\
        &= \sum_{|\alpha|\ge n+1} \bigg(\sumj \lambda_{\alpha_j}^j\bigg)^{-2\ell}\frac{\sp{ Q^\ell f}{\Phi_\alpha}_\LZ^2}{\|\Phi_\alpha \|_\LZ^2} .
    \end{align*}
    Further define the lower eigenvalue sum bound 
    \begin{equation*}
        d(n)\ce\min_{|\alpha|\ge n} \sum_{j=0}^{N-1} \lambda_{\alpha_j}^j \quad (n \in \N_0).
    \end{equation*}    
    Then 
    \begin{align*}
        \|f-R_nf\|_\LZ^2    
        &\le d(n+1)^{-2\ell}\sum_{\alpha \in \calN} \frac{1}{\|\Phi_\alpha \|_\LZ^2} \sp{ Q^\ell f}{\Phi_\alpha}_\LZ^2\\
        &= d(n+1)^{-2\ell}\sum_{\alpha \in \calN} \big(  \widehat{(Q^\ell f)}_\alpha\big)^2 \|\Phi_\alpha \|_\LZ^2
        = d(n+1)^{-2\ell} \big\|Q^\ell f\big\|_\LZ^2. 
    \end{align*}
    A repeated application of Proposition \ref{prop:sobnormQfvsf} results in
    \begin{align*}
        \|f-R_nf\|_\LZ
        &\le C_{\ell,N} d(n+1)^{-\ell} \|f\|_{H_\rho^{2\ell}(\R^N, \PP_Z)}.
    \end{align*}
    The assertion follows from $d(n+1)\geq n$, which is immediate from the eigenvalues stated in Example \ref{ex:Qs_1D}.
\end{proof}

\begin{remark}
    \begin{enumerate}
        \item We remark that we can slightly improve the result in Theorem \ref{thm:PCEbound} by making use of the lower eigenvalue sum bound as in the proof.
        \item Note that for the Beta-distributed components of $(Z_0,\ldots,Z_{N-1})$, for the corresponding coordinates, it suffices to have weak derivatives up to order $\ell$ only; see also Proposition \ref{prop:PCEconvergence1D}.
    \end{enumerate}
\end{remark}

The error bound for the polynomial chaos approximation can easily be extended from the scalar- to the vector-valued case. This extension is required for the joint convergence rate in space, time, and randomness of Theorem \ref{thm:jointrate_symmetric} in Section \ref{sec:jointRate}. Note that for a Hilbert space $H$, $L_2(\R^N,\PP_Z;H) = L_2(\R^N,\PP_Z)\otimes H$.

\begin{definition}
\label{def:PCvector}
    Let $H$ be a separable Hilbert space with an orthonormal basis $\calB$ and let $n \in \N_0$. Then the \emph{(vector-valued) polynomial chaos approximation} $\bfR_n\bff$ of $\bff = \sum_{\varphi \in \calB} f_\varphi \varphi \in L_2(\R^N,\PP_Z;H)$ of \emph{order} $n$ is given by 
    \[
        \bfR_n\bff \ce \sum_{\varphi \in \calB} (R_n f_\varphi) \varphi \in \calP_n^N\otimes H.
    \]
\end{definition}

\begin{remark}
Note that $\bfR_n$ is well-defined. Indeed, let $\calB$ and $\calC$ be orthonormal bases of $H$, $n\in\N_0$, and $\bff=\sum_{\varphi \in \calB} f_\varphi \varphi = \sum_{\psi \in \calC} f_\psi \psi \in L_2(\R^N,\PP_Z;H)$. Then
$f_\psi = \sp{\bff(\cdot)}{\psi}_H$ with the overloaded notation $\sp{\bff(\cdot)}{\psi}_H \ce [z \mapsto \sp{\bff(z)}{\psi}_H]\in \LZ$ and $\psi = \sum_{\varphi\in\calB} \sp{\psi}{\varphi}_H \varphi$ for all $\psi\in \calC$, and therefore
\begin{align*}
    R_n f_\psi & = R_n \sp{\bff(\cdot)}{\psi}_H 
    = R_n \Big(\bff(\cdot)\Big\vert\sum_{\varphi\in\calB} \sp{\psi}{\varphi}_H \varphi\Big)_H \\
    & = R_n \sum_{\varphi\in\calB} \sp{\bff(\cdot)}{\varphi}_H \overline{\sp{\psi}{\varphi}_H}  
    = \sum_{\varphi\in\calB} R_n \sp{\bff(\cdot)}{\varphi}_H \sp{\varphi}{\psi}_H
\end{align*}
by continuity of $R_n$. Since also $f_\varphi = \sp{\bff(\cdot)}{\varphi}_H$ and $\varphi = \sum_{\psi\in\calC} \sp{\varphi}{\psi}_H \psi$ for all $\varphi\in \calB$, we conclude well-posedness from
\begin{align*}
    \sum_{\psi\in\calC} (R_n f_\psi) \psi & = \sum_{\psi\in\calC} \sum_{\varphi\in\calB} R_n \sp{\bff(\cdot)}{\varphi}_H \sp{\varphi}{\psi}_H \psi 
    = \sum_{\varphi\in\calB} R_n \sp{\bff(\cdot)}{\varphi}_H \varphi 
    = \sum_{\varphi\in\calB} (R_n f_\varphi) \varphi.
\end{align*}
\end{remark}

\begin{remark}
\label{rem:PCE_subspace}
    Let $H$ be a separable Hilbert space, $Y\subseteq H$ a subspace, and $\bff\in L_2(\R^N,\PZ;Y)$. Then
     $\bfR_n \bff\in \calP_n^N \otimes Y$ for all $n\in\N_0$.
\end{remark}

Analogously to the scalar-valued case, we define convergence orders.
\begin{definition}
    Let $H$ be a separable Hilbert space, $\Yz \hra L_2(\R^N,\P_Z;H)$, $\pz>0$. The polynomial chaos approximation is said to be \emph{convergent of order $\pz$} on $\Yz$ if there exists a constant $\Cz \geq0$ such that 
    \begin{equation*}
       \|\bfR_n\bff-\bff\|_{L_2(\R^N,\P_Z;H)}\leq \Cz n^{-\pz} \|\bff\|_{\Yz}
    \end{equation*}
    for all $\bff \in \Yz$ and $n \in \N_0$.
\end{definition}

From the scalar-valued case in Theorem \ref{thm:PCEbound}, we deduce the main result of this subsection, an error estimate for the vector-valued PC approximation.

\begin{corollary}
\label{cor:PCEboundvector}
    Suppose that $Z$ satisfies Assumption \ref{ass:Zdistribution}. 
    Let $\ell \in \N_0$. Then there exists $C=C(\ell,N)\geq 0$ such that
    for all separable Hilbert spaces $H$, we have
    \begin{equation*}
        \|\bfR_n\bff-\bff\|_{L_2(\R^N, \P_Z; H)} \leq Cn^{-\ell} \|\bff\|_{H^{2\ell}_\rho(\R^N,\P_Z;H)}
    \end{equation*}
    for all $\bff \in H^{2\ell}_\rho(\R^N,\P_Z;H)$ and $n\in\N_0$.
    In particular, the polynomial chaos approximation is convergent of order $\ell$ on $H^{2\ell}_\rho(\R^N,\P_Z;H)$ if $\ell>0$.
\end{corollary}

\begin{proof}
    Let $H$ be a separable Hilbert space and $\calB$ an orthonormal basis of $H$. Let $\bff\in H^{2\ell}_\rho(\R^N,\P_Z;H)$. Then $\bff = \sum_{\psi \in \calB} \sp{\bff(\cdot)}{\psi}_H \psi$ and thus 
    \begin{align*}
        (\bfR_n\bff)(z) = \sum_{\psi \in \calB} (R_n f_\psi)(z) \psi
        = \sum_{\psi \in \calB} \bigl(R_n \sp{\bff(\cdot)}{\psi}_H(z)\bigr)\psi
        =  \big(R_n\sp{\bff(\cdot)}{\varphi}_H\big)(z) 
    \end{align*}
    for $\PZ$-a.e.\ $z \in \R^N$. Parseval's identity, the monotone convergence theorem, and Theorem \ref{thm:PCEbound}  then imply
    \begin{align*}
        \|\bfR_n\bff-\bff\|_{L_2(\R^N,\PP_Z; H)}^2 &= \sum_{\varphi \in \calB} \int_{\R^N} \big|\sp{(\bfR_n\bff-\bff)(z)}{\varphi}_H\big|^2\dPZ(z)\\
        &= \sum_{\varphi \in \calB} \int_{\R^N} \big| (R_n\sp{\bff(\cdot)}{\varphi}_H)(z) - \sp{\bff(z)}{\varphi}_H \big|^2 \dPZ(z)\\
        &= \sum_{\varphi \in \calB} \| R_n\sp{\bff(\cdot)}{\varphi}_H - \sp{\bff(\cdot)}{\varphi}_H \|_{L_2(\R^N,\PP_Z)}^2\\
        &\le C_{\ell,N}^2 n^{-2\ell} \sum_{\varphi \in \calB} \| \sp{\bff(\cdot)}{\varphi}_H \|_{H^{2\ell}_\rho(\R^N,\P_Z)}^2 
        = C_{\ell,N}^2 n^{-2\ell} \|\bff\|_{H^{2\ell}_\rho(\R^N,\P_Z;H)}^2. \qedhere
    \end{align*}
\end{proof}

\section{Deterministic Approximation Results}
\label{sec:jointRate_deterministic}

Let $V,H$ be separable Hilbert spaces with $V\hra H$ dense. Let $\form\from V\times V\to \K$ be a bounded and coercive form, i.e., there exist $K\geq 0$, $\kappa>0$ such that
\[
    \abs{\form(u,v)} \leq K\norm{u}_V\norm{v}_V\quad \text{and}\quad\Re \form(u) \geq \kappa \norm{u}_V^2
\]
for all $u,v\in V$. Let $A$ be the operator associated with $\form$ in $H$ and $(T(t))_{t\geq0}$ the contractive $C_0$-semigroup generated by $-A$. Denote the graph norm in $H$ by $\|\cdot\|_A$. Let $u_0\in H$.
Consider the abstract Cauchy problem
\[u'(t) = -A u(t) \quad(t>0),\quad u(0) = u_0.\]
Its mild solution is given by
\[u\from [0,\infty)\to H, \quad u(t) \ce T(t)u_0\quad(t\geq 0).\]

Let $(V_m)_{m \in \N}$ be an approximating sequence of $V$ and denote the closure of $V_m$ in $\|\cdot\|_H$ by $H_m$.
The spatial approximation of $a$ as discussed in Subsection \ref{subsec:spatial} yields approximating forms $(\form_{m})_{m \in \N}$ associated with operators $(A_{m})_{m \in \N}$. Moreover, for $m\in\N$ let $(T_{m}(t))_{t\geq 0}$ be the $C_0$-semigroup generated by $-A_{m}$.

For $m\in\N$, we recall the $H$-orthogonal projection $P_m\from H \to H_m \seq H$ and the corresponding embedding $J_m\from H_m\to H$ from Subsection \ref{subsec:spatial}. This gives rise to approximate solutions 
\[
    u_m\from [0,\infty)\to H_m,\quad u_m(t) \ce T_m(t)P_mu_0 \quad(t\geq 0).
\]

For $m \in \N$, let $F_m\from [0, \infty) \to \calL(H_m)$ be a time discretization method for $(T_m(t))_{t \ge 0}$ as in Subsection \ref{subsec:temporal}. Fix a final time $\bar{T} \ge 0$. For $k, N_k \in \N$ let $\tau_k=(\tau_k^i)_{1 \le i \le N_k} \in [0,\bar{T}]^{N_k}$ be time steps with associated grid $\calT_k$, $F_{m,k}(t) \ce \prod_{i=1}^{N_{t,k}} F_m(\tau_k^i)$, and  $N_{t,k}$ as defined in Subsection \ref{subsec:temporal}. Temporal approximation with $F_{m,k}$ then yields the approximate solution 
 \[
    u_{m,k}\from \calT_k \to H_m,\quad u_{m,k}(t) \ce F_{m,k}(t) P_m u_0 \quad(t\in \calT_k).
\]

The following theorem is a version of \cite[Theorem 7.8]{Thomee2006}.

\begin{theorem}
\label{thm:jointrate_deterministic_symmetric}
     Let $\form$ be symmetric, bounded, and coercive, and assume that the embedding $V\hra H$ is compact.
     Let $s>0$ and assume that the space discretization converges with order $\px>0$ on $\dom(A^s)$ for the stationary problem.
     Let the time discretization methods $F_m$ be induced by an A-stable rational function $r$ for all $m\in\N$ and let $\pt>0$ be the order of convergence of the time discretizations.\\
    Then for all $\bar{T}>0$, there exist $C_{\bar{T}}\geq 0$ and $\tau_0=\tau_0(\bar{T})>0$ such that for $\max_{i=1,\ldots,N_k} \tau_k^i \leq \tau_0$, we have
    \begin{align*}
        \|J_m u_{m,k}(t) - u(t)\|_{H}
        & \le C_{\bar{T}} \bigg(m^{-\px} + \Big(\max_{i=1,\ldots,N_k} \tau_k^i\Big)^{\pt}\bigg) \|u_0\|_{A^{\max\{s+1,\pt\}}}
    \end{align*}
     for all $t \in \calT_k$, $u_0 \in \dom(A^{\max\{s+1,\pt\}})$, and $m,k\in\N$.
\end{theorem}
Note that the single convergence rates for the discretization in space and time, respectively, appear exactly in the joint rate given sufficient regularity.
Since A-stability is required, the theorem does not cover time discretization methods for which a CFL-type condition arises.

\section{Approximation of Random Evolution Equations}
\label{sec:jointRate}

In view of Section \ref{sec:RandomEvolutionEquations}, we consider the random evolution equation 
\begin{equation}\label{eq:randomACP}
    \bfu'(t) = -\bfA \bfu(t) \quad(t>0),\quad \bfu(0) = \bfu_0 \in \bfH
\end{equation}
with the operator $\bfA$ associated with a random family of operators $(A_z)_{z \in \R^N}$, $N \in \N$. We are interested in the approximation of \eqref{eq:randomACP}, particularly its approximation in randomness. First, we introduce the setting and the standing assumptions.\\
Throughout, let $(\Omega,\calF,\P)$ be a probability space and $Z\from\Omega\to\R^N$ a random variable with independent components, whose distributions are either standard normal, Gamma (with rate 1), or Beta; cf. Assumption \ref{ass:Zdistribution}. Furthermore, let $V,H$ be separable Hilbert spaces, $V\hra H$ densely. For $z\in\R^N$ let $\form_z\from V\times V\to \K$ be a form and assume that $\R^N\ni z\mapsto \form_z(u)$ is measurable for all $u\in V$. 

\begin{assumption}
\label{ass:formUnifBdCoercive}
    Let Assumption \ref{ass:Zdistribution} hold for $Z$. Assume that $(\form_z)_{z\in\R^N}$ is $\P_Z$-almost surely uniformly bounded and $\P_Z$-almost surely uniformly coercive, i.e., 
    there exist a $\P_Z$-null set $\calN_Z\subseteq \R^N$ and $K,\kappa>0$ such that
    \[
        \abs{\form_z(u,v)} \leq K\norm{u}_V\norm{v}_V,\quad \Re \form_z(u) \geq \kappa \norm{u}_V^2
    \]
    for all $u,v\in V$ and $z\in \R^N\setminus \calN_Z$. Further, suppose that $(\form_z)_{z\in\R^N}$ is $\P_Z$-almost surely symmetric.
\end{assumption}

\begin{remark}
    We remark that Assumption \ref{ass:formUnifBdCoercive} can be weakened in the sense that $K$ and $\kappa$ may depend on $z$ such that they are integrable w.r.t.\ $\PZ$ with positive integral, see e.g.\ \cite[Remark 12.14]{Sullivan2015}.
\end{remark}

For $z\in\R^N\setminus \calN_Z$ let $A_z$ in $H$ be the operator associated with $\form_z$, while for $z\in \calN_Z$ we set $A_z=0$. For $z\in\R^N\setminus \calN_Z$ let $(T_z(t))_{t\geq 0}$ be the contractive $C_0$-semigroup generated by $-A_z$.
As introduced in Section \ref{sec:RandomEvolutionEquations}, we consider the operator $\bfA$ in $\bfH= L_2(\R^N,\P_Z;H)$ associated with $(A_z)_{z\in\R^N}$. Then Proposition \ref{prop:multiplication_operator} yields that  $-\bfA$ generates a contractive $C_0$-semigroup $(\bfT(t))_{t\geq 0}$ on $\bfH$.
The mild solution of \eqref{eq:randomACP} is given by
\[
    \bfu\from [0,\infty)\to \bfH, \quad \bfu(t) \ce \bfT(t)\bfu_0\quad(t\geq 0).
\]

\begin{remark}
    Let $\bfV\ce L_2(\R^N,\PZ;V)$. Then it is easy to see that $\bfV\hra \bfH$ densely. Define $\bfform\from \bfV\times\bfV\to \K$ by
    \[\bfform(\bfu,\bfv) \ce \int_{\R^N} \form_z(\bfu(z),\bfv(z))\,\rmd\PZ(z),\]
    noting that $z\mapsto \form_z(\bfu(z),\bfv(z))$ is measurable for all $\bfu,\bfv\in\bfV$.
    Then $\bfform$ is bounded and coercive, hence sectorial, and $\bfA$ is the operator associated with $\bfform$. Symmetry of $\bfform$ even implies self-adjointness of $\bfA$.
\end{remark}

To approximate the random evolution equation \eqref{eq:randomACP}, an approximation in randomness is performed by means of the polynomial chaos expansion $\bfR_n$ of order $n \in \N_0$ (see Definition \ref{def:PCvector}). 

\begin{remark}
\label{rem:straightPCEfails}
    The straightforward approach of applying $\bfR_n$ to \eqref{eq:randomACP} fails, since
    \begin{equation*}
        \bfR_n\bfu'(t) = (\bfR_n\bfu)'(t) = -\bfR_n\bfA \bfu(t) \quad(t>0),\quad 
        \bfR_n\bfu(0) = (\bfR_n \bfu)(0) = \bfR_n\bfu_0 
    \end{equation*}
    cannot be written as an abstract Cauchy problem for $\tilde{\bfu}_n \ce \bfR_n\bfu$. Indeed, in general, $-\bfR_n \bfA \bfu(t) \neq -\bfR_n \bfA \tilde{\bfu}_n(t)=-\bfR_n \bfA \bfR_n \bfu(t)$ for the right-hand side. However, an abstract Cauchy problem is required in order to apply the space and time approximation results from Section \ref{sec:jointRate_deterministic}, requiring a different choice of approximation $\bfu_n \neq \bfR_n \bfu$.
\end{remark}

This section is organized as follows. In Subsection \ref{subsec:PCEapproxPb}, we show how to use PCE correctly to approximate the mild solution of \eqref{eq:randomACP}. This results in a coupled system of deterministic equations. Subsection \ref{subsec:randomSpaceTime} illustrates how a joint convergence rate for the semi-discretization in space and time can be obtained therefrom. It remains to estimate the semi-discretization error in randomness, which is done in Subsection \ref{subsec:semiDiscreteRandom}, composed of auxiliary estimates in Subsections \ref{subsec:pointwiseEst}-\ref{subsec:resolventDiffEst} and the main error estimate in randomness in Theorem \ref{thm:errorunuTrotterKato} in Subsection \ref{subsec:TrotterKato}. The main result on a joint convergence rate for the full discretization in randomness, space, and time is stated in Theorem \ref{thm:jointrate_symmetric} in Subsection \ref{subsec:fullErrorEstimate}.

\subsection{Derivation of the approximate problem}
\label{subsec:PCEapproxPb}

Our aim is to find abstract Cauchy problems approximating \eqref{eq:randomACP} whose solution $\bfu_n$ is contained in $\calP_n^N \ot H \seq \bfH$ and can thus be obtained from a finite number of elements in $H$ depending on the polynomial degree $n \in \N_0$. In other words, we want to find a (potentially coupled) system of deterministic evolution equations from whose solution we can reconstruct the approximate solution $\bfu_n$.
To circumvent the issue discussed in Remark \ref{rem:straightPCEfails}, we start by restricting the form $\bfform$ to polynomial spaces w.r.t.\ the random parameters. Consider the restrictions $\bfform_n \ce \bfform|_{(\calP_n^N\ot V)\times (\calP_n^N\ot V)}$. Since $\bfR_n$ is an $\bfH$-orthogonal projection, $\bfR_n\bfv_n = \bfv_n$ for any $\bfv_n \in \calP_n^N \ot V \seq \calP_n^N \ot H$. It follows that for all $\bfw_n, \bfv_n \in \calP_n^N\ot V$, $n \in \N_0$,
\begin{align*}
    \bfform_n(\bfw_n,\bfv_n) & 
    = \bfform(\bfR_n\bfw_n,\bfR_n\bfv_n) = \sp{\bfA\bfR_n\bfw_n}{\bfR_n\bfv_n}_\bfH  
    = \sp{\bfR_n\bfA\bfR_n\bfw_n}{\bfv_n}_\bfH = \sp{\bfR_n\bfA\bfR_n\bfw_n}{\bfv_n}_{\calP_n^N\ot H}.
\end{align*}
Hence, the restricted form $\bfform_n$ is associated with $\bfR_n\bfA\bfR_n$. (To be precise,  $\bfR_n\bfA\bfR_n$ is the operator associated with $\bfform_n$ on the closed subspace $\calP_n^N\ot V\subseteq \bfV$. We can and will extend $\bfR_n\bfA\bfR_n$ by zero to $\bfH$.)
Thus, we postulate that $\bfu_n$ arises as the solution of the abstract Cauchy problem associated with $\bfR_n\bfA\bfR_n$.

\begin{definition}
    Let $n \in \N_0$. Define $\bfu_n:[0,\infty) \to \calL(\bfH)$ as the mild solution of the abstract Cauchy problem
    \begin{equation}
    \label{eq:defunACP}
        \bfu_n'(t) =-\bfR_n\bfA\bfR_n \bfu_n(t) \quad (t >0),\quad
        \bfu_n(0)=\bfu_{0n} \ce \bfR_n \bfu_0.
    \end{equation}
\end{definition} 
Note that \eqref{eq:defunACP} is not the PC approximation of the original abstract Cauchy problem, since we additionally truncate $\bfu_n$ before applying $\bfA$ on the right-hand side, cf. Remark \ref{rem:straightPCEfails}. Naturally, it remains to verify the well-posedness of \eqref{eq:defunACP}. The mild solution exists uniquely, since $-\bfR_n\bfA\bfR_n$ generates a $C_0$-semigroup. This is a consequence of the following resolvent bound, which will also be needed for future estimates.

\begin{lemma}
\label{lem:RnResolventBounded}
    Suppose Assumption \ref{ass:formUnifBdCoercive} holds.  Then for all $n \in \N_0$ and $\lambda\in\C$ with $\Re \lambda>0$, we have
    \[
        \|\lambda (\lambda+\bfR_n\bfA\bfR_n)^{-1}\|_{\calL(\bfH)} \le 1.
    \]
    In particular, $-\bfR_n\bfA\bfR_n$ generates a contractive $C_0$-semigroup $(\bfT_n(t))_{t \ge 0}$ on $\bfH$ for all $n\in\N_0$.
\end{lemma}
\begin{proof}
    Since $\bfform$ is bounded and coercive as a consequence of Assumption \ref{ass:formUnifBdCoercive}, also the restricted forms $\bfform_n$ are bounded and coercive with uniform parameters. Hence, $(\bfform_n)_{n \in \N_0}$ are uniformly sectorial of angle less than $\frac{\pi}{2}$. Consequently, the negative associated operators $-\bfR_n\bfA\bfR_n$ are (uniformly) m-sectorial, i.e., satisfy the resolvent estimate stated. In particular, they are generators of contractive $C_0$-semigroups.
\end{proof}

Abbreviate the index set by $\calN\ce \N_0^N$ and let $(\Phi_\alpha)_{\alpha\in\calN}$ be an orthonormal basis of $L_2(\R^N, \PZ)$ such that $\Phi_\alpha\in \calP_n^N$ for $\abs{\alpha}\leq n$ and $n\in\N_0$. In contrast to Subsection \ref{subsec:random}, we consider $\Phi_\alpha$ of unit norm in the following.
It follows from \eqref{eq:defunACP} that $\bfu_n(t) \in \calP_n^N \otimes H$, since $\bfu_n'(t) \in \calP_n^N \otimes H$. Hence, $\bfu_n(t)$ can be completely described by $d_n \ce \sharp \{\alpha\in\calN:\; \abs{\alpha}\leq n\} = \binom{n+N}{n}$ many $H$-valued coefficients. That is, $\bfu_n(t) = \sum_{|\beta|\le n} \Phi_\beta \ot u_\beta$ for $u_\beta \in H$ to be determined. We can rewrite the random abstract Cauchy problem \eqref{eq:defunACP} for $\bfu_n(t)$ as a system of $d_n$ deterministic equations, from whose solution $\bfu_n(t)$ can be obtained. To derive this deterministic system, the following simple observation proves helpful. 

\begin{notation}
\label{not:PCECo}
    To abbreviate, for $\bfw\in \bfH$ and $\alpha\in\calN$ let us write
    \[
        \widehat{\bfw}_\alpha\ce \sp{\bfw}{\Phi_\alpha}_{L_2(\R^N,\PZ)}\ce \int_{\R^N} \bfw(z)\Phi_\alpha(z)\,\rmd \PZ(z) \in H
    \]
    with a slight abuse of notation.
    For $u,v \in V$, denote the $\alpha$-th generalized Fourier coefficient of $\form_\cdot(u,v)$ by $\hat{a}_\alpha(u,v)$, i.e., $\hat{a}_\alpha(u,v) \ce \sp{\form_\cdot(u,v)}{\Phi_\alpha}_{L_2(\R^N,\PZ)}$ for $\alpha \in \calN$. Further, for $\alpha, \beta, \gamma \in \calN$, denote 
    \[
        \varepsilon_{\alpha,\beta,\gamma}\ce  \sp{\Phi_\alpha}{\Phi_\beta\Phi_\gamma}_{L_2(\R^N,\PZ)} = \int_{\R^N} \Phi_\alpha\Phi_\beta\Phi_\gamma\,\rmd \PZ.
    \]
\end{notation}

\begin{remark}
\label{rem:eps0}
    Note that $\varepsilon_{\alpha,\beta,\gamma}=0$ for all $\alpha \in \calN$ with $|\alpha|>|\beta|+|\gamma|$. Indeed, since $(\Phi_\alpha)_{\alpha \in \calN}$ is an orthonormal basis, $\Phi_\alpha \perp \calP_\ell^N$ for all $\ell<|\alpha|$. The statement then follows from $\Phi_\beta \Phi_\gamma \in \calP_{|\beta|+|\gamma|}^N$. Moreover, $\Phi_\alpha$ is $\R$-valued (see Subsection \ref{subsec:random}), which is why the complex conjugates can be omitted in the definition of $\varepsilon_{\alpha,\beta,\gamma}$ and the indices can be interchanged.
\end{remark}

Let $\bfw_n = \sum_{|\beta| \le n} \Phi_\beta \ot u_\beta, \bfv_n = \sum_{|\gamma|\le n} \Phi_\gamma \ot v_\gamma \in \calP_n^N \ot V$ for $n \in \N_0$. Then a PCE of $\form_\cdot(u_\beta,v_\gamma)$ combined with Remark \ref{rem:eps0} results in
\begin{align}
\label{eq:bfformnFinite}
    \bfform_n(\bfw_n,\bfv_n) &= \bfform(\bfw_n,\bfv_n) = \sum_{|\beta|\le n} \sum_{|\gamma|\le n} \int_{\R^N} \form_z( u_\beta, v_\gamma) (\Phi_\beta \Phi_\gamma)(z)\dPZ(z)\nonumber\\
    &= \sum_{|\beta|\le n} \sum_{|\gamma|\le n} \int_{\R^N} \left(\sum_{\alpha \in \calN} \hat{\form}_\alpha( u_\beta, v_\gamma) \Phi_\alpha(z)\right) (\Phi_\beta \Phi_\gamma)(z)\dPZ(z)\nonumber\\
    &\overset{(*)}{=} \sum_{|\beta|\le n} \sum_{|\gamma|\le n} \sum_{\alpha \in \calN} \hat{\form}_\alpha( u_\beta, v_\gamma) \varepsilon_{\alpha,\beta,\gamma}
    = \sum_{|\beta|\le n} \sum_{|\gamma|\le n} \sum_{|\alpha| \le 2n} \hat{\form}_\alpha( u_\beta, v_\gamma) \varepsilon_{\alpha,\beta,\gamma},
\end{align}
where in $(*)$ we were allowed to interchange integration and summation, since the polynomial chaos expansion converges in $L_2(\R^N,\PZ)$ and thus also in $L_1(\R^N,\PZ)$. 
Motivated by the above observation that $\bfform_n$ can be represented as a finite linear combination, we define the following deterministic forms.

\begin{definition}
    For $n\in\N_0$ and $\beta,\gamma\in\calN$ we define $\form_{n,\beta,\gamma}\from V\times V\to\K$ by
    \[
        \form_{n,\beta,\gamma}\ce  \sum_{\abs{\alpha}\leq n} \varepsilon_{\alpha,\beta,\gamma} \widehat{\form}_\alpha.
    \]
    Let $\frakV_n\ce V^{d_n}$ be equipped with the norm $\|(u_\beta)_{|\beta \le n}\|_{\frakV_n}^2 \ce \sum_{|\beta| \le n} \|u_\beta\|_V^2$ and define $\frakform_n\from \frakV_n\times \frakV_n\to\K$ by
    \[
        \frakform_n((u_\beta)_{\abs{\beta}\leq n}, (v_\gamma)_{\abs{\gamma}\leq n}) \ce  \sum_{\abs{\gamma}\leq n} \sum_{\abs{\beta}\leq n} \form_{2n,\beta,\gamma}(u_\beta,v_\gamma).
    \]
\end{definition}

The corresponding quadratic form $\frakform_n:\frakV_n \to \K$ is denoted by $\frakform_n$ as well.

\begin{proposition}
\label{prop:frakform_n_bddcoercive}
    Suppose that Assumption \ref{ass:formUnifBdCoercive} holds for some $K\ge \kappa >0$. Let $n\in\N_0$. Then
    \begin{align*}
        \abs{\frakform_n((u_\beta)_{\abs{\beta}\leq n},(v_\gamma)_{\abs{\gamma}\leq n})} & \leq K \norm{(u_\beta)_{\abs{\beta}\leq n}}_{\frakV_n} \norm{(v_\gamma)_{\abs{\gamma}\leq n}}_{\frakV_n},\\
        \Re \frakform_n((u_\beta)_{\abs{\beta}\leq n}) & \geq \kappa \norm{(u_\beta)_{\abs{\beta}\leq n}}_{\frakV_n}^2
    \end{align*}
    for all $(u_\beta)_{\abs{\beta}\leq n}, (v_\gamma)_{\abs{\gamma}\leq n} \in \frakV_n$. In particular, $\frakform_n$ are sectorial of angle less than $\frac{\pi}{2}$.
\end{proposition}

\begin{proof}
    Let $\bfw_n \ce \sum_{|\beta| \le n} \Phi_\beta \ot u_\beta, \bfv_n \ce \sum_{|\gamma|\le n} \Phi_\gamma \ot v_\gamma$. From \eqref{eq:bfformnFinite}, we deduce
    \begin{align}
    \label{eq:frakformIsBfform}
        \frakform_n\bigl((u_\beta)_{\abs{\beta}\leq n},(v_\gamma)_{\abs{\gamma}\leq n}\bigr) = \sum_{|\beta|\le n} \sum_{|\gamma|\le n} \sum_{|\alpha| \le 2n} \hat{\form}_\alpha( u_\beta, v_\gamma) \varepsilon_{\alpha,\beta,\gamma}=\bfform(\bfw_n,\bfv_n).  
    \end{align}
    Boundedness and coercivity of $\bfform$ allow us to estimate $\frakform_n$ in terms of $\|\bfw_n\|_\bfV$ and $\|\bfv_n\|_\bfV$. 
    Parseval's identity yields the desired norm identity
    \begin{align*}
        \|\bfw_n\|_\bfV ^2
        &=  \sum_{\beta \in \calN}\|(\widehat{\bfw_n})_\beta\|_V^2 =  \sum_{\beta \in \calN}\big\|\sp{\bfw_n}{\Phi_\beta}_{\LZ}\big\|_V^2
        = \sum_{|\beta| \le n}\|u_\beta\|_V^2 = \big\|(u_\beta)_{|\beta| \le n} \big\|_{\frakV_n}^2. \qedhere
    \end{align*}
\end{proof}

Note that the above proposition implies sectoriality of $\frakform_n$. 
Having established the connection between the random form $\bfform$ and the deterministic forms $\frakform_n$ acting on the finite-dimensional Cartesian product $\frakV_n$ of the (typically infinite-dimensional) space $V$, the question of associated operators arises. Being a bounded and coercive form, we can associate an operator $\frakA_n$ with $\frakform_n$. To get a representation for this operator, we start by considering the operators associated with the building blocks $a_{n, \beta,\gamma}$ of $\frakform_n$.

\begin{lemma}
\label{lem:form_n}
    Let $n\in\N_0$ and $\beta,\gamma\in\calN$. Suppose that Assumption \ref{ass:formUnifBdCoercive} holds. Then $\form_{n,\beta,\gamma}$ is bounded.
\end{lemma}

\begin{proof}
    By Bessel's inequality, $\widehat{\form}_\alpha$ is bounded for $\alpha\in\calN$. As the sum defining $\form_{n,\beta,\gamma}$ is finite, we obtain the assertion.
\end{proof}

By Lemma \ref{lem:form_n}, we can associate an operator in $H$ to $\form_{n,\beta,\gamma}$ for $n\in\N_0$, $\beta,\gamma\in\calN$. Denote this operator by $A_{n,\beta,\gamma}$. An assumption is needed to ensure that these operators are densely defined.

\begin{assumption}
\label{ass:DAzDconstant}
    Suppose that Assumption \ref{ass:formUnifBdCoercive} holds. Further, assume that there exists $\ols>0$ and a subspace $D^{\ols}\subseteq H$ such that $\dom(A_z^{\ols}) = D^{\ols}$ for $\PZ$-almost every $z\in\R^N$ and the associated graph norms are $\PZ$-almost surely equivalent.
\end{assumption}

\begin{remark}
    Suppose that Assumption \ref{ass:DAzDconstant} holds for some $s>0$, and let $0<s< \ols$. Then, by interpolation theory, we have that $D^s\ce [H,D^{\ols}]_{s/\ols} = \dom(A_z^s)$ for $\PZ$-almost every $z\in\R^N$ and the associated graph norms are $\PZ$-almost surely equivalent.
    If $\ols\geq 1$ we abbreviate $D\ce D^1$.
\end{remark}

\begin{lemma}
\label{lem:Def-Bereich_A_n}
    Suppose that Assumption \ref{ass:DAzDconstant} holds for $\ols=1$. Then $\dom(A_{n,\beta,\gamma}) = D$ for all $n\in\N_0$, $\beta,\gamma\in\calN$.
\end{lemma}

\begin{proof}
    Let $\alpha\in\calN$ and let $\widehat{A}_\alpha$ be the operator in $H$ associated with $\widehat{\form}_\alpha$. We show that $\dom(\widehat{A}_\alpha) = D$. This suffices to yield the assertion by definition of $A_{n,\beta,\gamma}$ as a linear combination.     
     Let $u\in D$ and $v\in V$. Then $1\otimes u\in \dom(\bfA)$ and
    \begin{align*}
        \widehat{\form}_\alpha(u,v) &= \sp{\form_\cdot(u,v)}{\Phi_\alpha}_{L_2(\R^N,\PZ)} = \int_{\R^N} \form_z(1\cdot u,\Phi_\alpha(z) v)\,\rmd\PP_Z(z)=\bfform(1\ot u,\Phi_\alpha \ot v)\\
        & = \sp{\bfA(1\ot u)}{\Phi_\alpha \ot v}_\bfH
        = \big(\sp{\bfA(1\otimes u)}{\Phi_\alpha}_{L_2(\R^N,\PZ)}\big|v\big)_H.
    \end{align*}
    Thus, $u\in \dom(\widehat{A}_\alpha)$, so $D\subseteq \dom(\widehat{A}_\alpha)$.

    Now, let $u\in \dom(\widehat{A}_\alpha)$. Then $u\in V$ and therefore $1\otimes u\in \bfV$.
    Thus, for $\bfv\in \bfV$, a PCE of $\form_z$ yields
    \begin{align*}
    \bfform(1\otimes u,\bfv) &
    = \int_{\R^N} \sum_{\alpha\in\calN} \widehat{\form}_\alpha(u,\bfv(z)) \Phi_\alpha(z)\,\rmd \PZ(z) 
    = \int_{\R^N} \sum_{\alpha\in\calN} \sp{\widehat{A}_\alpha u}{\bfv(z)}_H \Phi_\alpha(z)\,\rmd \PZ(z) \\
    & \overset{(*)}{=} 
    \int_{\R^N} \sp{\sum_{\alpha\in\calN} \Phi_\alpha(z)\widehat{A}_\alpha u}{\bfv(z)}_H \,\rmd \PZ(z) 
    = 
    \sp{\sum_{\alpha\in\calN} \Phi_\alpha \otimes \widehat{A}_\alpha u }{\bfv}_{\bfH}.
    \end{align*}
    Since the polynomial chaos expansion converges in $L_2(\R^N,\PP_Z)$ and thus the integrand converges also in $L_1(\R^N,\PP_Z)$, and taking scalar products is a linear functional, we were allowed to interchange series and scalar product in $(*)$. Thus, $1\otimes u\in \dom(\bfA)$, i.e., $u\in D$.
\end{proof}

This ensures that $\frakform_n$ is associated with a densely defined operator. 

 \begin{definition}
 \label{def:frakAHuon}
    Let $\frakH_n\ce \oplus_{\abs{\beta}\leq n} H = H^{d_n}$ be equipped with the scalar product
    \[
        \sp{(u_\beta)_{|\beta| \le n}}{(v_\gamma)_{|\gamma| \le n}}_{\frakH_n} \ce \sum_{|\beta| \le n} \sp{u_\beta}{v_\gamma}_H
    \]
    and the corresponding induced norm for $n \in \N_0$.
     Define $\frakA_n$ in $\frakH_n$ via
    \begin{align*}
        \dom(\frakA_n) & \coloneqq \bigoplus_{\abs{\beta}\leq n} D = D^{d_n},\quad
        \frakA_{n} ((u_\beta)_{\abs{\beta}\leq n})  \ce  \bigg(\sum_{\abs{\beta}\leq n} A_{2n,\beta,\gamma} u_\beta\bigg)_{\abs{\gamma}\leq n} \in \frakH_n.
    \end{align*}
 \end{definition}

Proposition \ref{prop:frakform_n_bddcoercive} implies sectoriality of $\frakA_n$. Even more is true due to $\PZ$-a.s.\ symmetry of $(a_z)_{z \in \R^N}$.

\begin{proposition}
\label{prop:frakform_n_symmetric}
    Suppose that Assumption \ref{ass:DAzDconstant} holds for $\ols=1$. Then $\frakform_n$ is symmetric, and $\frakA_n$ is self-adjoint.
\end{proposition}

\begin{proof}
    Let $\alpha\in\calN$. It is easy to see that $\widehat{\form}_\alpha$ is symmetric. Thus, $\form_{n,\beta,\gamma}$ is also symmetric. Moreover, $\varepsilon_{\alpha,\beta,\gamma} = \varepsilon_{\alpha,\gamma,\beta}$ for all $\beta,\gamma\in \calN$; cf. Remark \ref{rem:eps0}.
    Hence, $\form_{n,\beta,\gamma} = \form_{n,\gamma,\beta}$. Thus, we easily observe the symmetry of $\frakform_n$, which directly implies the self-adjointness of $\frakA_n$.
\end{proof}

Finally, we can now relate the solution $\bfu_n$ of the random abstract Cauchy problem to the solution of the coupled deterministic system given by $\frakA_n$. We find that the solution vector of the deterministic system consists precisely of the PCE coefficients of order up to $n$ of $\bfu_n$. We recall Notation \ref{not:PCECo}.

\begin{proposition}
\label{prop:systemMildSolPCE}
    Suppose that Assumption \ref{ass:DAzDconstant} holds for $\ols=1$. For $n \in \N_0$ let $\bfu_n$ be the mild solution of \eqref{eq:defunACP}. Define the approximate initial condition    
    \begin{equation}
    \label{eq:deffraku0n}
        \fraku_{0n}\ce\big(\widehat{(\bfu_0)}_{\alpha}\big)_{\abs{\alpha}\leq n} = \big(\sp{\bfu_0}{\Phi_\alpha}_{\LZ}\big)_{\abs{\alpha}\leq n} .
    \end{equation}  
    Then the coupled deterministic system
    \begin{equation}
    \label{eq:frakACP}
        \fraku_n'(t) = -\frakA_n\fraku_n(t) \quad(t>0), \quad \fraku_n(0) = \fraku_{0n}
    \end{equation}
    is solved by $\fraku_{n}:[0,\infty) \to \frakH_n$,
    \[
        \fraku_{n}(t)\ce \big(\sp{\bfu_n(t)}{\Phi_\beta}_{L_2(\R^N,\PP_Z)}\big)_{|\beta|\le n} \in  \frakH_n.
    \]
\end{proposition}
\begin{proof}
    We show the initial condition and $\fraku_n'(t) = -\frakA_n\fraku_n(t)$ componentwise. Let $\gamma \in \calN$ with $|\gamma|\le n$. Since the $\gamma$-th PCE coefficients of $\bfR_n\bfu_0$ and $\bfu_0$ agree for $\abs{\gamma}\le n$,
    \begin{align*}
        (\fraku_n(0))_\gamma &= \sp{\bfu_n(0)}{\Phi_\gamma}_{\LZ} = \widehat{(\bfR_n\bfu_0)}_\gamma= \widehat{(\bfu_0)}_\gamma = (\fraku_{0n})_\gamma.
    \end{align*}
    Abbreviate the PCE coefficients of $\bfu_n(t)$ by $u_\beta \ce (\widehat{\bfu_n(t)})_\beta$ for $\abs{\beta}\le n$.
    Testing with $v \in V$ yields
    \begin{align}
    \label{eq:righthandsideProof}
        \sp{(\frakA_n\fraku_n(t))_\gamma}{v}_H 
        = \sum_{|\beta| \le n} \form_{2n,\beta,\gamma}(u_\beta,v)
        &= \sum_{|\beta| \le n} \sum_{|\alpha|\le 2n} \varepsilon_{\alpha,\beta,\gamma} \hat{a}_\alpha(u_\beta,v) \overset{\eqref{eq:bfformnFinite}}{=} \bfform_n(\bfu_n(t),\Phi_\gamma \ot v)
    \end{align}
    for the negative right-hand side of \eqref{eq:frakACP} by definition of $\fraku_n,\frakA_n, \form_{n,\beta,\gamma}$ and using that $\form_{2n,\beta,\gamma}$ is the form corresponding to $ A_{2n,\beta,\gamma}$.
    For the negative left-hand side, by definition of $\bfu_n$, we rewrite
    \begin{align*}
        (-\fraku_n'(t))_\gamma &= - \frac{\rmd}{\rmd t} \sp{\bfu_n(t)}{\Phi_\gamma}_{\LZ} = \sp{-\bfu_n'(t)}{\Phi_\gamma}_\LZ
        = \sp{\bfR_n\bfA\bfR_n\bfu_n(t)}{\Phi_\gamma}_\LZ.
    \end{align*}
    Testing with $v \in V$, we can use \eqref{eq:bfformnFinite} and that $\bfform_n$ corresponds to $\bfR_n\bfA\bfR_n$ to conclude
    \begin{align*}
        \sp{(-\fraku_n'(t))_\gamma}{v}_H 
        &= \big(\sp{\bfR_n\bfA\bfR_n\bfu_n(t)}{\Phi_\gamma}_\LZ\big|v\big)_H= \sp{\bfR_n\bfA\bfR_n\bfu_n(t)}{\Phi_\gamma \ot v}_\bfH 
        = \bfform_n(\bfu_n(t),\Phi_\gamma \ot v), 
    \end{align*}
    which agrees with \eqref{eq:righthandsideProof} and thus finishes the proof.
\end{proof}

\subsection{Semi-discretization in space and time}
\label{subsec:randomSpaceTime}

Due to its abstract Cauchy problem structure, we can now discretize \eqref{eq:frakACP} in space and time as
in Section \ref{sec:jointRate_deterministic}. 
Let $n\in\N_0$ and $(V_m)_{m \in \N}$ be an approximating sequence of $V$. Then, $(\frakV_{n,m})_{m \in \N}$ with $\frakV_{n,m}\ce V_m^{d_n}$ is an approximating sequence of $\frakV_n$.
The spatial approximation of $\frakform_n$ as discussed in Subsection \ref{subsec:spatial} yields approximating forms $(\frakform_{n,m})_{m \in \N}$ on $\frakV_{n,m}$ and associated operators $(\frakA_{n,m})_{m \in \N}$ on $\frakH_{n,m} \ce  H_m^{d_n}$, where $H_m \ce \overline{V_m}^{_{\|\cdot\|_H}}$. Moreover, for $m\in\N$, let $(\frakT_{n,m}(t))_{t\geq 0}$ be the $C_0$-semigroup generated by $-\frakA_{n,m}$.

For $m\in\N$ we recall the $H$-orthogonal projection $P_m\from H \to H_m \seq H$ and the corresponding embedding $J_m\from H_m\to H$ from Subsection \ref{subsec:spatial}, which can be naturally lifted to $\frakP_{n,m}\from \frakH_n \to \frakH_{n,m} \seq \frakH_{n}$ and $\frakJ_{n,m}\from \frakH_{n,m}\to \frakH_{n}$, as well as
\begin{equation}
\label{eq:defbfJnmPnm}
    \bfP_{n,m}: \calP_n^N \ot H \to \calP_n^N \ot H_m,\quad \bfJ_{n,m}: \calP_n^N \ot H_m \to \calP_n^N \ot H
\end{equation}
for $n \in \N_0$. Then $\bfP_{n,m}(\sum_{|\beta|\le n} \Phi_\beta \ot u_\beta)=\sum_{|\beta|\le n} \Phi_\beta \ot (P_m u_\beta)$.
This yields semi-discretizations
\[\fraku_{n,m}\from [0,\infty)\to \frakH_{n,m},\quad \fraku_{n,m}(t) \ce \frakT_{n,m}(t)\frakP_{n,m}\fraku_{0n} \quad(t\geq 0).\]

For $n\in\N_0$ and $m \in \N$, let $\frakF_{n,m}\from [0, \infty) \to \calL(\frakH_{n,m})$ be a time discretization method for $(\frakT_{n,m}(t))_{t \ge 0}$ as in Subsection \ref{subsec:temporal}. Fix a final time $\bar{T} \ge 0$. For $k \in \N$ and $N_k \in \N$ let $\tau_k=(\tau_k^i)_{1 \le i \le N_k} \in [0,\bar{T}]^{N_k}$ be a family of time steps, $\calT_k$ the associated grid, and $N_{t,k}\ce \max\{j\in\{0,\ldots,N_k\}:\; \sum_{i=1}^j \tau_k^i=t\}$. For $t\in \calT_k$ let $\frakF_{n,m,k}(t) \ce \prod_{i=1}^{N_{t,k}} \frakF_{n,m}(\tau_k^i)$. Temporal approximation with $\frakF_{n,m,k}$ then yields the discretizations 
 \[
    \fraku_{n,m,k}\from \calT_k \to \frakH_{n,m},\quad \fraku_{n,m,k}(t) \ce \frakF_{n,m,k}(t) \frakP_{n,m} \fraku_{0n} \quad(t\in \calT_k).
\]
Following the idea of Proposition \ref{prop:systemMildSolPCE}, we consider the corresponding element in $\bfH_m \ce L_2(\R^N,\PZ;H_m)$ with PCE coefficients $\fraku_{n,m,k}(t)$. That is, we let
\begin{equation}
\label{eq:defbfunmk}
    \bfu_{n,m,k}\from \calT_k \to \bfH_m,\quad \bfu_{n,m,k}(t) \ce  \sum_{\abs{\beta}\leq n} \Phi_\beta \ot (\fraku_{n,m,k}(t))_\beta  \quad(t\in \calT_k).
\end{equation}
Combining the results from Section \ref{sec:jointRate_deterministic} and Subsection \ref{subsec:PCEapproxPb}, we are in a position to estimate the semi-discretization error in space and time. 

\begin{proposition}
\label{prop:convergence_rate_random_symmetricPartI}
    Suppose that Assumption \ref{ass:DAzDconstant} holds for some $\ols\ge 1$. Further, assume that $V\hookrightarrow H$ is compact.
    Suppose that the space discretization converges with order $\px>0$ on $D^s$
    for some $0<s\leq \ols$ for the stationary problem. 
    Let $r$ be an A-stable rational function as in Definition \ref{def:rational_time_discretization}, suppose that the time discretization methods $\frakF_{n,m}$ are induced by $r$ for all $n,m\in\N$, and let $\pt>0$ be the order of convergence of the time discretizations. Let $\bfJ_{n,m}, \bfu_{n,m,k},\bfu_n$, and $\fraku_{0n}$ as in \eqref{eq:defbfJnmPnm}, \eqref{eq:defbfunmk}, \eqref{eq:defunACP}, and \eqref{eq:deffraku0n}, respectively.\\
    Then for all $\bar{T}>0$ there exist $C_{\bar{T}}\geq 0$ and $\tau_0=\tau_0(\bar{T})>0$ such that for $\max_{i=1,\ldots,N_k} \tau_k^i \leq \tau_0$, we have 
    \begin{align*}
        &\|\bfJ_{n,m} \bfu_{n,m,k}(t) - \bfu_n(t)\|_{\bfH} \le C_{\bar{T}} \bigg(m^{-\px} + \Big(\max_{i=1,\ldots,N_k} \tau_k^i\Big)^{\pt}\bigg) \|\fraku_{0n}\|_{\frakA_n^{\max\{s+1,\pt\}}}
    \end{align*}
     for all $t \in \calT_k$, $\bfu_0 \in \bfH$ such that  $\fraku_{0n}\in \dom(\frakA_n^{\max\{s+1,\pt\}})$, $n\in\N_0$, and $m,k\in\N$.
\end{proposition}

\begin{proof}
    Note that by Parseval's identity, we have
    \[\norm{\bfJ_{n,m} \bfu_{n,m,k}(t) - \bfu_n(t)}_{\bfH} = \norm{\frakJ_{n,m} \fraku_{n,m,k}(t) - \fraku_n(t)}_{\frakH_n}.\]
    Since $\frakform_n$ is symmetric by Proposition \ref{prop:frakform_n_symmetric}, we can apply Theorem \ref{thm:jointrate_deterministic_symmetric}, which results in
    \begin{align*}
        \norm{\bfJ_{n,m} \bfu_{n,m,k}(t) - \bfu_n(t)}_{\bfH}
        & \leq C\Big(m^{-\px} + \Big(\max_{i=1,\ldots,N_k} \tau_k^i\Big)^{\pt}\Big) \|\fraku_{0n}\|_{\frakA_n^{\max\{s+1,\pt\}}}. \qedhere
    \end{align*}
\end{proof}

Rather than estimating the error in $\bfH$ in terms of the graph norm of $\fraku_{0n}$, we would prefer to estimate it in terms of a suitable norm of $\bfu_{0n}$ or $\bfu_0$ directly. A further condition is required to do so; this will be discussed in Subsection \ref{subsec:fullErrorEstimate}.

\subsection{Semi-discretization in randomness}
\label{subsec:semiDiscreteRandom}

For an estimate of the full discretization error, it  further remains to estimate the randomness semi-dis\-cretiza\-tion error $\|\bfu_n(t)-\bfu(t)\|_\bfH$ in terms of a suitable norm of $\bfu_0$ with decay of some rate in $n$. Since $\bfu_n(t) \neq \bfR_n \bfu(t)$ (cf. Remark \ref{rem:straightPCEfails}), estimating the semi-discretization error in randomness is more intricate than a simple application of the PCE error estimate of Corollary \ref{cor:PCEboundvector}. 

The estimate is established in several steps, the main proof being presented in Subsection \ref{subsec:TrotterKato} based on technical estimates in the preceeding subsections.
In Subsection \ref{subsec:pointwiseEst}, we present an assumption that implies pointwise in $z$ estimates. With a composition estimate, they can be lifted to Sobolev subspaces of $\bfH$. This lifting leads to bounds for $\bfA$ and its resolvent in suitable norms, as illustrated in Subsection \ref{subsec:mappingPropertiesBfA}. These bounds in turn imply a resolvent difference estimate; cf. Subsection \ref{subsec:resolventDiffEst}. Via a modified Trotter--Kato argument in Subsection \ref{subsec:TrotterKato}, this yields a convergence rate for the corresponding semigroups, i.e. a convergence rate for the semi-discretization in randomness.

Henceforth, $C$, $C_\ell$, $C_{\ell,r}$, etc.\ denote generic constants, whose values can vary from line to line and may depend on the quantities indexed. 
  
\subsubsection{Pointwise (in \texorpdfstring{$z$}{\textit{z}}) estimates}
\label{subsec:pointwiseEst}

We recall the definition of $\rho$ from \eqref{eq:defRho}. In the case where $Z$ has only normally or Beta-distributed components, $\rho$ can be omitted in the following due to $\rho\equiv 1$. It is, however, required in the presence of Gamma-distributed components of $Z$.

\begin{assumption}
\label{ass:coeffCond}
    Suppose that Assumption \ref{ass:DAzDconstant} holds for some $\ols\geq 1$. Assume that for some $\ell \in \N_0$, $[z \mapsto \form_z(u,v)] \in C^\ell(\R^N)$ for all $u,v\in V$ and there exists a constant $C_\ell\ge 0$ such that
\begin{equation*}
    |\rho(z)^{\alpha/2}\partial_z^\alpha a_z(u,v)| \le C_\ell \|u\|_D \|v\|_H 
\end{equation*}
for $\PZ$-almost every $z\in\R^N$ and for all $u \in D$, $v \in V$, and $\alpha \in \calN$ with $|\alpha| \le \ell$.
\end{assumption}

Recall that then $(A_z)_{z \in \R^N}$ are $\PZ$-almost surely uniformly sectorial as a consequence of $\PZ$-almost sure boundedness and coercivity of $(\form_z)_{z \in \R^N}$.

\begin{definition}
    Let $\theta_0 \in (0,\frac{\pi}{2})$ be the $\PZ$-almost surely uniform angle of sectoriality of $(A_z)_{z \in \R^N}$ such that $\sigma(A_z) \seq \Sigma_{\theta_0}$, where $\Sigma_{\theta_0} \ce \{\lambda \in \C: \lambda \neq 0, \abs{\arg(\lambda)}<\theta_0\}$ is the open sector of angle $\theta_0$. Let $\eta_0 \ce \pi-\theta_0$ and $\eta \in (\frac{\pi}{2},\eta_0)$. For $r>0$, define the curve $\gamma_r$ as the union of the three curves $\gamma_r^1,\gamma^2_r$, and $\gamma^3_r$, where $\gamma^1_r\from (-\infty,-r)\to\C, \gamma^1_r(\rho) \ce -\rho \e^{-\i \eta}$, $\gamma^2_r\from (-\eta,\eta)\to\C, \gamma^2_r(\varphi) \ce r \e^{\i \varphi}$, and $\gamma^3_r\from (r,\infty)\to\C, \gamma^3_r(\rho) \ce \rho \e^{\i \eta}$.
\end{definition}
\begin{center}
\captionsetup{type=figure}
        \begin{tikzpicture}[>=stealth, scale=0.7]
    \def\etaa{3/5*pi * (1/pi*180)};
    \def\r{1};
    \draw[->] (-3,0)--(3,0) node[right]{$\Re$};
    \draw[->] (0,-3)--(0,3) node[above]{$\Im$};
    \draw[domain=-3:(-\r),samples=50,thick] plot ({-\x*cos(-\etaa)}, {-\x*sin(-\etaa)});
    \draw[domain=-\etaa:\etaa,samples=50,thick] plot ({\r*cos(\x)},{\r*sin(\x)});
    \draw[domain=\r:3,samples=50,thick] plot ({\x*cos(\etaa)}, {\x*sin(\etaa)});
    \node at (-1.2,2.7) {$\gamma_r$};
    \draw[-] (1,0.3)--(1,-0.3); \node at (1,-0.6) {$r$};
    \draw[->] (1.4,0) arc (0:\etaa:1.4) node[midway, right] {$\eta$};
    \draw[dashed] (0,0) -- (130:3);
    \draw[dashed] (0,0) -- (-130:3);
    \draw[->, densely dashdotted] (1.8,0) arc (0:-130:1.8) node[midway, below right] {$\eta_0$};
    \draw[->, densely dashdotted] (-1.8,0) arc (-180:-130:1.8) node[midway, left] {$\theta_0$};
    \node at (-1.8,0.8) {$-\Sigma_{\theta_0}$};
\end{tikzpicture}
\captionof{figure}{The curve $\gamma_r$.}
\end{center}

We summarise some consequences of Assumption \ref{ass:coeffCond}.
\begin{lemma}\label{lem:PointwiseEstimatesCoeffCond}
    Suppose that Assumption \ref{ass:coeffCond} holds for some $\ell\in \N_0$. Then the following statements hold for all $\alpha \in \calN$ with $|\alpha|\le \ell$ and $\PZ$-almost every $z \in \R^N$, with all constants independent of $\alpha$ and $z$. 
    \begin{enumerate}[label=(\alph*)]
        \item\label{item:pointwiseAzDerivativeBdd}
        We have $\|\rho(z)^{\alpha/2}\partial_z^\alpha A_z\|_{\calL(D,H)}\le C_\ell$ for some $C_\ell\ge 0$.
        \item \label{item:derivativeResolventHD}
        For all $r >0$ there exists $C_{\ell,r}\geq0$ such that for all $\lambda \in \gamma_r$, we have $\|\rho(z)^{\alpha/2}\partial_z^\alpha(\lambda+A_z)^{-1}\|_{\calL(H,D)}\le C_{\ell,r}$ with $C_{0,r}=2+\frac{1}{r}$.
        \item \label{item:derivativeResolventHBddonGamma}
        For all $\bar{T}>0$, there exists $C_{\bar{T},\ell}\geq0$ such that $\|\rho(z)^{\alpha/2}\partial_z^\alpha(\lambda+A_z)^{-1}\|_{\calL(H)}\le \frac{C_{\bar{T},\ell}}{|\lambda|}$ for all $\lambda\in\gamma_{\frac{1}{t}}$ with $t \in (0,\bar{T}]$.
        \item \label{item:pointwiseTzderivativesBdd}  For all $\bar{T}>0$ there exists $C_{\bar{T},\ell}\ge 0$ such that
        $\norm{\rho(z)^{\alpha/2}\partial_z^\alpha T_z(t)}_{\calL(H)} \le C_{\bar{T},\ell}$ for all $t \in [0,\bar{T}]$.
        \item  \label{item:pointwiseAnalytic} For all $\bar{T}>0$ there exists $C_{\bar{T},\ell}\ge 0$ such that for all $t \in [0,\bar{T}]$,
        \[
            \sup_{0 \le \tau \le t}\|\rho(z)^{\alpha/2} \partial_z^\alpha(\tau A_zT_z(\tau))\|_{\calL(H)}\le C_{\bar{T},\ell}.
        \] 
        \end{enumerate}
\end{lemma}
\begin{proof}
    We start by showing \ref{item:pointwiseAzDerivativeBdd}. For $u \in D$ and $\PZ$-almost every $z\in\R^N$, we estimate
        \begin{align*}
            \|\rho(z)^{\alpha/2}\partial_z^\alpha A_z u\|_{H}&= \sup_{v \in V, \|v\|_H=1} |\rho(z)^{\alpha/2} \sp{\partial_z^\alpha A_z u}{v}_H|= \sup_{v \in V, \|v\|_H=1} |\rho(z)^{\alpha/2} \partial_z^\alpha\form_z(u,v)|\\
            & \le \sup_{v \in V, \|v\|_H=1} C_\ell \|u\|_D \|v\|_H = C_\ell \|u\|_D.
        \end{align*}
    
    By uniform sectoriality of $(A_z)_{z\in\R^N}$, 
    \begin{equation}
    \label{eq:AzSectorial}
        \|\lambda(\lambda+A_z)^{-1}\|_{\calL(H)} \le 1
    \end{equation}
    for all $\lambda \in \Sigma_\eta\supseteq \gamma_r$ with $\eta<\eta_0$. This implies \ref{item:derivativeResolventHD} for $\alpha=0$ via $\rho(z)^{0/2}=1$ and
    \begin{align*}
        \|(\lambda+A_z)^{-1}u\|_D &=\sqrt{\|I-\lambda(\lambda+A_z)^{-1}u\|_H^2+\|(\lambda+A_z)^{-1}u\|_H^2}\\
        &\le \sqrt{4+\frac{1}{|\lambda|^2}} \|u\|_H \le \sqrt{4+\frac{1}{r^2}}\|u\|_H \le \left(2+\frac{1}{r}\right)\|u\|_H
    \end{align*}
    for all $\lambda \in \gamma_r$, where we have used that $\abs{\lambda}\ge r$ on $\gamma_r$ and $A_z(\lambda+A_z)^{-1}=I-\lambda(\lambda+A_z)^{-1}$. 
    To show \ref{item:derivativeResolventHD} for general $\abs{\alpha}\le \ell$, we assume without loss of generality that $\alpha-e_1 \in \calN$. We note that by Leibniz' rule and iteratively rewriting the derivatives of the resolvent, for $\PZ$-almost every $z\in\R^N$, we obtain
    \begin{align*}
         \partial_z^{\alpha}(\lambda+A_z)^{-1} &= \partial_z^{\alpha-e_1}(-(\lambda+A_z)^{-1}(\partial_z^{e_1} A_z)(\lambda+A_z)^{-1})\\
        &= -\sum_{\beta\leq \alpha-e_1} \binom{\alpha-e_1}{\beta} \partial_z^{\alpha-e_1-\beta} (\lambda+A_z)^{-1} \sum_{\gamma\le \beta} \binom{\beta}{\gamma} (\partial_z^{\beta-\gamma+e_1} A_z) \partial_z^{\gamma} (\lambda+A_z)^{-1}\\
        &= \ldots =(\lambda+A_z)^{-1} \sum_{P_j} \bigg[\prod_{i=1}^{p_j} C_{\alpha,i,j} \big(\partial_z^{\beta_i^j} A_z\big)  (\lambda+A_z)^{-1}\bigg]
    \end{align*}
    for some $C_{\alpha,i,j}\in \R$. The sum is taken over all partitions $P_j=(\beta_i^j)_{1 \le i \le p_j}$ of $\alpha$ such that $\alpha=\beta_1^j+\ldots+\beta_{p_j}^j$ for some $1 \le p_j \le |\alpha|$ and with $\beta_i^j \neq 0$ for all $i,j$. Consequently, also
    \begin{align}
    \label{eq:LeibnizProofResolventHD}
        \rho(z)^{\alpha/2}\partial_z^{\alpha}(\lambda+A_z)^{-1}
        &=(\lambda+A_z)^{-1} \sum_{P_j} \bigg[\prod_{i=1}^{p_j} C_{\alpha,i,j} \big(\rho(z)^{\beta_i^j/2}\partial_z^{\beta_i^j} A_z\big)  (\lambda+A_z)^{-1}\bigg].
    \end{align}
    Since as many derivatives (of some order) of $A_z$ appear (weighted with the respective power of $\rho(z)$) as resolvents of $A_z$, we can estimate the $\calL(H)$-norm of each factor in the product by $\abs{C_{\alpha,i,j}}C_\ell(2+\frac{1}{r})$ by part \ref{item:pointwiseAzDerivativeBdd} and the sectoriality estimate \eqref{eq:AzSectorial}. Since sum and product are finite, the estimate $\alpha=0$ applied to the first factor in \eqref{eq:LeibnizProofResolventHD} yields the claim.
    
    We continue with \ref{item:derivativeResolventHBddonGamma}. Let $r=\frac{1}{t}$. For $\alpha=0$, \eqref{eq:AzSectorial} yields the claim with $C_{\bar{T},\ell}=1$. For general $\alpha$, we note that the $\calL(H)$-norm of each factor of the product in \eqref{eq:LeibnizProofResolventHD} is bounded by $C_\ell(2+\frac{1}{r})=C_\ell(2+t)\le C_\ell(2+\bar{T})$.
    The $\calL(H)$-norm of \eqref{eq:LeibnizProofResolventHD} can thus be estimated by $C_{\bar{T},\ell}\|(\lambda+A_z)^{-1}\|_{\calL(H)}$, where the norm is bounded by $\abs{\lambda}^{-1}$ due to \eqref{eq:AzSectorial}.
    
    We pass to the proof of \ref{item:pointwiseTzderivativesBdd}. Observe that \ref{item:pointwiseTzderivativesBdd} and \ref{item:pointwiseAnalytic} are trivially satisfied for $t=0$. Hence, let $t>0$ and $r=\frac{1}{t}$. The Laplace inversion formula \cite[Corollary~III.5.15]{EngelNagel2000} yields the representation
    \[
        T_z(t)u = \frac{1}{2\pi \i} \int_{\gamma_r} \e^{t\lambda} (\lambda + A_z)^{-1}u \,\rmd \lambda
    \]
    of the semigroup for all $t>0$, $u\in H$, and $\PZ$-almost every $z\in\R^N$. Interchanging differentiation and integration yields that $z\mapsto T_z(t)u$ is continuously differentiable up to order $\ell$. Furthermore, by part \ref{item:derivativeResolventHBddonGamma} there exists $C_{\bar{T},\ell}\ge 0$ with
    \begin{align*}
        \big\|\rho(z)^{\alpha/2}\partial_z^\alpha T_z(t)u\big\|_H 
        &\le \frac{1}{2\pi} \int_{\gamma_r} |\e^{t\lambda}| \|\rho(z)^{\alpha/2}\partial_z^\alpha(\lambda + A_z)^{-1}u\|_H \,\rmd \lambda
        \leq \frac{C_{\bar{T},\ell}}{2\pi} \norm{u}_H  \int_{\gamma_r} \frac{|\e^{t\lambda}|}{\abs{\lambda}}  \,\rmd \lambda 
    \end{align*}
    for all $t\in [0,\bar{T}]$, $u\in H$, $\abs{\alpha}\leq \ell$, and $\PZ$-almost every $z\in\R^N$. It remains to estimate the curve integral uniformly in $t$. On $\gamma_r^3$, due to $r=\frac{1}{t}$, we obtain
    \begin{equation*}
        \int_{\gamma_r^3} \frac{|\e^{t\lambda}|}{\abs{\lambda}} \,\rmd \lambda = \int_r^\infty \frac{\e^{t\rho \cos(\eta)}}{\rho} \,\rmd \rho 
        \le \frac{1}{r} \int_r^\infty \e^{t\rho \cos(\eta)} \,\rmd \rho 
        = \frac{1}{rt} \frac{\e^{rt \cos(\eta)}}{ \abs{\cos(\eta)}} 
        = \frac{\e^{ \cos(\eta)}}{ \abs{\cos(\eta)}}.
    \end{equation*}
    An analogous estimate holds on $\gamma_r^1$. The proof of \ref{item:pointwiseTzderivativesBdd} is finished by estimating
    \begin{equation*}
        \int_{\gamma_r^2} \frac{|\e^{t\lambda}|}{\abs{\lambda}} \,\rmd \lambda = \int_{-\eta}^\eta \frac{\e^{rt \cos(\varphi)}}{r} \cdot r\,\rmd \varphi 
        \le \int_{-\eta}^\eta \e^{1} \,\rmd \varphi
        = 2\e \eta.
    \end{equation*}

    Lastly, we show \ref{item:pointwiseAnalytic}. Leibniz' rule, multiplying by $1$, and part \ref{item:pointwiseAzDerivativeBdd} yield
    \begin{align*}
        \big\|\rho(z)^{\alpha/2}\partial_z^\alpha(\tau A_zT_z(\tau))u\big\|_H 
        &= \bigg\|\sum_{\beta \le \alpha} \binom{\alpha}{\beta} \big(\rho(z)^{(\alpha-\beta)/2}\partial_z^{\alpha-\beta}A_z\big) A_z^{-1} \tau A_z(\rho(z)^{\beta/2} \partial_z^{\beta}T_z(\tau))u\bigg\|_H\\
        &\le C_{\abs{\alpha}}C_\ell  \|A_z^{-1}\|_{\calL(H,D)} \max_{\beta \le \alpha}\big\|\tau A_z(\rho(z)^{\beta/2} \partial_z^{\beta}T_z(\tau))u\big\|_H.
    \end{align*}
    The assumed coercivity implies $\|A_zu\|_{H} \ge \kappa\|u\|_H$ for all $u\in D$ via the embedding $V \hra H$, and thus $\|A_z^{-1}\|_{\calL(H)} \le \frac{1}{\kappa}$. Consequently, $\|A_z^{-1}u\|_D^2 = \|u\|_H^2+\|A_z^{-1}u\|_H^2 \le (1+\frac{1}{\kappa^2})\|u\|_H^2$. Hence, $\|A_z^{-1}\|_{\calL(H,D)}^2 \le 1+\frac{1}{\kappa^2}$. It remains to estimate the last factor. Again using the Laplace inversion formula, for $r=\bar{T}^{-1}$, \eqref{eq:LeibnizProofResolventHD} with all partitions $P_j$ of $\beta$ instead of $\alpha$ allows us to rewrite
    \begin{align*}
        \tau A_z(\rho(z)^{\beta/2} \partial_z^\beta T_z(\tau))u &= \frac{1}{2\pi \i} \int_{\gamma_r} \tau\e^{\tau\lambda} A_z (\rho(z)^{\beta/2} \partial_z^\beta(\lambda + A_z)^{-1})u \,\rmd \lambda\\
        & = \frac{C_\ell}{2\pi \i} \int_{\gamma_r} \tau\e^{\tau\lambda} A_z  (\lambda+A_z)^{-1} \sum_{P_j} \bigg[\prod_{i=1}^{p_j} \big(\rho(z)^{\beta_i^j/2} \partial_z^{\beta_i^j} A_z\big)(\lambda+A_z)^{-1} \bigg] u \,\rmd \lambda.
    \end{align*}
    In the proof of \ref{item:derivativeResolventHBddonGamma}, we have already shown that the $\calL(H)$-norm of the sum over all partitions $P_j$ is bounded by some $C_{\bar{T},\ell}\ge 0$. We further recall that
    \[
        \|A_z(\lambda+A_z)^{-1}\|_{\calL(H)} = \|I-\lambda(\lambda+A_z)^{-1}\|_{\calL(H)} \le 2
    \]
    by the sectoriality estimate \eqref{eq:AzSectorial}. Hence,
    \begin{align*}
        \|\rho(z)^{\beta/2}\tau A_z\partial_z^\beta T_z(\tau)u\|_H &\le \frac{C_\ell C_{\bar{T},\ell}}{\pi} \int_{\gamma_r}  \tau |\e^{\tau\lambda}|\,\rmd \lambda\le \frac{C_\ell C_{\bar{T},\ell}}{\pi} \left(\frac{2}{\abs{\cos(\eta)}} +2\e\eta\right) \|u\|_H,
    \end{align*}
    where the last step is obtained by an estimation of the curve integral.
\end{proof}

\subsubsection{Mapping properties of \texorpdfstring{$\bfA$}{\textbf{A}} and its resolvent}
\label{subsec:mappingPropertiesBfA}

From the pointwise estimates for $A_z$, we now derive estimates for $\bfA$, its resolvent, and the semigroup generated by $-\bfA$ in suitable Sobolev spaces.

\begin{proposition}
    \label{prop:SobolevEstBfA}
    Suppose that Assumption \ref{ass:coeffCond} holds for some $\ell\in \N_0$. Then for all $q\in\{0,\ldots,\ell\}$, $r>0$, and $\bar{T}>0$ there are constants $C_q,C_{q,r}, C_{\bar{T},q} \ge 0$ such that the following estimates hold for all $\bff\in H_\rho^q(\R^N,\PZ;H)$.
    \begin{enumerate}[label=(\alph*)]
    \item \label{item:bfAfSobolevnorm}  
        $\|\bfA \bff\|_{H_\rho^q(\R^N,\PZ;H)} \le C_q\|\bff\|_{H_\rho^q(\R^N,\PZ;D)}$ if additionally $\bff\in H_\rho^q(\R^N,\PZ;D)$. In particular, $\|\bfA \bff\|_\bfH \le C_0\|\bff\|_{L_2(\R^N,\PZ;D)}$.
        \item\label{item:ResolventBfH} $\|(\lambda+\bfA)^{-1}\bff\|_{H_\rho^{q}(\R^N,\PZ;D)} \le C_{q,r} \|\bff\|_{H_\rho^{q}(\R^N,\PZ;H)}$ for all $\lambda \in \gamma_r$. In particular, $\|(\lambda+\bfA)^{-1}\bff\|_{L_{2}(\R^N,\PZ;D)} \le C_{
        0,r} \|\bff\|_{\bfH}$.
        \item \label{item:bfTSobolevBdd} $\|\bfT(t)\bff\|_{H_\rho^{q}(\R^N,\PZ;H)} \le C_{\bar{T},q} \|\bff\|_{H_\rho^{q}(\R^N,\PZ;H)}$ for all  $t \in [0,\bar{T}]$.
        \item \label{item:bfTanalyticEst} $\sup_{0 \le \tau \le t}\|\tau\bfA\bfT(\tau)\bff\|_{H_\rho^{q}(\R^N,\PZ;H)}\le C_{\bar{T},q}\|\bff\|_{H_\rho^{q}(\R^N,\PZ;H)}$ for all $t \in [0,\bar{T}]$.
    \end{enumerate}
\end{proposition}

We need a composition estimate to lift the pointwise estimates from Subsection \ref{subsec:pointwiseEst} to subspaces of $\bfH$.

\begin{lemma}
\label{lem:Smoothness_composition}
    Let $H_1$ and $H_2$ be Hilbert spaces, $\ell\in\N_0$, and $F\from \R^N\times H_1\to H_2$ such that $F(z,\cdot) \in \calL(H_1,H_2)$ for $\PZ$-almost every $z\in \R^N$ as well as $\PZ$-$\esssup_{z\in\R^N}\norm{F(z,\cdot)}_{\calL(H_1,H_2)} < \infty$. Further, suppose that $F(\cdot,u)\in C^\ell(\R^N;H_2)$ and there exists $c_\ell\geq 0$ such that $\|\rho(z)^{\alpha/2}\partial_z^\alpha F(z,u)\|_{H_2} \leq c_\ell\norm{u}_{H_1}$ for all $u\in H_1$, $\abs{\alpha}\leq \ell$, and $\PZ$-almost every $z\in\R^N$.
    Let $G\in H_\rho^\ell(\R^N,\PZ;H_1)$.   
    Then $F(\cdot,G(\cdot))\in H_\rho^\ell(\R^N,\PZ;H_2)$ and there exists $C_\ell\geq 0$ such that 
    \[
        \norm{F(\cdot,G(\cdot))}_{H_\rho^\ell(\R^N,\PZ;H_2)} \leq C_\ell \norm{G}_{H_\rho^\ell(\R^N,\PZ;H_1)}.
    \]
\end{lemma}
\begin{proof}
    By induction, the product rule and the chain rule, or more precisely, their generalizations Leibniz' rule and Faà di Bruno's formula, we obtain the assertion. We sketch the proof for $\ell=1$, since the induction step $\ell \to \ell +1$ follows analogously using the two rules.\\
    The induction start is an immediate consequence of  $\|F(z,u)\|_{H_2} \le c_0 \|u\|_{H_1}$ for $\PZ$-almost every $z \in \R^N$ taking $u=G(z) \in H_1$. 
    Since by assumption, $F(z,\cdot)$ is linear for $\PZ$-almost every $z \in \R^N$, we obtain $\partial_2 F(z,\cdot) = F(z,\cdot) \in \calL(H_1,H_2)$ for $\PZ$-almost every $z \in \R^N$ and higher derivatives vanish. For $\PZ$-almost every $z \in \R^N$, we can thus find $C \ge 0$ such that
    \begin{align*}
        \|\partial_z^{e_j} F(z,G(z))\|_{H_2}
        &= \|\partial_1^{e_j} F(z,G(z)) + [\partial_2 F(z,G(z))](\partial_z^{e_j} G(z))\|_{H_2}\\
        &\le \|\partial_1^{e_j} F(z,G(z))\|_{H_2} + \Big( \PZ\text{-}\esssup_{z \in \R^N}\|F(z,G(z))\|_{\calL(H_1,H_2)}\Big)\|\partial_z^{e_j} G(z)\|_{H_1}.
    \end{align*}
    Multiplying with $\rho(z)^{e_j/2}=(\rho(z))_j^{1/2}\in \R$, from the assumptions we can deduce
    \[
        \|\rho(z)^{e_j/2}\partial_z^{e_j} F(z,G(z))\|_{H_2} \le c_\ell \|G(z)\|_{H_1} + C \|\rho(z)^{e_j/2}\partial_z^{e_j} G(z)\|_{H_1}.
    \]
    Integrating in $z$, noting that $\rho(z)^0=1$, and summing over $j\in\{1,\ldots,N\}$ yields that, for some $C_\ell \ge 0$,
    \begin{align*}
        [F(\cdot,G(\cdot))]_{H_\rho^1(\R^N,\PZ;H_2)}^2 \le C_\ell \|G\|_{H_\rho^1(\R^N,\PZ;H_1)}^2
    \end{align*}
    for the seminorm $[\cdot]_{H_\rho^1(\R^N,\PZ;H_2)}$ in $H_\rho^1(\R^N,\PZ;H_2)$. The norm estimate follows from the induction assumption, i.e., for $\ell=1$ from the induction start.
\end{proof}

\begin{proof}[Proof of Proposition \ref{prop:SobolevEstBfA}]
    All statements are consequences of Lemma \ref{lem:Smoothness_composition} with $F$ chosen as below and $G=\bff$, which can be applied due to Lemma \ref{lem:PointwiseEstimatesCoeffCond}\ref{item:pointwiseAzDerivativeBdd}, \ref{item:derivativeResolventHD}, \ref{item:pointwiseTzderivativesBdd}, and \ref{item:pointwiseAnalytic}, respectively.
    \begin{enumerate}[label=(\alph*)]
        \item
            $F\from \R^N\times D\to H$, $F(z,u)\ce A_zu$
        \item 
            $F\from \R^N\times H\to D$, $F(z,u)\ce (\lambda+A_z)^{-1}u$
        \item 
            $F\from \R^N\times H\to H$, $F(z,u)\ce T_z(t)u$
        \item 
           $F\from \R^N\times H\to H$, $F(z,u)\ce \sup_{0 \le \tau \le t} \tau A_zT_z(\tau)u$ \qedhere
    \end{enumerate}
\end{proof}

\subsubsection{Estimate of the difference of resolvents}
\label{subsec:resolventDiffEst}

We now prove decay in $n$ of the difference of the resolvents of $-\bfA$ and $-\bfR_n\bfA\bfR_n$, which is the key ingredient for the convergence rate in randomness in Subsection \ref{subsec:TrotterKato}. To be able to speak of convergence, we henceforth consider approximation orders $n \in \N$ rather than $n \in \N_0$.

\begin{proposition}
\label{prop:resolventEstimate}
    Let $\ell \in \N$ and $r>0$. Suppose that Assumption \ref{ass:coeffCond} holds for $2\ell$. Then there exists $C_{\ell,r}\geq 0$ such that for all $\lambda\in\gamma_r$, we have 
    \begin{equation*}
        \|[(\lambda+\bfR_n\bfA\bfR_n)^{-1}\bfR_n-(\lambda+\bfA)^{-1}]\bff\|_\bfH \le C_{\ell,r}n^{-\ell} \|\bff\|_{H_\rho^{2\ell}(\R^N,\PZ;H)}
    \end{equation*}
    for all $n \in \N$ and $\bff \in H_\rho^{2\ell}(\R^N,\PZ;H)$.
\end{proposition}

Note that convergence of order $\ell$ requires regularity of order $2\ell$ in the sense of $\bff \in H_\rho^{2\ell}(\R^N,\PZ;H)$. Given sufficient regularity in $z$, the convergence is thus of arbitrarily high (polynomial) order. We write $C_\ell$ also for constants depending on the value of $2\ell$.
For the proof, we need a lemma on the difference of the generators. 
\begin{lemma}
\label{lem:generatorEstimate}
    Let $\ell \in \N$. Suppose that Assumption \ref{ass:coeffCond} holds for $2\ell$. Then there exists $C_\ell\geq 0$ such that for all $n\in\N$ and $\bff \in H_\rho^{2\ell}(\R^N,\PZ;D)$, we have
    \begin{equation*}
        \|(\bfR_n\bfA\bfR_n-\bfA)\bff\|_\bfH \le C_\ell n^{-\ell} \|\bff\|_{H_\rho^{2\ell}(\R^N,\PZ;D)}.
    \end{equation*}
\end{lemma}
\begin{proof}
    Note that $\bfR_n$ is contractive on $\bfH$ (being an orthogonal projection). Proposition \ref{prop:SobolevEstBfA}\ref{item:bfAfSobolevnorm} applied to $r \in \{0,2\ell\}$ and the PCE error estimate from Corollary \ref{cor:PCEboundvector} applied to both $D$- and $H$-valued functions yield
    \begin{align*}
        \|\bfR_n\bfA\bfR_n\bff-\bfA\bff\|_\bfH &\le \|\bfR_n\bfA\bfR_n\bff- \bfR_n\bfA\bff\|_\bfH+\|\bfR_n\bfA\bff-\bfA\bff\|_\bfH \le \|\bfA(\bfR_n-I)\bff\|_\bfH+\|(\bfR_n-I)\bfA\bff\|_\bfH\\
        &\le C_0\|(\bfR_n-I)\bff\|_{L_2(\R^N,\PZ;D)}+\|(\bfR_n-I)\bfA\bff\|_\bfH\\
        &\le C n^{-\ell} \big(C_0\|\bff\|_{H_\rho^{2\ell}(\R^N,\PZ;D)}+\|\bfA\bff\|_{H_\rho^{2\ell}(\R^N,\PZ;H)} \big)\le C (C_0+C_{\ell})n^{-\ell} \|\bff\|_{H_\rho^{2\ell}(\R^N,\PZ;D)}. \qedhere
    \end{align*}
\end{proof}

\begin{proof}[Proof of Proposition \ref{prop:resolventEstimate}]
    The second resolvent identity and $\bfR_n^2=\bfR_n$ imply
    \begin{align*}
        (\lambda+\bfR_n\bfA\bfR_n)^{-1}\bfR_n-(\lambda+\bfA)^{-1}
        = (\lambda+\bfR_n\bfA\bfR_n)^{-1}\bfR_n(\bfA-\bfR_n\bfA\bfR_n)(\lambda+\bfA)^{-1}.
    \end{align*}
    Applying, in this order, Lemma \ref{lem:generatorEstimate}, Lemma \ref{lem:RnResolventBounded}, and Proposition \ref{prop:SobolevEstBfA}\ref{item:ResolventBfH} for $q=2\ell$ results in
    \begin{align*}
        \|[(\lambda+\bfR_n\bfA\bfR_n)^{-1}\bfR_n-(\lambda+\bfA)^{-1}]\bff\|_\bfH
        &\le \frac{C}{|\lambda|}\|(\bfA-\bfR_n\bfA\bfR_n)(\lambda+\bfA)^{-1}\bff\|_\bfH\\
        \le \frac{C}{r} C_\ell n^{-\ell}\|(\lambda+\bfA)^{-1}\bff\|_{H_\rho^{2\ell}(\R^N,\PZ;D)}
        &\le \frac{C}{r} C_\ell C_{\ell,r} n^{-\ell}\|\bff\|_{H_\rho^{2\ell}(\R^N,\PZ;H)}. \qedhere
    \end{align*}
\end{proof}

\subsubsection{Error estimate in randomness: A Trotter--Kato argument}
\label{subsec:TrotterKato}

Finally, we conclude convergence of the approximate mild solutions via strong convergence of the associated semigroups. To show the latter, we prove a suitable adaptation of the quantified version of the Trotter--Kato theorem \cite[Proposition~2.3a]{kappel1}. It relates the convergence rate for the resolvents from the last subsection to a convergence rate for the semigroups. 

\begin{theorem}
\label{thm:errorunuTrotterKato}
    Let $\ell \in \N$. Suppose that Assumption \ref{ass:coeffCond} holds for $2\ell$. Then, for any $\bar{T}>0$, there exists a constant $C_{\bar{T},\ell}>0$ such that
    \[
        \|\bfT_n(t)\bfR_n\bff-\bfT(t)\bff\|_\bfH \le C_{\bar{T},\ell} n^{-\ell} \|\bff\|_{H_\rho^{2\ell}(\R^N,\PZ;D)}
    \]
    for all $\bff \in H_\rho^{2\ell}(\R^N,\PZ;D)$, $n \in \N$, and $t \in [0,\bar{T}]$. In particular, for the mild solutions $\bfu_n(t)$ and $\bfu(t)$ of \eqref{eq:defunACP} and \eqref{eq:randomACP} with $\bfu_0 \in H_\rho^{2\ell}(\R^N,\PZ;D)$, respectively, we have
    \[
        \|\bfu_n(t)-\bfu(t)\|_\bfH \le C_{\bar{T},\ell} n^{-\ell} \|\bfu_0\|_{H_\rho^{2\ell}(\R^N,\PZ;D)} \quad (t \in [0,\bar{T}]).
    \]
\end{theorem}
\begin{proof}
    This proof is inspired by \cite[Proposition~2.3a]{kappel1} for the deterministic case. Let $\bff \in H_\rho^{2\ell}(\R^N,\PZ;D) \seq \dom(\bfA)$. Recall that for $n\in\N$ the semigroup $(\bfT_n(t))_{t \ge 0}$ is generated by $-\bfR_n\bfA\bfR_n$. Consider the error $\bfe_n(t) \ce [\bfT_n(t)\bfR_n-\bfR_n\bfT(t)]\bff$. Then, the scaled error $\bfs_n(t) \ce t \bfe_n(t)$ satisfies
    \begin{align*}
        \bfs_n'(t) = -\bfR_n\bfA\bfR_n \bfs(t)+\bfe_n(t)+t\bfR_n\bfA\bfR_n \Delta_n \bfA\bfT(t)\bff+t\,\Box_n\bfT(t)\bff\quad(t>0),\quad \bfs_n(0) = 0,
    \end{align*}
    where we have used $\bfR_n^2=\bfR_n$ and set 
    \begin{align*}
        \Delta_n &\ce \Delta_n(\kappa) \ce \left(\frac{\kappa}{2}+\bfA\right)^{-1}-\left(\frac{\kappa}{2}+\bfR_n\bfA\bfR_n\right)^{-1}\bfR_n,\\
        \Box_n &\ce \Box_n(\kappa) \ce \frac{\kappa}{2}\Big[ \Big(\frac{\kappa}{2}+\bfR_n\bfA\bfR_n\Big)^{-1}\bfR_n\bfA - \bfR_n\bfA\bfR_n\Big(\frac{\kappa}{2}\bfA\Big)^{-1} \Big],
    \end{align*}
    noting that $\frac{\kappa}{2}$ is in the respective resolvent sets by coercivity of $\bfform$ and $\bfform_n$.
    Thus, $\bfs_n$ is given by the variation-of-constants formula
    \begin{align*}
        \bfs_n(t) &= \int_0^t \bfT_n(t-\tau) \bfe_n(\tau) \,\rmd \tau 
        -\int_0^t \bfT_n(t-\tau)\bfR_n\bfA\bfR_n \Delta_n \tau\bfA\bfT(\tau)\bff\,\rmd \tau
        +\int_0^t \bfT_n(t-\tau)\tau\,\Box_n\bfT(\tau)\bff \,\rmd \tau.
    \end{align*}
    Integration by parts in the second integral and division by $t$ result in
    \begin{align*}
        \bfe_n(t) &= \frac{1}{t} \int_0^t \bfT_n(t-\tau) \bfe_n(\tau)\,\rmd \tau -  \Delta_n \bfA \bfT(t)\bff+ \frac{1}{t} \int_0^t \bfT_n(t-\tau)  \Delta_n \bfA \bfT(\tau)\bff\,\rmd \tau\\
        &\phantom{= }+ \frac{1}{t} \int_0^t \bfT_n(t-\tau)  \Delta_n \tau \bfA^2 \bfT(\tau)\bff\,\rmd \tau +\frac{1}{t}\int_0^t \bfT_n(t-\tau) \tau \,\Box_n \bfT(\tau)\bff\,\rmd\tau  \ec E_1+E_2+E_3+E_4+E_5,
    \end{align*}
    where $E_j \ce E_j(t,n)$ for $j \in \{1,\ldots,5\}$. We bound the norm of all five terms separately, starting with the second one.
    The resolvent difference estimate of Proposition \ref{prop:resolventEstimate} for $r=\frac{\kappa}{2}$, the boundedness of the semigroup on $H_\rho^{2\ell}(\R^N,\PZ;H)$ from Proposition \ref{prop:SobolevEstBfA}\ref{item:bfTSobolevBdd}, and the generator bound from Proposition \ref{prop:SobolevEstBfA}\ref{item:bfAfSobolevnorm} imply
    \begin{align}
    \label{eq:E2estimate}
        \|E_2\|_\bfH &\le \|\Delta_n\|_{\calL(H_\rho^{2\ell}(\R^N,\PZ;H),\bfH)} \|\bfA\bfT(t)\bff\|_{H_\rho^{2\ell}(\R^N,\PZ;H)}\le C_{\ell,\kappa}n^{-\ell} \|\bfT(t)\bfA\bff\|_{H_\rho^{2\ell}(\R^N,\PZ;H)}\nonumber\\
        &\le C_{\ell,\kappa}C_{\bar{T},\ell} n^{-\ell} \|\bfA\bff\|_{H_\rho^{2\ell}(\R^N,\PZ;H)}\le C_{\ell,\kappa}C_{\bar{T},\ell}  C_{\ell} n^{-\ell}  \|\bff\|_{H_\rho^{2\ell}(\R^N,\PZ;D)},
    \end{align}
    where we have written $C_{\ell,\kappa}$ for $C_{\ell,r}$ with $r=\frac{\kappa}{2}$. Using the contractivity of $(\bfT_n(t))_{t \ge 0}$ from Lemma \ref{lem:RnResolventBounded} allows us to proceed as for $E_2$ to obtain
    \begin{align}
    \label{eq:E3estimate}
        \|E_3\|_\bfH &\le \frac{1}{t} \int_0^t \|\Delta_n\|_{\calL(H_\rho^{2\ell}(\R^N,\PZ;H),\bfH)} \|\bfA\bfT(\tau)\bff\|_{H_\rho^{2\ell}(\R^N,\PZ;H)}\,\rmd \tau
        \le C_{\ell,\kappa}C_{\bar{T},\ell}  C_{\ell} n^{-\ell}\|\bff\|_{H_\rho^{2\ell}(\R^N,\PZ;D)}.
    \end{align}
    To estimate the fourth error term, combining the arguments issued for $E_3$ with the analyticity estimate from Proposition \ref{prop:SobolevEstBfA}\ref{item:bfTanalyticEst} yields
    \begin{align}
    \label{eq:E4estimate}
        \|E_4\|_\bfH &\le C_{\ell,\kappa} n^{-\ell}\cdot\frac{1}{t}\int_0^t  \|\tau\bfA\bfT(\tau)\bfA\bff\|_{H_\rho^{2\ell}(\R^N,\PZ;H)}\,\rmd \tau\le C_{\ell,\kappa} C_{\bar{T},\ell} n^{-\ell}  \|\bfA\bff\|_{H_\rho^{2\ell}(\R^N,\PZ;H)}\nonumber\\
        &\le C_{\ell,\kappa} C_{\bar{T},\ell} C_{\ell} n^{-\ell} \|\bff\|_{H_\rho^{2\ell}(\R^N,\PZ;D)}.
    \end{align}
    Let $0\le \tau\le \bar{T}$. Applying, in this order, the triangle inequality, Lemma \ref{lem:generatorEstimate}, Proposition \ref{prop:resolventEstimate}, and Proposition \ref{prop:SobolevEstBfA}\ref{item:ResolventBfH}, \ref{item:bfTSobolevBdd}, and \ref{item:bfAfSobolevnorm}, we obtain the estimate
    \begin{align}
    \label{eq:boxnEstimate}
        \|\Box_n \bfT(\tau)\bff\|_\bfH
        &\le \frac{\kappa}{2} \Big\|[\bfR_n\bfA\bfR_n -\bfA]\Big(\frac{\kappa}{2}+\bfA\Big)^{-1}\bfT(\tau)\bff \Big\|_\bfH + \frac{\kappa}{2} \|\Delta_n\bfT(\tau)\bfA\bff \|_\bfH\nonumber\\
        &\le \frac{\kappa}{2} C_{\ell,\kappa} C_{\bar{T},\ell} C_\ell n^{-\ell} \big(\|\bff\|_{H_\rho^{2\ell}(\R^N,\PZ;H)} + \|\bff\|_{H_\rho^{2\ell}(\R^N,\PZ;D)}\big)\nonumber\\
        &\le \kappa C_{\ell,\kappa} C_{\bar{T},\ell} C_\ell n^{-\ell} \|\bff\|_{H_\rho^{2\ell}(\R^N,\PZ;D)}.
    \end{align}
    Consequently, we can bound the fifth term by
    \begin{align}
    \label{eq:E5estimate}
        \|E_5\|_\bfH &\le \frac{1}{t}\int_0^t  \|\tau\,\Box_n\bfT(\tau)\bff\|_{\bfH}\,\rmd \tau\le \frac{\bar{T}}{2} \kappa C_{\ell,\kappa} C_{\bar{T},\ell} C_\ell n^{-\ell} \|\bff\|_{H_\rho^{2\ell}(\R^N,\PZ;D)}.
    \end{align}
    To estimate the remaining term $E_1$, we first note that $\bfe_n$ satisfies the ODE $\bfe_n(0)=0$,
    \begin{equation}
    \label{eq:ODEbfen}
        \bfe_n'(t) = -\bfR_n\bfA\bfR_n\bfe_n(t)-\bfR_n\bfA\bfR_n\Delta_n \bfA \bfT(t)\bff+\Box_n \bfT(t)\bff\quad (0<t\le \bar{T}),
    \end{equation}
    thus giving rise to the error representation for $0<\tau \le \bar{T}$
    \begin{align}
    \label{eq:enWithAninv}
        \bfe_n(\tau) = -(\bfR_n\bfA\bfR_n)^{-1} \bfe_n'(\tau)-\Delta_n \bfA \bfT(\tau)\bff+(\bfR_n\bfA\bfR_n)^{-1}\Box_n \bfT(\tau)\bff.
    \end{align}
   Furthermore, the $\PZ$-almost sure symmetry of $(\form_z)_{z \in \R^N}$ implies that $-(\bfR_n\bfA\bfR_n)^{-1}$ is self-adjoint. Thus, for $0 <\tau \le \bar{T}$,
    \begin{align}
    \label{eq:selfadjAninv}
        \Re \,\sp{ \bfe_n(\tau)}{-(\bfR_n\bfA\bfR_n)^{-1} \bfe_n'(\tau) }_\bfH = \frac{1}{2} \frac{\rmd}{\rmd \tau} \sp{ \bfe_n(\tau)}{-(\bfR_n\bfA\bfR_n)^{-1}\bfe_n(\tau)}_\bfH.
    \end{align}
    As a consequence of the coercivity of $\bfform_n$, $-\bfR_n\bfA\bfR_n$ is dissipative and thus also its inverse. Integrating the scalar product of $\bfe_n(\tau)$ with both sides of \eqref{eq:enWithAninv}, using \eqref{eq:selfadjAninv}, the dissipativity of $-(\bfR_n\bfA\bfR_n)^{-1}$, the Cauchy--Schwarz inequality, and Hölder's inequality, we deduce for $t \in (0, \bar{T}]$
    \begin{align*}
        \|\bfe_n\|_{L_2(0,t;\bfH)}^2 
        &= \frac12 \langle \bfe_n(t),-(\bfR_n\bfA\bfR_n)^{-1} \bfe_n(t) \rangle_\bfH
        +\int_0^t \langle \bfe_n(\tau),-\Delta_n \bfA \bfT(\tau)\bff \rangle_\bfH\,\rmd \tau\\
        &\phantom{\le }
        +\int_0^t \langle \bfe_n(\tau),(\bfR_n\bfA\bfR_n)^{-1}\Box_n \bfT(\tau)\bff \rangle_\bfH\,\rmd \tau\\
        &\le \|\bfe_n\|_{L_2(0,t;\bfH)} \big(\|\Delta_n \bfA \bfT(\cdot)\bff\|_{L_2(0,t;\bfH)} + \|(\bfR_n\bfA\bfR_n)^{-1}\Box_n \bfT(\cdot)\bff\|_{L_2(0,t;\bfH)}\big).
    \end{align*}
    Dividing by the norm of $\bfe_n$ and using the estimates \eqref{eq:E2estimate} and \eqref{eq:boxnEstimate} yields
    \begin{align*}
        \|\bfe_n\|_{L_2(0,t;\bfH)} 
        &\le 2\sqrt{t} C_{\ell,\kappa} C_{\bar{T},\ell}C_\ell n^{-\ell} \|\bff\|_{H_\rho^{2\ell}(\R^N,\PZ;D)},
    \end{align*}
    where we have used that $\|(\bfR_n\bfA\bfR_n)^{-1}\|_{\calL(\bfH)} \le \frac{1}{\kappa}$, which is an immediate consequence of the coercivity of $\bfform_n$. Finally, with the Cauchy--Schwarz inequality and uniform contractivity of $\bfT_n(t)$, this implies
    \begin{align}
    \label{eq:E1estimate}
        \|E_1\|_\bfH &= \bigg \| \frac{1}{t} \int_0^t \bfT_n(t-\tau) \bfe_n(\tau) \,\rmd \tau \bigg\|_\bfH \le \frac{1}{t} \|\bfT_n\|_{L_2(0,t;\calL(\bfH))} \|\bfe_n\|_{L_2(0,t;\bfH)}\nonumber\\
        &\le 2 C_{\ell,\kappa} C_{\bar{T},\ell}C_\ell n^{-\ell} \|\bff\|_{H_\rho^{2\ell}(\R^N,\PZ;D)}.
    \end{align}
    Combining \eqref{eq:E1estimate} with \eqref{eq:E2estimate}-\eqref{eq:E5estimate} and omitting the dependence on $\kappa$ results in
    \begin{equation*}
        \|\bfe_n(t)\|_\bfH \le C_{\bar{T},\ell}n^{-\ell} \|\bff\|_{H_\rho^{2\ell}(\R^N,\PZ;D)}.
    \end{equation*}
    Finally, the PCE estimate of Corollary \ref{cor:PCEboundvector} and Proposition \ref{prop:SobolevEstBfA}\ref{item:bfTSobolevBdd} give the desired estimate
    \begin{align*}
        \|\bfT_n(t)\bfR_n\bff-\bfT(t)\bff\|_\bfH 
        &\le \|\bfe_n(t)\|_\bfH + \|\bfR_n-I\|_{\calL(H_\rho^{2\ell}(\R^N,\PZ;H),\bfH)} \|\bfT(t)\bff\|_{H_\rho^{2\ell}(\R^N,\PZ;H)}\\
        &\le C_{\bar{T},\ell}n^{-\ell} \|\bff\|_{H_\rho^{2\ell}(\R^N,\PZ;D)} + C_{\bar{T},\ell} n^{-\ell}\|\bff\|_{H_\rho^{2\ell}(\R^N,\PZ;H)}, 
    \end{align*}
    from which the second inequality is obtained for $\bff=\bfu_0$.
\end{proof}

\begin{remark}
    In \eqref{eq:E5estimate}, the estimate of $E_5$ can be improved to an estimate in terms of the weaker norm $\|\bff\|_{H_\rho^{2\ell}(\R^N,\PZ;H)}$ by using Proposition \ref{prop:SobolevEstBfA}\ref{item:bfTanalyticEst} to estimate $\tau$ times the second term in \eqref{eq:boxnEstimate}. However, this does not lead to a stronger result due to the remaining terms $E_1$ to $E_4$.
\end{remark}

\subsection{Joint convergence rate of the full discretization error}
\label{subsec:fullErrorEstimate}

The last step required for an estimate of the full discretization error is to relate the graph norm of $\fraku_{0n}$ to the norm of $\bfu_0$ in a suitable subspace of $\bfH$ in order to make use of the semi-discretization error estimate in space-time of Subsection \ref{subsec:randomSpaceTime}.

 \begin{lemma}
\label{lem:DfrakAnalpha}
   Suppose that Assumption \ref{ass:coeffCond} holds for some $\ols\geq 1$ and $\ell=0$. 
    Let $0<q\le \ols$, $n \in \N_0$, $\bfw \in \bfH$, and $\frakw_n \ce (\widehat{\bfw}_\gamma)_{|\gamma|\le n}$. 
   Set $D^q\ce \dom(A_z^q)$. Then for $\bfw\in L_2(\R^N,\PZ;D^q)$, we have $\frakw_n\in \dom(\frakA_n^q)$ and there exists $C_q\geq 0$ independent of $\bfw$ and $n$ such that 
    \[
        \norm{\frakw_n}_{\frakA_n^q} \leq C_q \norm{\bfw}_{L_2(\R^N,\PZ;D^q)}.
    \]
\end{lemma}

\begin{proof}
    (i)
    We first let $q=1$. From $\dom(\bfA)=L_2(\R^N,\PZ;D)$ we conclude $\bfR_n\bfw \in \dom(\bfA) \seq \bfV$ and $\frakw_n \in \frakV_n$. Let $\frakv_n =(v_\gamma)_{|\gamma|\le n} \in \frakV_n$ and define $\bfv_n \ce \sum_{|\gamma|\le n} v_\gamma \Phi_\gamma \in \calP_n^N \ot V$. Then, from \eqref{eq:frakformIsBfform}, $\bfform_n \sim \bfR_n\bfA\bfR_n$, and $\bfR_n^2=\bfR_n$ we conclude
    \begin{align*}
        \frakform_n(\frakw_n,\frakv_n) &= \bfform_n(\bfR_n\bfw, \bfv_n) = \sp{\bfR_n\bfA\bfR_n^2\bfw}{\bfv_n}_\bfH= \sp{\bfR_n\bfA\bfR_n\bfw}{\bfv_n}_{\calP_n^N \ot H} = \sp{((\widehat{\bfR_n\bfA\bfR_n\bfw})_\gamma)_{|\gamma|\le n}}{\frakv_n}_{\frakH_n}
    \end{align*}
     by Plancherel's theorem.
    Hence, $\frakw_n \in \dom(\frakA_n)$ and $\frakA_n \frakw_n = ((\widehat{\bfR_n\bfA\bfR_n\bfw})_\gamma)_{|\gamma|\le n}$. Applying Parseval's identity twice and Proposition \ref{prop:SobolevEstBfA}\ref{item:bfAfSobolevnorm} with $q=0$, we conclude
    \begin{align*}
        \norm{\frakA_n \frakw_n}_{\frakH_n} &= \norm{\bfR_n \bfA\bfR_n \bfw}_{\bfH} \leq \norm{\bfA\bfR_n \bfw}_{\bfH} \le C \|\bfR_n\bfw\|_{L_2(\R^N,\PZ;D)}
        \le C \|\bfw\|_{L_2(\R^N,\PZ;D)}.
    \end{align*}
    (ii) Let $\N \ni m \le \ols$. Repeating the argument of step (i), we obtain $\frakw_n \in \dom(\frakA_n^m)$ for all $\bfw \in \dom(\bfA^m)$ and for some $C_m \ge 0$,
    \begin{equation}
    \label{eq:frakAnmpower}
        \|\frakA_n^m \frakw_n\|_{\frakH_n}  = \|(\bfR_n \bfA \bfR_n)^m \bfw\|_\bfH \leq C_m \|\bfw\|_{\dom(\bfA^m)}.
    \end{equation}
    (iii) Let $0<q \le \ols$ and choose $\N\ni m\geq q$. Due to symmetry of the form, $\frakA_n$ is self-adjoint and thus $\frakA_n^m$ is m-accretive.  A generalization of the Heinz inequality \cite{KatoHeinzIneq1961} then yields the statement for $\frakA_n^q$ via interpolation. Note that $\dom(\bfA^q) = L_2(\R^N,\PZ;D^q)$.
\end{proof}

We can now state our main result on the joint convergence rate for the approximation of random evolution equations. We recall the notions of rational schemes and A-stability from Definition \ref{def:rational_time_discretization}.

\begin{theorem}
\label{thm:jointrate_symmetric}
    Let $\ell \in \N$ and $\bar{T}>0$. Suppose that Assumption \ref{ass:coeffCond} holds for some $\ols\geq 1$ and $2\ell$ and that $V \hra H$ compactly.      
    Suppose that the space discretization converges with order $\px>0$ on $D^s\ce \dom(A_z^s)$ for some $0<s\leq \ols$ for the stationary problem. Further, suppose that the time discretization methods $\frakF_{n,m}$ are induced by an A-stable rational function $r$ for all $n\in\N_0$, $m\in\N$, and let $\pt>0$ be the order of convergence of the time discretizations.
    
    Let $\bfu$ be the mild solution of \eqref{eq:randomACP} with $\bfu_0 \in H_\rho^{2\ell}(\R^N,\PZ;D) \cap L_2(\R^N,\PZ;D^{\max\{s+1,\pt\}})$ where $\max\{s+1,\pt\} \leq \ols$ and, for $n,m,k\in\N$, let $\bfJ_{n,m}$ and $\bfu_{n,m,k}$ be as in \eqref{eq:defbfJnmPnm} and \eqref{eq:defbfunmk} via time discretization methods $\frakF_{n,m}$, respectively.
    
    Then there exist $ C_{\bar{T},\ell},C_{\bar{T},s,\pt}\geq 0$, $\tau_0>0$ such that for $\max_{i=1,\ldots,N_k} \tau_k^i \leq \tau_0$, 
    \begin{align*}
        \|\bfJ_{n,m}\bfu_{n,m,k}(t)-\bfu(t)\|_\bfH
        &\le C_{\bar{T},\ell} n^{-\ell} \|\bfu_0\|_{H_\rho^{2\ell}(\R^N,\PZ;D)}\\
        &\phantom{\le } +C_{\bar{T},s,\pt} \bigg(m^{-\px} + \Big(\max_{i=1,\ldots,N_k} \tau_k^i\Big)^{\pt}\bigg) \|\bfu_{0}\|_{L_2(\R^N,\PZ;D^{\max\{s+1,\pt\}})}
    \end{align*}
    for all $n,m,k\in\N$ and $t \in \calT_k$.
\end{theorem}
\begin{proof}
    Abbreviate $\tau_{\max,k}\ce \max_{i=1,\ldots,N_k} \tau_k^i$. Then Proposition \ref{prop:convergence_rate_random_symmetricPartI} and Theorem \ref{thm:errorunuTrotterKato} imply that 
    \begin{align*}
        \|\bfJ_{n,m}\bfu_{n,m,k}(t)-\bfu(t)\|_\bfH &\le \|\bfJ_{n,m}\bfu_{n,m,k}(t)-\bfu_n(t)\|_\bfH + \|\bfu_n(t)-\bfu(t)\|_\bfH\\
        & \le C_{\bar{T},s,\pt} \big(m^{-\px} + (\tau_{\max,k})^{\pt}\big) \|\fraku_{0n}\|_{\frakA_n^{\max\{s+1,\pt\}}} + C_{\bar{T},\ell} n^{-\ell} \|\bfu_0\|_{H_\rho^{2\ell}(\R^N,\PZ;D)}.
    \end{align*}
    The corollary then follows from Lemma \ref{lem:DfrakAnalpha} applied to $\bfw=\bfu_0$ and $q=\max\{s+1,\pt\}$, which gives
    \[
        \|\fraku_{0n}\|_{\frakA_n^{\max\{s+1,\pt\}}} \leq C_q \norm{\bfu_0}_{L_2(\R^N,\PZ;D^{\max\{s+1,\pt\}})}. \qedhere
    \]
\end{proof}

\begin{remark}
    Let us comment on the case of deterministic initial values, i.e., $u_0\in D^{q}$ for some $q>0$. Then $\bfu_0\ce \1_{\R^N}\otimes u_0 \in H^{2\ell}_\rho(\R^N,\PZ;D^q)$ for all $\ell\in\N_0$ and
    \[\norm{\bfu_0}_{H^{2\ell}_\rho(\R^N,\PZ;D^q)} = \norm{u_0}_{D^q}.\]
    Hence, we naturally recover the results of Section \ref{sec:jointRate_deterministic}.
\end{remark}

\begin{remark}
    Joint convergence rates can also be obtained for non-symmetric forms $(\form_z)_{z \in \R^N}$. This requires a technically more intricate proof and an additional assumption: Suppose that in addition to Assumption \ref{ass:coeffCond} for $2\ell$ and $\ols\ge 2$, there exists a constant $C_{\ell,\dom}\ge0$ such that
    \[ 
        \|\rho(z)^{\alpha/2} \partial_z^\alpha A_zu\|_D \le C_{\ell,\dom} \|u\|_{D^2}
    \]
    for $\PZ$-almost every $z \in \R^N$ and all $u\in D$ and $\alpha \in \calN$ with $|\alpha|\le 2\ell$. Then, non-symmetric versions of Theorem \ref{thm:errorunuTrotterKato} and Lemma \ref{lem:DfrakAnalpha} can be established (see \cite[Theorem~3.72, Lemma 3.73b]{thesisKatharina}). An analogue of Theorem \ref{thm:jointrate_symmetric} is obtained provided that for some $\varepsilon>0$, $\bfu_0 \in H_\rho^{2\ell}(\R^N,\PZ;D^{1+\varepsilon})$ as well as $\bfu_0 \in L_2(\R^N,\PZ;D^{1+\varepsilon+\max\{s,\pt\}})$, and the time discretization is done via the implicit Euler method ($\pt=1$) or the Crank--Nicolson method ($\pt=2$), cf. \cite[Theorem~3.75]{thesisKatharina}.
\end{remark}

\begin{remark}
    It is well-known from other contexts that holomorphic dependence on the random inputs gives rise to exponential rates of convergence for PCE, see e.g.\ \cite[Subsection 3.3.2]{Xiu2010}. We also expect  exponential rates of convergence for PCE in our situation, as long as both $z\mapsto \form_z(u,v)$ for all $u,v\in V$ and $z\mapsto \bfu_0(z)$ are holomorphic, and we can control higher-order Sobolev norms of these mappings. We do not follow this direction here, but postpone it to future work.
\end{remark}

\begin{remark}
\label{rem:nonlinear_extensions}
    Let us comment on possible extensions of Theorem \ref{thm:jointrate_symmetric} to non-linear situations.

    \begin{enumerate}[ref=({\alph*})]
        \item\label{rem:nonlinear_extensions:item:semilinear}
            We first consider the semilinear case. For $z\in\R^N$ let $F_z\from H\to H$, and let $\bfF \from \bfH\to \bfH$ be the corresponding multiplication operator, which we assume to be Lipschitz continuous on $\bfH$, and such that it maps $H^{2\ell}_\rho(\R^N,\PZ;D)$ continuously to itself with linear growth. 

            Then we consider the abstract Cauchy problem
            \[
                \bfu'(t) = -\bfA \bfu(t)+\bfF(\bfu(t)),\quad \bfu(0)=\bfu_0
            \]
        with mild solution $\bfu$ given by
        \[\bfu(t)=\bfT(t)\bfu_0 + \int_0^t \bfT(t-s)\bfF(\bfu(s))\,\rmd s.\]
        We may approximate this problem by
        \[
            \bfu_n'(t) = -\bfR_n\bfA\bfR_n \bfu(t)+\bfF_n(\bfu_n(t)),\quad \bfu(0)=\bfu_{0,n} \ce \bfR_n\bfu_0
        \]
        with
        $\bfF_n \ce \bfR_n\bfF$ and mild solutions $\bfu_n$ given by
        \[\bfu_n(t)=\bfT_n(t)\bfu_{0,n} + \int_0^t \bfT_n(t-s)\bfF_n(\bfu_n(s))\,\rmd s.\]
        Then we can treat the semi-discretization in randomness analogously and afterwards perform space-time discretizations to obtain a version of Theorem \ref{thm:jointrate_symmetric}.
        
        \item 
        In case $H=L_2(G)$ for some (bounded Lipschitz) domain $G\subseteq\R^d$, a special case of \ref{rem:nonlinear_extensions:item:semilinear} is given by Nemytskii operators, that is, for $z\in\R^N$ let $f_z\from G\times \C\to \C$ satisfy the Carath\'{e}odory conditions (see, e.g.\cite[Subsection 9.3.4]{RenardyRogers1993}) and set $F_z(v):=f_z(\cdot,v(\cdot))$. At least in the autonomous situation, $f_z\from \C\to\C$ (no explicit dependence on $G$) and $F_z(v):=f_z(v) = f_z\circ v$, characterizations of mapping properties on Sobolev spaces are well-known \cite{Isaia2022}, which are needed for the mapping properties with values in $D$. Note that Sobolev regularity of $\bfF$ with respect to $z$ depends only on the corresponding Sobolev regularity of $z\mapsto f_z$. 

        \item 
        A more general situation is given by monotone operators $A_z$ for $z\in \R^N$. Here, the generated semigroups turn out to be non-linear, but still contractive. Moreover, a basic version of the Trotter--Kato theorem is available, see, e.g., \cite{Brezis1975} and references therein.
        As a first step towards convergence rates in the monotone setting, a quantified version of the Trotter--Kato theorem for maximal monotone operators would be beneficial.
        The authors leave this to future work.
    \end{enumerate}
\end{remark}

\section{Application to Random Anisotropic Diffusion}
\label{sec:applications}

Let $G\subseteq \R^2$ be open, bounded, convex, and polygonal, and let $H\ce L_2(G)$ and $V\ce H_0^1(G)$.
Let $(\Omega,\calF,\P)$ be a probability space, $Z\from \Omega\to\R^N$ a random variable with independent components and distribution $\P_Z$. Assume that each of the components is standard normally distributed, Beta-distributed, or Gamma-distributed; cf. Assumption \ref{ass:Zdistribution}. Let $\R^N\times G \ni (z,x)\mapsto M_z(x)\in \K^{2\times 2}$ such that $M_z(x)$ is Hermitian for all $x\in G$ and $\PZ$-a.e.\ $z\in\R^N$, and assume there exist $\kappa,K>0$ such that
\begin{equation}
\label{eq:coeffMatrixBdd}
    \kappa \leq M_z(x) \leq K \quad(x\in G, \PZ\text{-a.e.}\,z\in\R^N).
\end{equation}
Moreover, let $z\mapsto \int_G M_z(x)_{jk} h(x)\,\rmd x$ be measurable for all $j,k\in\{1,2\}$ and $h\in L_1(G)$.
For those $z\in\R^N$ for which \eqref{eq:coeffMatrixBdd} holds, define $\form_z\from V\times V\to \K$ by
\[
    \form_z(u,v) \ce \int_G M_z(x)\grad u(x) \cdot \overline{\grad v(x)}\,\rmd x \quad(u,v\in H_0^1(G)),
\]
and let $a_z(u,v)=0$ otherwise. Here and in the following, $\grad$ and $\diverg$ are acting on the spatial coordinates $x$ only.
Then $z\mapsto \form_z(u)$ is measurable for all $u\in V$ and 
\[
    \kappa \norm{u}_V^2 \leq \form_z(u), \quad \abs{\form_z(u,v)} \leq K \norm{u}_V\norm{v}_V \quad(u,v\in V, \PZ\text{-a.e.}\,z\in\R^N).
\]
That is, $(a_z)_{z \in \R^N}$ is $\PZ$-almost surely uniformly bounded and coercive, as required in Assumption \ref{ass:formUnifBdCoercive}.

For $z\in\R^N$ let $A_z$ be the self-adjoint operator in $H$ associated with $\form_z$. Note that $[0,\infty)$ is contained in the resolvent set of $-A_z$ for $\PZ$-a.e.\ $z\in\R^N$.

\begin{remark}
    Provided that $M_z\in W_\infty^1(G)$ for $\PZ$-a.e. $z\in\R^N$, we $\PZ$-almost surely have
    \begin{align*}
        \dom(A_z)  = H_0^1(G)\cap H^2(G),\quad
        A_z u  = -\div M_z(x) \grad u.
    \end{align*}
    Hence, we can set $D \ce H_0^1(G) \cap H^2(G)$ equipped with the $H^2$-norm. 
\end{remark}

To establish the estimate of derivatives of the form in $z$ required for Assumption \ref{ass:coeffCond} to hold, smoothness of the coefficient matrix is required. We endow $\Kzz$ with the spectral norm $\|\cdot\|_2$ and $\K^2$ with the Euclidean norm. Denote by $\Div\from \Kzz\to\K^2$ the row-wise divergence operator so that $\Div(b(\cdot)^T)$ calculates the divergence of a matrix-valued function $b\in C^1(G;\Kzz)$ column-wise.

\begin{definition}
\label{def:Crhol1Ex}
    For $\ell \in \N_0$, define $C_\rho^{\ell,1}(\R^N \times G;\Kzz)$ as the space of all $f\from \R^N \times G \to \Kzz$ such that $f(\cdot,x) \in C^\ell(\R^N;\Kzz)$ for all $x \in G$, $f(z,\cdot) \in C^1(G;\Kzz)$ for $\PZ$-a.e.\ $z \in \R^N$ and there is $C_\ell \ge 0$ such that for all $\alpha \in \N_0^N$ with $|\alpha|\le \ell$, we have 
    \begin{align*}
        \max\Big\{\PZ\text{-}\esssup_{z \in \R^N} \big\|\rho(z)^{\alpha/2} \Div \big((\partial_z^\alpha f(z,\cdot))^T\big)\big\|_{L_\infty(G;\K^2)},\,
        \PZ\text{-}\esssup_{z \in \R^N} \|\rho(z)^{\alpha/2} \partial_z^\alpha f(z,\cdot)\|_{L_\infty(G;\Kzz)}\Big\} \le C_\ell.
    \end{align*}
    Inductively, for $k\in\N$ let $C_\rho^{\ell,k+1}(\R^N \times G;\Kzz)$ denote the space of all $f \in C_\rho^{\ell,k}(\R^N \times G;\Kzz)$ such that $f(z,\cdot) \in C^{k+1}(G;\Kzz)$ for $\PZ$-a.e.\ $z \in \R^N$ and $[(z,x) \mapsto \partial_x^{e_j} f(z,x)] \in C_\rho^{\ell,k}(\R^N \times G; \Kzz)$ for $j=1,2$.
\end{definition}

Assume that $[(z,x) \mapsto M_z(x)]\in C_\rho^{\ell,1}(\R^N \times G;\K^{2\times 2})$. Then the graph norms of $A_z$ are equivalent for $\PZ$-a.e.\ $z \in \R^N$ since there are $C,c \ge 0$ such that for all $u \in \dom(A_z)$ and $\PZ$-a.e.\ $z \in \R^N$ we have 
\begin{equation}
\label{eq:normEquivaExample}
    \|u\|_{H^2} \le c \|A_zu\|_{L_2} \le C \|u\|_{H^2}.
\end{equation}
Indeed, the first inequality follows from \cite[Theorems~3.2.1.2~and~3.1.3.1]{Grisvard2011} because $G$ is convex. Since $[(z,x) \mapsto M_z(x)]$ is continuously differentiable in $x$, for $u \in \dom(A_z)$, we can rewrite
\[
    A_z u = -\div( M_z(\cdot) \grad u)  = -\Div\big( M_z(\cdot)^T\big)\cdot \grad u -M_z(\cdot):\Hess u,
\]
where $\Hess u$ is the Hessian matrix of $u$ and $B:C=\sum_{i,j=1}^2 b_{i,j}{c_{i,j}}$ is the sum of the entry-wise product of two matrices $B,C\in \Kzz$. By assumption for $\alpha=0$ and due to the definition of $C_\rho^{\ell,1}(\R^N \times G;\Kzz)$, we can thus estimate
\begin{align}
\label{eq:DivHessExample}
        \|A_z u\|_{L_2} &\le \norm{\Div\big( M_z(\cdot)^T\big)\cdot \grad u}_{L_2}
        + \norm{M_z(\cdot):\Hess u}_{L_2}\nonumber\\
        &\le \sup_{x\in G}\norm{\Div\big( M_z(x)^T\big)}_\infty \norm{ \grad u}_{L_2}+\bigg(\sup_{x \in G} \max_{i,j \in \{1,2\}} \abs{(M_z(x))_{i,j}} \bigg) \bigg\|\sum_{i,j=1}^2 \partial_{ij}u\bigg\|_{L_2}\nonumber\\
        &\le C_0 \|\grad u\|_{L_2} + C_0 \sum_{i,j=1}^2 \|\partial_{ij}u\|_{L_2} \le \sqrt{6} C_0 \|u\|_{H^2}
\end{align}
for all $u \in \dom(A_z)$ and $\PZ$-a.e.\ $z\in \R^N$, where $\|\cdot\|_\infty$ denotes the maximum norm in $\K^2$. This yields the second inequality.
Consequently, Assumption \ref{ass:DAzDconstant} is satisfied for $\ols=1$. However, we would like to have the assumption satisfied for some $\ols> 1$. This can be shown under additional assumptions on the diffusion coefficients. In case we want $1< \ols < 5/4$, reasoning as in \cite{HajdukRobinson2021} we observe that for $[(z,x) \mapsto M_z(x)]\in C_\rho^{\ell,2}(\R^N \times G;\K^{2\times 2})$ we have $D^{\ols} = H_0^1(G)  \cap H^{2\ols}(G)$ with corresponding graph norms of $A_z^{\ols}$ being equivalent, i.e., Assumption \ref{ass:DAzDconstant} is satisfied. In case we want $\ols\geq \frac{5}{4}$, further regularity on the coeeficients and, due to the Dirichlet boundary conditions, further assumptions on the traces of the coefficients may be needed. We will consider the case $\ols=2$. In order to obtain that the domains $\dom(A_z^2)$ are $\PZ$-almost surely constant, we may assume that $[(z,x) \mapsto M_z(x)]\in C_\rho^{\ell,3}(\R^N \times G;\K^{2\times 2})$ as well as $\tr M_z$ and $\tr \Div(M_z(\cdot)^T)$ are constant for $\PZ$-almost every $z\in\R^N$, where $\tr$ denotes the trace operator for $G$.
Then, a similar reasoning as above yields that Assumption \ref{ass:DAzDconstant} is satisfied for $\ols=2$, and the graph norms of $A_z^2$ can be compared to the $H^4(G)$-norm.

To verify the form estimate from Assumption \ref{ass:coeffCond}, it suffices to show that there is $C_\ell\ge 0$ such that
\begin{equation}
\label{eq:derivAzExample}
    \big\|\rho(z)^{\alpha/2}\partial_z^\alpha A_z u\big\|_{L_2} \le C_\ell \|u\|_{H^2}
\end{equation}
for all $\alpha \in \N_0^N$ with $\abs{\alpha}\le \ell$ and $\PZ$-a.e.\ $z\in \R^N$. Indeed, for $u \in \dom(\partial_z^\alpha A_z) \cap H^2(G)$ and $v \in H_0^1(G)$, \eqref{eq:derivAzExample} implies
\begin{align*}
    \big|\rho(z)^{\alpha/2}\partial_z^\alpha \form_z(u,v)\big| &= \big|\big(\rho(z)^{\alpha/2} \partial_z^\alpha A_zu\big|v\big)\big|
    \le \big\|\rho(z)^{\alpha/2} \partial_z^\alpha A_zu\big\|_{L_2}\norm{v}_{L_2}\le C_\ell \|u\|_{H^2}\|v\|_{L_2} = C_\ell \|u\|_{D}\|v\|_{H}.
\end{align*}
Now, we show \eqref{eq:derivAzExample}. By assumption, $[(z,x)\mapsto M_z(x)]$ is continuously differentiable in $x$ (and $\ell$ times continuously differentiable in $z$) so that
\[
    \partial_z^\alpha A_z u = -\div( \partial_z^\alpha M_z(\cdot) \grad u)  = -\Div\big( \partial_z^\alpha M_z(\cdot)^T\big)\cdot \grad u -\big(\partial_z^\alpha M_z(\cdot)\big):\Hess u.
\]
Analogously to \eqref{eq:DivHessExample}, we infer that for $[(z,x)\mapsto M_z(x)] \in C_\rho^{\ell,1}(\R^N\times G;\Kzz)$,
\begin{align*}
     \big\|\rho(z)^{\alpha/2}\partial_z^\alpha A_z u\big\|_{L_2}
     &\le \big\|\rho(z)^{\alpha/2}\Div\big( (\partial_z^\alpha M_z(\cdot))^T\big)\cdot \grad u\big\|_{L_2}
        + \big\|\rho(z)^{\alpha/2}\big(\partial_z^\alpha M_z(\cdot)\big):\Hess u\big\|_{L_2} \\
        & \le \sqrt{6} C_\ell \|u\|_{H^2}.
\end{align*}
In conclusion, Assumption \ref{ass:coeffCond} is satisfied for $\ols$ and $\ell \in \N$ in any of the following cases:
\begin{enumerate}[label=(\roman*)]
    \item\label{item:alphaOne} $[(z,x)\mapsto M_z(x)] \in C_\rho^{\ell,1}(\R^N\times G;\Kzz)$ and $\ols=1$,
    \item\label{item:alphaFiveFour} $[(z,x)\mapsto M_z(x)] \in C_\rho^{\ell,2}(\R^N\times G;\Kzz)$ and $1<\ols<5/4$,
    \item\label{item:alphaTwo} $[(z,x)\mapsto M_z(x)] \in C_\rho^{\ell,3}(\R^N\times G;\Kzz)$ and $\frac54 \le \ols \le 2$  as well as $\tr M_z$ and $\tr \Div(M_z(\cdot)^T)$ are $\PZ$-almost surely constant.
\end{enumerate} 

As space discretization, we employ the quadratic finite element method (FEM). More precisely, we consider quasi-uniform triangulations $G_h$ of $G$, $h>0$, consisting of triangles with a circumference no larger than $h=\frac{2}{m}$. Let
\begin{equation*}
    V_m \ce \{u \in H_0^1(G):~u\vert_L \in \calP_2^2~~\forall\,L \in G_h\}
\end{equation*}
be the corresponding quadratic triangular finite element (FE) space. Due to the finite dimension of $V_m$, the spaces $H_m$ and $V_m$ coincide. For $m \in \N$, $P_m$ is the $L_2(G)$-orthogonal projection from $L_2(G)$ onto $V_m$. These are not to be confused with the $V$-orthogonal projections from $V=H_0^1(G)$ onto $V_m$ yielding the FE approximation $u_h$ of $u$ from $V_m$. Then the error estimate \cite[Satz~6.4]{braessFEM} 
\begin{equation}
\label{eq:FEMerrest2D}
    \|u-u_h\|_{H^r(G)} \le C h^{\ell-r} |u|_{H^\ell(G)},\quad(r\in \{0,\ldots,\ell\}),
\end{equation}
holds true for some $C \ge 0$ for $\ell =2, 3$ and all $u \in H^\ell(G)$, where $|\cdot|_{H^\ell(G)}$ denotes the standard $H^\ell(G)$-seminorm. Setting $r=1$ and $\ell\in \{2,3\}$ yields linear and quadratic decay of $\|u-u_h\|_{H^1(G)}$ in $m$ for $u \in H^2(G)$ and $u \in H^3(G)$, respectively, since $h=\frac{2}{m}$.

From these FEM estimates for the stationary problem, we can deduce a spatial convergence rate for the time-dependent problem via Theorem \ref{thm:result_evolution}. To this end, we determine the decay rates $p_1(s)$, $0<s\le\ols-1$, and $p_2$ of $(\gamma_m(D^s))_{m \in \N}$ and $(\gamma_m^*(H))_{m \in \N}$, respectively. We let $\ols> 1$ and start with the latter.

Firstly, $\form_z$ is symmetric $\PZ$-almost surely since we assumed $M_z(x)$ to be Hermitian for all $x \in G$ and $\PZ$-a.e.\ $z \in \R^N$. Secondly, $\calA_z^{-1}H=\dom(A_z)=D$ $\PZ$-almost surely by definition of $A_z$ as the restriction of $\calA_z$ to the preimage of $H$ under $\calA_z$. From the FE estimate \eqref{eq:FEMerrest2D} with $\ell=2$ and the almost sure equivalence \eqref{eq:normEquivaExample} of the $H^2$-norm and $\|A_z\cdot\|_{L_2}$, we deduce that $\PZ$-almost surely
\[
    \inf_{v \in V_m} \|u-v\|_{H^1} \le \|u-u_h\|_{H^1} \le \frac{C}{m} |u|_{H^2} \le \frac{cC}{m} \|A_zu\|_{L_2}
\]
provided that $[(z,x)\mapsto M_z(x)] \in C_\rho^{\ell,1}(\R^N\times G;\Kzz)$. Therefore, the decay rate $p_2=1$ is obtained.

Next, we calculate $p_1(s)$ depending on the value of $0<s \le \ols-1$. Since this implies $\ols>1$, we are only interested in the cases \ref{item:alphaFiveFour} and \ref{item:alphaTwo} listed above. Calculating $p_1(s)$ requires establishing an estimate of the form
\[
    \inf_{u \in V_m} \|u-v\|_{H^1} \le \frac{C_{\gamma,D^s}}{m^{p_1(s)}} \|A_z u\|_{D^{s}}\quad (u \in \dom(A_z) \text{ with }A_zu \in \dom(A_z^{s})=D^s)
\]
for $\PZ$-a.e.\ $z \in \R^N$.
Repeating the estimate performed for $p_2$, we obtain $p_1(s)=1$ for all $0<s\le \ols-1$ since $D^s \hra L_2(G)$. However, this estimate does not take the increased smoothness of $u$ for larger $s$ into account. Provided that $s \ge \frac{1}{2}$, we can make use of the smoothness of $u \in \dom(A_z)$ such that $A_z u \in \dom(A_z^s)$ almost surely to obtain a higher decay rate.

Suppose that $\ols \ge \frac{3}{2}$ and $s \ge \frac{1}{2}$. Then $u \in \dom(A_z)=H^2(G) \cap H_0^1(G)$ and $A_zu \in D^{1/2}$ $\PZ$-almost surely. The higher-order FEM estimate \eqref{eq:FEMerrest2D} with $\ell=3$ and the equivalence of the $H^2$-norm and the graph norm; cf. \eqref{eq:normEquivaExample}, imply
\begin{align}
\label{eq:FEMalpha12FirstPart}
    \inf_{v \in V_m} \|u-v\|_{H^1}^2 &\le \|u-u_h\|_{H^1}^2 \le \Big(\frac{C}{m^2}\Big)^2 |u|_{H^3}^2 
    = \Big(\frac{C}{m^2}\Big)^2 \left(|\partial_1 u|_{H^2}^2+|\partial_2 u|_{H^2}^2\right)\nonumber\\
    &\le c^2\Big(\frac{C}{m^2}\Big)^2 \left(\|A_z\partial_1 u\|_{L_2}^2+\|A_z\partial_2 u\|_{L_2}^2\right)\quad(\PZ\text{-a.e.}\,z\in\R^N).
\end{align}
In order to estimate $\|A_z\partial_j u\|_{L_2}^2$ for $j=1,2$, we need $[(z,x) \mapsto M_z(x)] \in C_\rho^{\ell,2}(G;\Kzz)$. In particular, $M_z(\cdot) \in C_b^2(G;\Kzz)$ for $\PZ$-a.e.\ $z \in \R^N$. This smoothness allows us to explicitly compute
\[
    \partial_j(\Div(b)\grad u) = \Div(b)\grad(\partial_j u) + \Div(\partial_j b) \grad u
\]
for $j=1,2$, $u \in H^3(G)$, and coefficients $b\in C^2(G;\K^{2\times 2})$. Hence, $\PZ$-almost surely
\begin{align*}
    A_z\partial_j u &= -\div(M_z(\cdot) \grad (\partial_j u)) = -\Div (M_z(\cdot)^T) \grad \partial_j u - M_z(\cdot):\Hess \partial_j u\\
    &= -\partial_j ( \Div(M_z(\cdot)^T) \grad u) + \Div(\partial_j (M_z(\cdot)^T)) \grad u - \partial_j (M_z(\cdot):\Hess u) + (\partial_j M_z(\cdot)):\Hess u\\
    &= \partial_j (A_zu) + \Div(\partial_j (M_z(\cdot)^T)) \grad u+ (\partial_j M_z(\cdot)):\Hess u.
\end{align*}
By assumption, $\partial_j (M_z(\cdot)^T)$ and $\partial_j M_z(\cdot)$ have bounded derivatives, so that an analogous estimate to \eqref{eq:DivHessExample} results in
\begin{align*}
    &\|A_z\partial_j u\|_{L_2}
    \le  \|\partial_j (A_zu)\|_{L_2} + C\|u\|_{H^2} \quad (u \in H^3(G),\, j=1,2).
\end{align*}
Thus, with $C$ denoting a generic constant taking different values,
\begin{align*}
    \sum_{j=1}^2\|A_z\partial_j u\|_{L_2}^2
    &\le C\sum_{j=1}^2 \|\partial_j (A_zu)\|_{L_2}^2 + C\|u\|_{H^2}^2 = C \|\grad(A_zu)\|_{L_2}^2 + C\|u\|_{H^2}^2\\
    &\le C \|A_z^{1/2}A_z u\|_{L_2}^2 + C\|A_zu\|_{L_2}^2 \le C \|A_z u \|_{\dom(A_z^{1/2})}^2,
\end{align*}
where in the penultimate inequality we have used the Kato square root property for self-adjoint elliptic operators \cite[Theorem~1]{Katosquareroot} for the first term and \eqref{eq:normEquivaExample} for the second one. Combining this estimate with \eqref{eq:FEMalpha12FirstPart} yields $p_1(\frac12)=2$, and thus $p_1(s)=2$ for all $s \ge \frac12$.

As a consequence of Theorem \ref{thm:result_evolution}, we obtain a spatial convergence rate depending on the smoothness of the space $\Yx$ by summing $p_1(s)$ and $p_2$.

\begin{proposition}
\begin{enumerate}[label=(\alph*)]
    \item Suppose that $[(z,x) \mapsto M_z(x)]$ satisfies \ref{item:alphaFiveFour}. Then the space discretization defined here above converges on $\Yx = D^{s+1}$ with order $\px=2$ for $0< s < 1/4$.
    \item Suppose that $[(z,x) \mapsto M_z(x)]$ satisfies \ref{item:alphaTwo}. Then the space discretization defined here above converges on $\Yx = D^{s+1}$ with order $\px=2$ for $0< s < 1/2$ and with order $\px=3$ for $\frac{1}{2}\leq s\leq 1$.
\end{enumerate}
    
\end{proposition}

\begin{remark}
    Higher values of $s>\frac12$ do not result in higher values of $p_1(s)$ or $\px$. This restriction is due to quadratic FE spaces being used. To further increase $p_1(s)$ for large $s$, higher-order FE spaces, such as piecewise cubic polynomials on a quasi-uniform triangulation, have to be used. 
\end{remark}

Having verified Assumption \ref{ass:coeffCond} and calculated the spatial convergence rate for our example, we can apply Theorem \ref{thm:jointrate_symmetric} to obtain a joint convergence rate for the full discretization of the abstract Cauchy problem associated with $A_z$.

\begin{theorem}
\label{thm:exampleCoeffMatrix_smallalpha}
    Let $Z\from \Omega \to \R^N$ be a random variable satisfying Assumption \ref{ass:Zdistribution}. Suppose that $\R^N\times G \ni (z,x) \mapsto M_z(x) \in \Kzz$ satisfies \eqref{eq:coeffMatrixBdd} and is contained in $C_\rho^{2\ell,2}(\R^N\times G;\Kzz)$ for some $\ell \in \N$. Let $0<s<1/4$ and  
    $\bfu=(u_z)_{z \in \R^N}$ be the mild solution of 
    \[
        u_z'(t)=-\div M_z(x) \grad u_z(t) \quad(t>0),\quad u_z(0)=u_{0,z}
    \]
    with $\bfu_0=(u_{0,z})_{z \in \R^N} \in H_\rho^{2\ell}(\R^N,\PZ;D) \cap L_2(\R^N,\PZ;D^{s+1})$. For $n,m,k\in\N$, let $\bfu_{n,m,k}$ as in \eqref{eq:defbfunmk}, where a quadratic finite element method and implicit Euler are used for discretization in space and time, respectively.\\
    Then there exist $ C_{\bar{T},\ell},C_{\bar{T},s}\geq 0$, $\tau_0>0$ such that for $\max_{i=1,\ldots,N_k} \tau_k^i \leq \tau_0$, 
    \begin{align*}
        \|\bfu_{n,m,k}(t)-\bfu(t)\|_\bfH
        &\le C_{\bar{T},\ell} n^{-\ell} \|\bfu_0\|_{H_\rho^{2\ell}(\R^N,\PZ;D)} +C_{\bar{T},s} \bigg(m^{-2} + \max_{i=1,\ldots,N_k} \tau_k^i\bigg) \|\bfu_{0}\|_{L_2(\R^N,\PZ;D^{s+1})}
    \end{align*}
for all $n,m,k\in\N$ and $t \in \calT_k$.
\end{theorem}

\begin{theorem}
\label{thm:exampleCoeffMatrix}
    Let $Z\from \Omega \to \R^N$ be a random variable satisfying Assumption \ref{ass:Zdistribution}. Suppose that $\R^N\times G \ni (z,x) \mapsto M_z(x) \in \Kzz$ satisfies \eqref{eq:coeffMatrixBdd} and is contained in $C_\rho^{2\ell,3}(\R^N\times G;\Kzz)$ for some $\ell \in \N$ such that $\tr M_z$ and $\tr \Div(M_z(\cdot)^T)$ are constant for $\PZ$-almost every $z\in\R^N$. Let $0<s\leq 1$ and  
    $\bfu=(u_z)_{z \in \R^N}$ be the mild solution of 
    \[
        u_z'(t)=-\div M_z(x) \grad u_z(t) \quad(t>0),\quad u_z(0)=u_{0,z}
    \]
    with $\bfu_0=(u_{0,z})_{z \in \R^N} \in H_\rho^{2\ell}(\R^N,\PZ;D) \cap L_2(\R^N,\PZ;D^{s+1})$. For $n,m,k\in\N$, let $\bfu_{n,m,k}$ as in \eqref{eq:defbfunmk}, where a quadratic finite element method and implicit Euler are used for discretization in space and time, respectively.\\
    Then there exist $ C_{\bar{T},\ell},C_{\bar{T},s}\geq 0$, $\tau_0>0$ such that for $\max_{i=1,\ldots,N_k} \tau_k^i \leq \tau_0$, 
    \begin{align*}
        \|\bfu_{n,m,k}(t)-\bfu(t)\|_\bfH
        &\le C_{\bar{T},\ell} n^{-\ell} \|\bfu_0\|_{H_\rho^{2\ell}(\R^N,\PZ;D)} +C_{\bar{T},s} \bigg(m^{-\px} + \max_{i=1,\ldots,N_k} \tau_k^i\bigg) \|\bfu_{0}\|_{L_2(\R^N,\PZ;D^{s+1})}
    \end{align*}
    for all $n,m,k\in\N$ and $t \in \calT_k$ with $\px=3$ for $s \ge \frac12$ and $\px=2$ for $s \in (0,\frac12)$.
\end{theorem}

As a concrete example, we consider a random anisotropic diffusion. Let $Z\from \Omega \to \R$ be a standard normally distributed random variable (i.e., $N=1$). Let $G=(0,1)^2$ and for $z \in \R$ and $x \in G$, consider the coefficient matrix given by
\begin{alignat}{3}
\label{eq:logisticCoeff}
    M_z(x)&=f(z)\cdot \begin{pmatrix}g_1(x)&0\\0&g_2(x)\end{pmatrix},\quad &f(z)&= \frac{1}{1+e^{-z}}+1,\quad
    g_1(x)=1+\|x\|_2^2,\quad g_2(x)=3-\|x\|_2^2. 
\end{alignat}
We observe that $f(z)\in [1,2]$ for all $z \in \R$, $g_1(x),g_2(x)\in [1,3]$ for all $x \in (0,1)^2$. Hence, $\kappa \le M_z(x) \le K$ is satisfied for $\kappa=1$ and $K=6$. It remains to verify that $(z,x) \mapsto M_z(x)$ is in $C_\rho^{2\ell,2}(\R\times(0,1)^2;\Rzz)$ in order to apply Theorem \ref{thm:exampleCoeffMatrix_smallalpha}. 
Note that $f \in C^\infty(\R;\Rzz)$ and thus $a^\cdot(x)\in C^{2\ell}(\R;\Rzz)$ for any $\ell \in \N$.  Also, clearly, $M_z(\cdot) \in C^2((0,1)^2;\Rzz)$. Likewise, the product structure of $M_z(x)$ implies $\partial_x^{e_j} a^\cdot (x)\in C^{2\ell}(\R;\Rzz)$. To establish the bounds stated in Definition \ref{def:Crhol1Ex}, we first observe that derivatives of $f$ are bounded in $z$. Indeed, as a logistic function, the derivatives of $F \ce f-1$ can be expressed as a polynomial of $F$. Since all derivatives of $f$ and $F$ agree, $F(z)\in [0,1]$ for $z \in \R$, and continuous functions are bounded on bounded closed intervals,
\[
    c_\ell \ce \sup_{z \in\R} \max_{0\le i \le \ell} \abs{(\partial_z^i f)(z)} <\infty
\]
for all $\ell\in \N$. Consequently, for all $0 \le \alpha \le 2\ell$,
\begin{align*}
    \PZ\text{-}\esssup_{z \in \R} \|\partial_z^\alpha M_z(\cdot)\|_{L_\infty} &\le \Big(\sup_{z \in \R} \abs{(\partial_z^\alpha f)(z)}\Big) \bigg(\sup_{x\in (0,1)^2}\norm{\begin{pmatrix}g_1(x)&0\\0&g_2(x)\end{pmatrix}}_2 \bigg) \le 3c_{2\ell},
\end{align*}
where $\rho\equiv 1$ by choice of distribution of $Z$. Since for all $z\in \R$,
\begin{align*}
    \norm{\Div((\partial_z^\alpha M_z(\cdot))^T)}_{L_\infty} = \abs{\partial_z^\alpha f(z)}\cdot \sup_{x \in (0,1)^2}\norm{ \begin{pmatrix} 2x_1\\2x_2
    \end{pmatrix}}_2
    = 2\sqrt{2}\abs{\partial_z^\alpha f(z)} 
\end{align*}
it also holds that for $0 \le \alpha \le 2\ell$
\[
    \PZ\text{-}\esssup_{z \in \R}\norm{\Div((\partial_z^\alpha M_z(\cdot))^T)}_{L_\infty(G;\R^2)} \le 2 \sqrt{2}c_{2\ell}.
\]
Proceeding likewise, we conclude analogous statements to the last two ones for $M_z(\cdot)$ replaced by $\partial_x^{e_j}M_z(\cdot)$, $j=1,2$, with constant $2c_{2\ell}$ in both cases. In conclusion, the coefficients belong to $C_\rho^{2\ell,2}(\R\times(0,1)^2;\Rzz)$. Hence, we conclude convergence of arbitrary polynomial order in randomness for the random anisotropic diffusion problem. For simplicity, we fix $0<s<1/4$.

\begin{corollary}
    Adopt the assumptions and notation of Theorem \ref{thm:exampleCoeffMatrix_smallalpha}. Let $Z\from \Omega \to \R$ be standard normally distributed, and consider $M_z(x)$ as in \eqref{eq:logisticCoeff}. 
    Then for all $\ell \in \N$ there exist $ C_{\bar{T},\ell}\geq 0$ and $\tau_0>0$ such that for $\max_{i=1,\ldots,N_k} \tau_k^i \leq \tau_0$, 
    \begin{align*}
        \|\bfu_{n,m,k}(t)-\bfu(t)\|_\bfH
        &\le C_{\bar{T},\ell}  \bigg(n^{-\ell} + m^{-2} + \max_{i=1,\ldots,N_k} \tau_k^i\bigg)  \|\bfu_0\|_{H^{2\ell}(\R,\PZ;D^{1+s})}
    \end{align*}
    for all $n,m,k\in\N$, $t \in \calT_k$, $0<s<1/4$, and $\bfu_0\in H^{2\ell}(\R,\PZ;D^{1+s})$.
\end{corollary}

\appendix
\section{Sturm--Liouville operators related to orthogonal polynomials}
\label{sec:appendix}

This appendix provides operator norm estimates of the Sturm--Liouville operators associated to the distributions of Example \ref{ex:distributions} with explicit constants. They are required for the polynomial chaos error estimate of Theorem \ref{thm:PCEbound}. The proofs follow by simple calculations using integration by parts.

\begin{example}
\label{ex:Qs_1D}
    The one-dimensional Sturm--Liouville operators associated to the distributions and their corresponding orthogonal polynomials stated in Example \ref{ex:distributions} are the following.
    \begin{enumerate}[label=(\alph*)]
        \item\label{ex:Qs_1D:Hermite} For $k\in\N_0$, the Hermite polynomial $H_k$ is an eigenfunction of $      Q = - \frac{\mathrm{d}^2}{\mathrm{d}z^2}+z \frac{\mathrm{d}}{\mathrm{d}z}
        $
        to the eigenvalue $\lambda_k=k$.
        \item\label{ex:Qs_1D:Jacobi} For $k\in\N_0$, the Jacobi polynomial $J_k = J_k^{(\alpha,\beta)}$ is an eigenfunction of 
         $
            Q = Q^{(\alpha,\beta)} = -(1-z^2) \frac{\mathrm{d}^2}{\mathrm{d}z^2}+(\alpha-\beta+(\alpha+\beta+2)z) \frac{\mathrm{d}}{\mathrm{d}z}
         $
        to the eigenvalue $\lambda_k =k(k+\alpha+\beta+1)$.
        \item\label{ex:Qs_1D:Laguerre} For $k\in\N_0$, the Laguerre polynomial $L_k = L_k^{(\alpha)}$ is an eigenfunction of
        $
            Q  = Q^{(\alpha)} = -z \frac{\mathrm{d}^2}{\mathrm{d}z^2}+(z-\alpha-1) \frac{\mathrm{d}}{\mathrm{d}z}
        $
        to the eigenvalue $\lambda_k=k$.
    \end{enumerate}
\end{example}
Let $I\subseteq \R$ be an open interval and $w\from I\to (0,\infty)$ be measurable such that $w\in L_1(I)$. We write $L_2(I,w)$ for the weighted $L_2$-space with weight $w$, equipped with the norm
$\norm{\cdot}_{L_2(I,w)}$ given by
\[\norm{f}_{L_2(I,w)}^2 \ce \int_I \abs{f(z)}^2 w(z)\,\rmd z.\]
Let $\rho\from I\to [0,\infty)$ be measurable. For $\ell\in \N_0$ we define $H^\ell_\rho(I,w)$ to be the space of all $f\in L_2(I,w)$ such that $f$ is weakly differentiable up to order $\ell$ and $\rho^{k/2}\partial^k f\in L_2(I,w)$ for all $k\in\{0,\ldots,\ell\}$. We equip $H_\rho^\ell(I,w)$ with the norm $\norm{\cdot}_{H_\rho^\ell(I,w)}$ given by
\[
    \norm{f}_{H_\rho^\ell(I,w)}^2 \ce \sum_{k=0}^\ell \big\|\rho^{k/2} \partial^k f\big\|_{L_2(I,w)}^2.
\]
Note that $H^\ell_\rho(I,w)$ is a Hilbert space and for $\rho= 1$ we write $H^\ell(I,w)\ce H^\ell_\rho(I,w)$ for short. 
We start by considering the density of a standard normal distribution as the weight $w$ and estimate the norm of the associated Sturm--Liouville operator. 

\begin{lemma}
\label{lem:Qbounded_Hermite}
    Let $I=\R$ and $w(z) \ce \frac{1}{\sqrt{2\pi}}\e^{-z^2/2}$ for $z\in I$. 
    Let $Q \ce -\partial^2 + z\partial$ and $\ell\in \N_0$. Then for $f\in H^{\ell+2}(I,w)$, we have $Qf\in H^{\ell}(I,w)$ and 
    \[\norm{Qf}_{H^{\ell}(I,w)} \leq \sqrt{21+3\ell^2} \norm{f}_{H^{\ell+2}(I,w)}.\]
\end{lemma}

\begin{lemma}
\label{lem:Qbounded_Jacobi}
    Let $\alpha,\beta>-1$, $I=(-1,1)$, and 
    \[
        w(z) \ce \frac{\Gamma(\alpha+\beta+2)}{2^{\alpha+\beta+1} \Gamma(\alpha+1)\Gamma(\beta+1)}(1-z)^{\alpha}(1+z)^{\beta} \quad (z\in I).
    \]
    Let $\ell \in \N_0$, $Q \ce -(1-z^2)\partial^2 + (\alpha-\beta+(\alpha+\beta+2)z)\partial$, $\ell\in \N_0$, and $C_{\alpha,\beta}\ce \alpha+\beta+1$. Then for $f\in H^{\ell+2}(I,w)$, we have $Qf\in H^{\ell}(I,w)$ and
    \begin{align*}
        \norm{Qf}_{H^{\ell}(I,w)} &\leq \sqrt{3\big(1+4(\ell+1+\max\{
        \alpha, \beta\})^2+\ell^2(\ell+C_{\alpha,\beta})^2\big)} \norm{f}_{H^{\ell+2}(I,w)}.
    \end{align*}
\end{lemma}
Note that if $Z$ is Beta-distributed, then $w$ is the density of $\P_Z$.

\begin{lemma}
\label{lem:Qbounded_Laguerre}
    Let $\alpha>-1$, $I=(0,\infty)$, $w(z) \ce \frac{1}{\Gamma(\alpha+1)}  z^\alpha\e^{-z}$ and $\rho(z)\ce z$ for $z\in I$. Let $Q \ce -z\partial^2 + (z-\alpha-1)\partial$, and $\ell\in \N_0$. Then for $f\in H_\rho^{\ell+2}(I,w)$, we have $Qf\in H_\rho^{\ell}(I,w)$ and 
    \[\norm{Qf}_{H_\rho^{\ell}(I,w)} \leq \sqrt{24\alpha + 87 + 24\ell + 3\ell^2} \norm{f}_{H_\rho^{\ell+2}(I,w)}.\]
\end{lemma}
Note that if $Z$ is Gamma-distributed, then $w$ is the density of $\P_Z$.

We now move on to the multidimensional case. Let $N \in \N$ and $Z=(Z_0,\ldots,Z_{N-1})$ satisfy Assumption \ref{ass:Zdistribution}.
As in Subsection \ref{subsec:random}, let $Q \ce \sum_{j=0}^{N-1} \widetilde{Q}_j$ with $\widetilde{Q}_j\ce I\otimes \ldots \otimes I \otimes Q_j \otimes I \otimes \ldots \otimes I$, where $Q_j$ is the Sturm--Liouville differential operator associated with the orthogonal polynomials corresponding to $\P_{Z_j}$ for $j\in\{0,\ldots,N-1\}$ (see Example \ref{ex:Qs_1D}; also recall the definition of $H^\ell_\rho(\R^N,\P_Z)$ for $\ell\in\N_0$.

\begin{proposition}
\label{prop:sobnormQfvsf}
    Let $f\in H^{2\ell+2}_\rho(\R^N,\P_Z)$ for some $\ell \in \N_0$. Then $Qf\in H^{2\ell}_\rho(\R^N,\P_Z)$ and
    \begin{equation*}
        \norm{Qf}_{H^{2\ell}_\rho(\R^N,\P_Z)} \leq \sqrt{N}\max_{j\in\{0,\ldots,N-1\}} C_j(2\ell) \norm{f}_{H^{2\ell+2}_\rho(\R^N,\P_Z)},
    \end{equation*}
    where
    \[C_j(\ell)\ce \begin{cases}
    \sqrt{21+3\ell^2}, & j\in J_{\mathrm{Normal}},\\
    \sqrt{24\alpha + 87 + 24\ell + 3\ell^2}, & j\in J_{\mathrm{Gamma}},\\
    \sqrt{3\bigl(1+4(\ell+1+\max\{
    \alpha, \beta\})^2+\ell^2(\ell+\alpha+\beta+1)^2\bigr)}, & j\in J_{\mathrm{Beta}}.
    \end{cases}
    \]
\end{proposition}

\end{document}